\documentclass[11pt]{article}

\usepackage{epsfig,amsmath}
\usepackage{graphicx}
\usepackage{amssymb}
\usepackage{enumerate}
\usepackage{bbm}
\usepackage{array}
\usepackage{lscape}
\usepackage{a4wide}
\usepackage{paralist}
\usepackage{booktabs}

\usepackage[usenames, dvipsnames]{color}

\usepackage{hyperref}

\numberwithin{equation}{section}
\newtheorem{lemma}{Lemma}[section]
\newtheorem{proposition}[lemma]{Proposition}
\newtheorem{theorem}[lemma]{Theorem}
\newtheorem{corollary}[lemma]{Corollary}
\newtheorem{definition}[lemma]{Definition}
\newtheorem{remark}[lemma]{Remark}

\DeclareMathOperator{\CR}{CR}

\DeclareMathOperator{\SLE}{SLE}
\DeclareMathOperator{\CLE}{CLE}

\newenvironment{proof}{{\em Proof.}}{\hspace*{\fill} $\square$  }
\newenvironment{proofof}[1]{{\em Proof of #1.}}{\hspace*{\fill} $\square$ }

\newcommand{\R}{\mathbb{R}}
\newcommand{\D}{\mathbb{D}}
\newcommand{\C}{\mathbb{C}}

\newcommand{\N}{\mathbb{N}}
\newcommand{\Z}{\mathbb{Z}}
\newcommand{\I}{\mathbbm{1}}
\newcommand{\HH}{\mathbb{H}}
\newcommand{\e}{\operatorname{e}}
\newcommand{\im}{\operatorname{i}}

\newcommand{\eps}{\varepsilon}
\newcommand{\epn}{{\eps_n}}

\newcommand{\ka}{\kappa}
\newcommand{\kp}{{\kappa'}}

\newcommand{\cL}{\mathcal{L}}

\newcommand{\lcS}{\mathcal{D}}
\newcommand{\lcB}{\mathbf{D}}
\newcommand{\lcTB}{\mathrm{D}}
\newcommand{\Loop}{\mathcal{L}}
\newcommand{\bub}{\cB}
\newcommand{\loops}{\frk{cle}}
\newcommand{\LQG}{\frk{lqg}}
\newcommand{\BM}{\frk{be}}

\newcommand{\aryb}{\begin{eqnarray*}}
	\newcommand{\arye}{\end{eqnarray*}}
\def\alb#1\ale{\begin{align*}#1\end{align*}}
\newcommand{\eqb}{\begin{equation}}
	\newcommand{\eqe}{\end{equation}}
\newcommand{\eqbn}{\begin{equation*}}
	\newcommand{\eqen}{\end{equation*}}

\newcommand{\BB}{\mathbbm}
\newcommand{\op}{\operatorname}

\newcommand{\frk}{\mathfrak}
\newcommand{\eqD}{\overset{d}{=}}

\newcommand{\ep}{\epsilon}
\newcommand{\rta}{\rightarrow}

\newcommand{\wh}{\widehat} 
\newcommand{\wt}{\widetilde}

\def\cS{\mathcal{S}}

\def\cQ{\mathcal{Q}}

\def\cO{\mathcal{O}}

\def\cL{\mathcal{L}}

\def\cI{\mathcal{I}}

\def\cD{\mathcal{D}}
\def\cC{\mathcal{C}}
\def\cB{\mathcal{B}}
\def\cA{\mathcal{A}}

\newcommand{\cart}{Carath\'{e}odory }
\newcommand{\fl}{\mathfrak{l}}

\newcommand{\nina}[1]{{#1}}

\newcommand{\corr}[1]{{#1}} 

\begin{document}

	\title{{Brownian half-plane excursion and critical Liouville quantum gravity}}

\date{}

\author{
	\begin{tabular}{c}Juhan Aru\\[-5pt]\small EPFL \end{tabular}\; 
	\begin{tabular}{c}Nina Holden\\[-5pt]\small ETH Z\"urich\end{tabular}\; 
	\begin{tabular}{c}Ellen Powell\\[-5pt]\small  Durham University \end{tabular}
	\begin{tabular}{c}Xin Sun\\[-5pt]\small  University of Pennsylvania \end{tabular}
}
\setcounter{tocdepth}{2}
\maketitle

\begin{abstract}
	In a groundbreaking work, Duplantier, Miller and Sheffield showed that subcritical Liouville quantum gravity (LQG) coupled with Schramm-Loewner evolutions (SLE) can be obtained by gluing together a pair of Brownian motions. In this paper, we study the counterpart of their result in the critical case via a limiting argument. In particular, we prove that as one sends $\kappa' \downarrow 4$ in the subcritical setting, the space-filling SLE$_{\kappa'}$ in a disk degenerates to the CLE$_4$ exploration introduced by Werner and Wu, along with a collection of i.i.d.\ coin tosses indexed by the branch points of the exploration. Furthermore, in the same limit, we observe that although the pair of initial Brownian motions collapses to a single one, one can  still extract two different independent Brownian motions $(A,B)$ from this pair,  such that the Brownian motion $A$ encodes the LQG distance from the CLE loops to the boundary of the disk and the Brownian motion $B$ encodes the boundary lengths of the CLE$_4$ loops. In contrast to the subcritical setting, the pair $(A,B)$ does not determine the CLE-decorated LQG surface. Our paper also contains a discussion of relationships to random planar maps, the conformally invariant CLE$_4$ metric, and growth fragmentations. 
	
\end{abstract}

\tableofcontents

\section{Introduction}

The most classical object of random planar geometry is probably the two-dimensional
	Brownian motion together with its variants. Over the past 20 years,
	a plenitude of other interesting random geometric objects have been discovered and studied.
	Among those we find Liouville quantum gravity (LQG) surfaces
	\cite{DS11} and conformal loop ensembles (CLE) \cite{SW12,Sh09}. LQG surfaces
	aim to describe the fields appearing in the study of 2D Liouville
	quantum gravity and can be viewed as canonical models for random surfaces. They can be mathematically defined in terms of
	volume forms \cite{DS11,RV14,Kah85} (used in this paper), but recently also in
	terms of random metrics \cite{GM19,DDDF20}. CLE is a random collection
	of loops that correspond conjecturally to interfaces of the $q$-state Potts
	model and the FK random cluster model in the continuum limit 
	(see e.g.\ \cite{CLEPERC}).

In this paper we study a coupling of LQG measures, CLE and Brownian
motions, taking a form of the kind first discovered in \cite{DMS14}. On the one hand we consider a ``uniform''
exploration of a conformal loop ensemble, $\CLE_4$, drawn on top of an independent LQG surface known as the critical LQG disk. On the other hand, we take a seemingly simpler object: the Brownian half plane excursion. In this coupling one component of the Brownian excursion encodes the branching structure of the CLE$_4$ exploration, together with a certain (LQG surface dependent) distance of CLE$_4$ loops from the boundary. The other component of the Brownian excursion encodes the LQG boundary lengths of the discovered CLE$_4$ loops. 

{Our result can be viewed as the critical ($\kp=4$) analog of Duplantier-Miller-Sheffield’s mating of trees theorem for $\kp  > 4$, \cite{DMS14}. The original mating of trees theorem first observes that the quantum boundary length process defined by a space-filling SLE$_\kp$ curve drawn on a subcritical LQG surface is given by a certain correlated planar Brownian motion. Moreover, it says that one can take the two components of this planar Brownian motion, glue each one to itself (under its graph) to obtain two continuum random trees, and then mate these trees along their branches to obtain both the LQG surface and the space-filling SLE curve wiggling between the trees in a measurable way. This theorem has had far-reaching consequences and applications, for example to the study of random planar maps and their limits \cite{HS19,GM-RW,GM-conv}, SLE and CLE \cite{GHM-KPZ, MSWfrag,ahs-int,AS-CLE}, and LQG itself \cite{MSTBMI,ARS}. See the survey \cite{GHS19} for further applications.

Obtaining a critical analog of the mating of trees theorem was one of the main aims of this paper. The problem one faces is that the above-described picture degenerates in many ways as $\kp\downarrow 4$ (e.g.\ the correlation of the Brownian motions tends to one and the Liouville quantum gravity measure converges to the zero measure). However, it is known that the LQG measure can be renormalized in a way that gives meaningful limits \cite{APS18two}, and the starting point of the current project was the observation that the pair of Brownian motions can be renormalized via an affine transformation to give something meaningful as well.}

Still, not all the information passes nicely to the limit, and in particular extra randomness appears. Therefore, our limiting coupling is somewhat different in nature to that of \cite{DMS14} (or \cite{AG19} for the finite volume case of quantum disks). Most notably, one of the key results of \cite{DMS14,AG19} is that the CLE decorated LQG determines the Brownian motions, and vice versa. In our case neither statement holds in the same way; see Section \ref{sec:meas} for more details. For example, to define the Brownian excursion from the branching CLE$_4$ exploration, one needs 
a binary variable at every branching event to decide on an ordering of the branches.

{We believe that in addition to completing the critical version of Duplantier-Miller-Sheffield's  mating of trees theorem, the results of this paper are intriguing in their own right. Moreover, as explained below, this article opens the road for several interesting questions in the realm of SLE theory, about LQG-related random metrics, in the setting of random planar maps decorated with statistical physics models, and about links to growth-fragmentation processes.}

\subsection{Contributions}

Since quite some set up is required to describe our results for $\kappa=4$ precisely, we postpone the detailed statement to Theorem \ref{thm_main}. Let us state here a caricature version of the final statement. Some of the objects appearing in the statement will also be precisely defined only later, yet should be relatively clear from their names.

\begin{theorem}
	Let:
	\begin{itemize}
			\item  $\LQG$ be the field of a critical quantum disk together with associated critical LQG measures \corr{(see Section \ref{sec:LQG})}; 
		\item $\loops$ denote the uniform space-filling $\SLE_4$ in the unit disk \corr{parametrized by critical LQG mass}, which is defined in terms of a uniform $\CLE_4$ exploration plus a collection of independent coin tosses \corr{(see Section \ref{sec:ucle4})};
		\item and  $\BM$ describe a Brownian (right) half plane excursion $(A, B)$ 
		\corr{(see Section \ref{sec:Bfs})}.
	\end{itemize}
	Then one can couple $(\loops, \LQG, \BM)$ such that $\loops$ and $\LQG$ are independent, $A$ encodes a certain quantum distance for $\CLE_4$ loops from the boundary, and $B$ encodes the quantum boundary lengths of the $\CLE$ loops. Moreover $(\loops, \LQG)$ determines $\BM$, but the opposite does not hold.
	\label{thm:caricature}	
\end{theorem}

\begin{figure}[h]
	\centering
	\includegraphics[width=\textwidth]{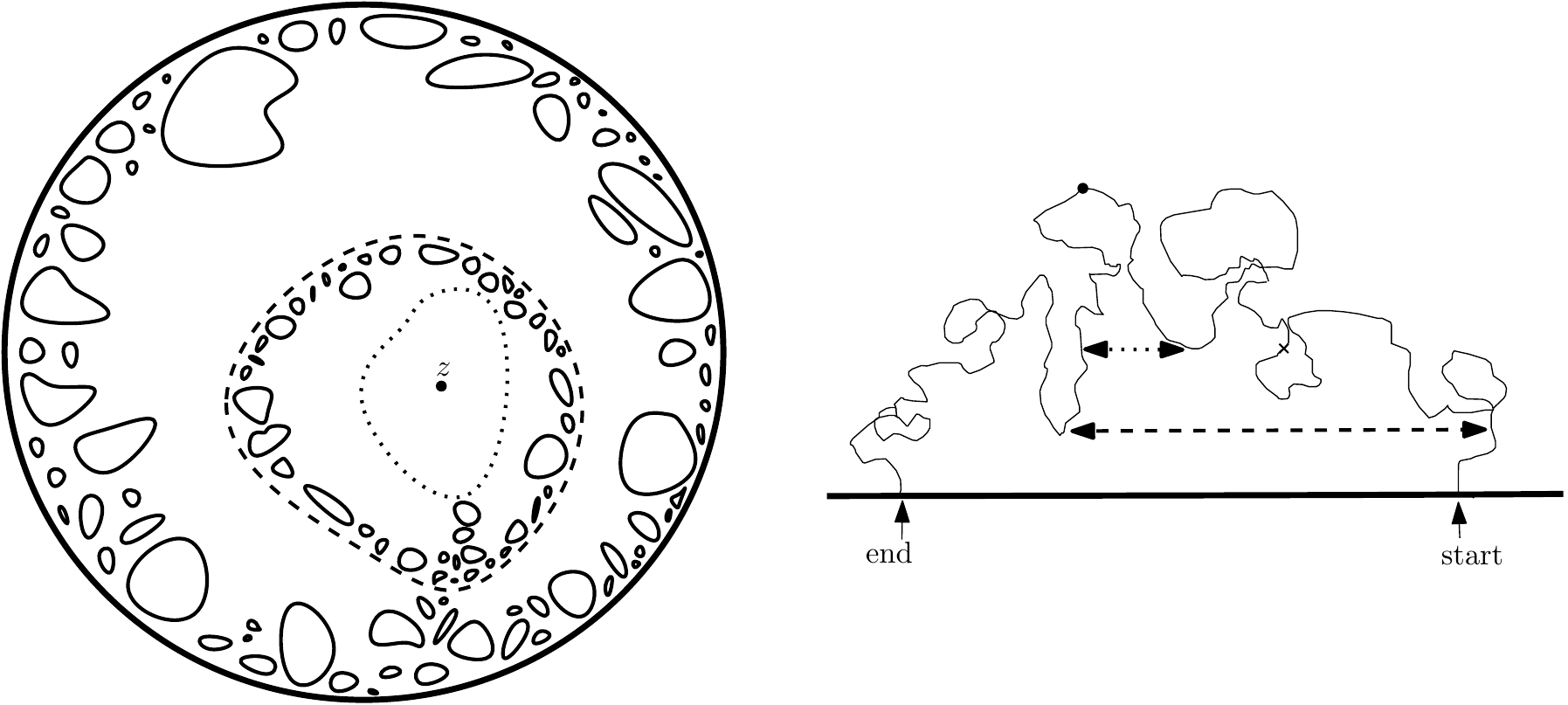}
	\caption{A simplistic sketch of the correspondence in Theorem \ref{thm:caricature}. \textbf{On the left:} all the outermost $\CLE_4$ loops discovered by the space filling SLE$_4$ before the dashed loop surrounding $z$ is discovered, together with all of the second-level nested CLE$_4$ loops discovered before the dotted loop surrounding $z$ is discovered. \textbf{On the right:} the corresponding half-planar Brownian excursion, with the coordinate axes switched for ease of viewing. The sub-excursion marked by the dashed (resp. dotted) line -  i.e., the portion of Brownian path starting and ending at the endpoints of this line - corresponds to the exploration within the dashed (resp. dotted) loop. The lengths of these lines are the Liouville quantum gravity lengths of the corresponding loops, and the duration of the sub-excursions are their Liouville quantum gravity areas. The time that $z$ is visited is marked by a dot, and the time that the dotted loop is discovered is marked by a cross. When the dotted loop is discovered, a coin is tossed to determine which of the two disconnected yet-to-be explored domains is visited first by the space filling SLE$_4$; in this example, the component containing $z$ is visited second. See also Figure \ref{fig:correspondence_2} below.} \label{fig:correspondence}
\end{figure}
\begin{figure}[h]
	\centering
	\includegraphics[width=\textwidth]{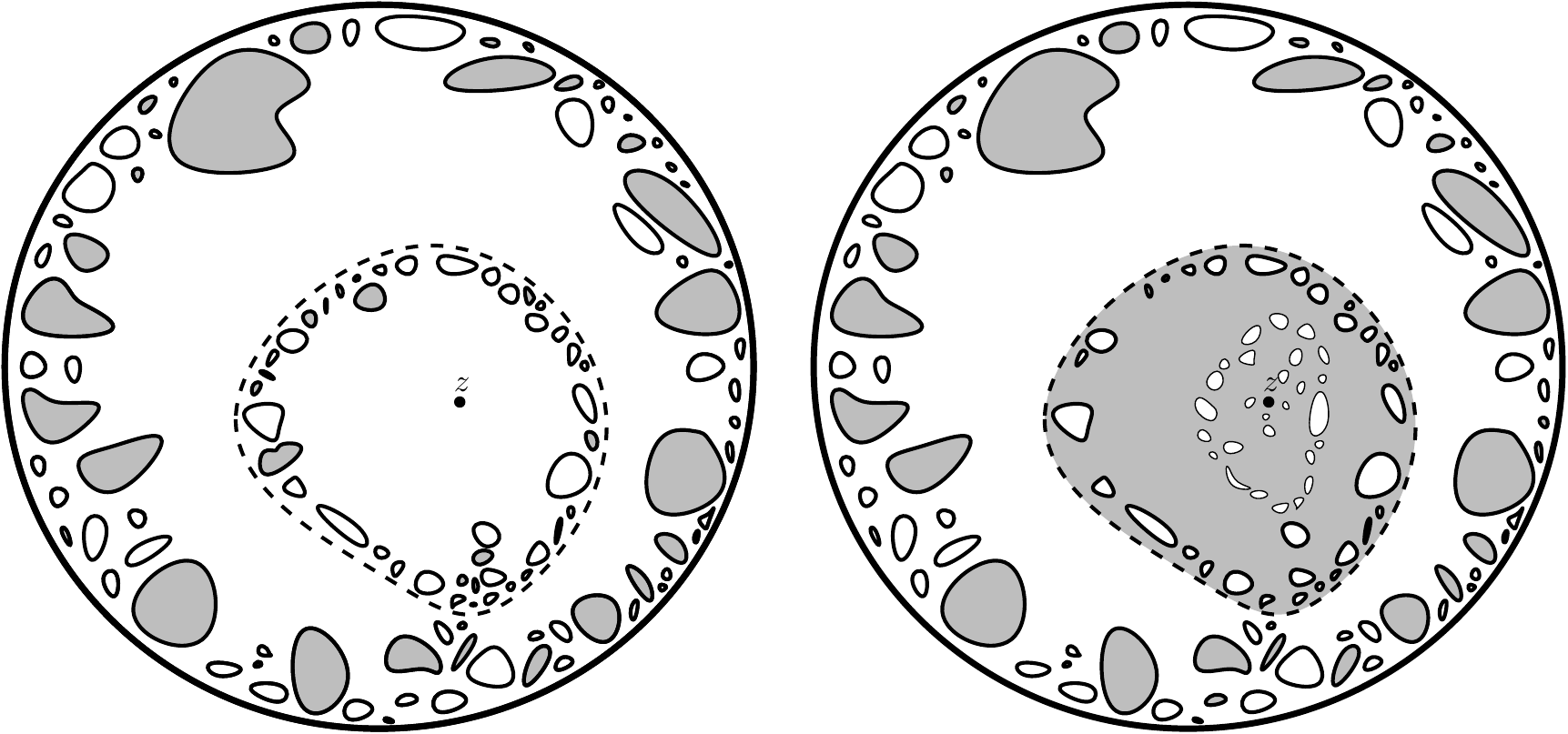}
	\caption{An illustration of the subset of the unit disk, shaded gray, that has been explored by the space-filling SLE$_4$ at two different times. \textbf{On the left:} at the time that the second level CLE$_4$ loop surrounding $z$ is discovered (marked by a cross on the right of Figure \ref{fig:correspondence}). \textbf{On the right:} at the time that $z$ is reached (marked by a dot on the right of Figure \ref{fig:correspondence}). Note that, although this is not apparent from the sketch, the explored subset of the unit disk at any given time is actually a connected set.}
		\label{fig:correspondence_2}
\end{figure}

In terms of limit results, we for example prove the following:
\begin{itemize}
	\item We show that a $\SLE_\kp(\kp-6)$ in the disk converges to the uniform CLE$_4$ exploration introduced by Werner and Wu, \cite{WW13}, as $\kp\downarrow 4$ (Proposition \ref{prop:convfullbranch}). Here an extra level of randomness appears in the limit, in the sense that new CLE$_4$ loops in the exploration are always added at a uniformly chosen point on the boundary, in contrast to the $\kp>4$ case where the loops are traced by a continuous curve.
	\item Using a limiting argument, we also show in Section \ref{sec:conv_order} how to make sense of a ``uniform'' space-filling SLE$_4$ exploration, albeit no longer defined by a continuous curve. Again extra randomness appears in the limit: contrary to the  $\kp > 4$ case, the nested uniform CLE$_4$ exploration does not uniquely determine this space-filling SLE$_4$.
	\item Perhaps less surprisingly but nonetheless not without obstacles, we show that the nested CLE$_\kp$ in the unit disk converges to the nested CLE$_4$ with respect to Hausdorff distance (Proposition \ref{prop:cleloopconv}). We also show that after dividing the associated quantum gravity measures  
	by $(4-2\gamma)$, a $\gamma$-Liouville quantum gravity disk converges to a critical Liouville quantum gravity disk.
\end{itemize}

In terms of connections and open directions, let us very briefly mention a few examples and refer to Section \ref{secGF} for more detail. 

\begin{itemize}
	\item First, as stated above in Theorem \ref{thm:caricature}, $(\loops, \LQG)$ determines $\BM$, but the opposite does not hold. A natural question is whether there is another natural mating-of-trees type theorem for $\kappa=4$ where one has measurability \corr{in} both directions.
	\item Second, our coupling sheds light on recent work of A\"id\'ekon and Da Silva \cite{ADS}, who identify a (signed) growth fragmentation embedded naturally in  the Brownian half plane excursion. The cells in this growth fragmentation correspond to very natural observables in our exploration.
	\item Third, as we have already mentioned, one of the coordinates in our Brownian excursion encodes a certain LQG distance of CLE$_4$ loops from the boundary. It is reasonable to conjecture that this distance should be related to the CLE$_4$ distance defined in \cite{WW13} via a Lamperti transform.\footnote{We thank N.\ Curien for explaining this relation to us.}
	\item Fourth, several interesting questions can be asked in terms of convergence of discrete models. Critical FK-decorated planar maps and stable maps are two immediate candidates.
\end{itemize}

\subsection{Outline}

The rest of the article is structured as follows. In Section 2, after reviewing background material on branching SLE and CLE, we will prove the convergence of the  $\SLE_\kp(\kp-6)$ exploration in the disk to the uniform CLE$_4$ exploration, and also show the convergence of the nested CLE with respect to Hausdorff distance. In Section 3, we use the limiting procedure to give sense to a notion of space-filling SLE$_4$. In Section 4, we review the basics of Liouville quantum gravity surfaces and of the mating of trees story, and prove convergence of the Brownian motion functionals appearing in \cite{DMS14,AG19} after appropriate normalization. We also finalize a certain proof of Section 3, which is interestingly (seemingly) much easier to prove in the mating of trees context. Finally, in Section 5 we conclude the proof of joint convergence of Brownian motions, space-filling SLE and LQG. This allows us to state and conclude the proof of our main theorem. We finish the paper with a small discussion on connections, and an outlook on several interesting open questions.

Throughout, $\gamma\in(\sqrt{2},2]$ is related to parameters $\kappa,\kappa',\eps$ by 
\begin{equation}
	\label{eq:parameters}
	\kappa=\gamma^2,\quad \kappa'=16/\kappa,\quad\eps=2-\gamma.
\end{equation}

\subsection{Acknowledgements}
J.A.\ was supported by Eccellenza grant 194648 of the Swiss National Science Foundation. N.H.\ was supported by grant 175505 of the Swiss National Science Foundation, along with Dr.\ Max Rössler, the Walter Haefner Foundation, and the ETH Zürich Foundation. E.P.\ was supported by grant 175505 of the Swiss National Science Foundation. X.S.\ was supported by the NSF grant DMS-2027986 and the NSF Career grant DMS-2046514. J.A.\ and N.H.\ were both part of SwissMAP. We all thank Wendelin Werner and ETH for their hospitality.
We also thank Elie A\"id\'ekon, Nicolas Curien, William Da Silva, Ewain Gwynne, Laurent M\'enard, Avelio Sep\'ulveda, and Samuel Watson for useful discussions. \corr{Finally, we thank the anonymous referee for their careful reading of this article, and helping to improve the exposition in numerous places}.

\section{Convergence of branching SLE$_\kp$ and CLE$_\kp$ as $\kp\downarrow 4$} \label{sec:conv_clesle}

\subsection{Background on branching SLE and conformal loop ensembles}
\label{sec:bg}

 \subsubsection{Spaces of domains}
 Let $\lcS$ be the space of $\lcTB=\{\lcTB_t\, ; \, t\ge 0\}$ such that:
 \begin{itemize}
 	\item  for every $t\ge 0$, $0\in {\lcTB}_t\subset \lcS$ and ${\lcTB}_t$ is simply connected planar domain;
 	\item $\lcTB_t\subset \lcTB_s$ for all $0\le s < t<\infty;$ 
 	\item  for every $t\ge 0$, if $f_t=f_t[{\lcTB}]$ is the unique conformal map from $\D$ to ${\lcTB}_t$ that sends $0$ to $0$ and has $f_t'(0)>0$, then $f_t'(0)=\CR(0;\lcTB_t)=e^{-t}$.
 \end{itemize}
We also write $g_t=g_t[{\lcTB}]$ for the inverse of $f_t$.

Recall that a sequence of simply connected domains $(U^n)_{n\ge 0}$ containing $0$ are said to converge to a simply connected domain $U$ in the Carath\'{e}odory topology (viewed from $0$) if we have $f_{U^n}\to f_U$ uniformly in $r\D$ for any $r<1$, where $f_{U^{n}}$ (respectively $f_U$) are the unique conformal maps from $\D$ to $U^{n}$ (respectively $U$) sending $0$ to $0$ and with positive real derivative at $0$.  \cart convergence viewed from $z\ne 0$ is defined in the analogous way.

We equip $\lcS$ with the natural extension of this topology: that is, we say that a sequence $({\lcTB}^{n})_{n\ge 0}$ in $\lcS$ converges to ${\lcTB}$ in $\lcS$ if for any $r<1$ and $T\in [0,\infty)$
\begin{equation}\label{eq:cartconvdef} \sup_{t\in [0,T]}\sup_{z\in r\D}|f^{n}_t(z)-f_t(z)|\to 0\end{equation}
as $n\to \infty$, where $f_t^{n}=f_t[\lcTB^{n}]$ and $f_t=f_t[\lcTB]$.  With this topology, $\lcS$ is a metrizable and separable space: see for example \cite[Section 6.1]{QLE}.

\subsubsection{Radial Loewner chains}
In order to introduce radial SLE, we first need to recall the definition of a (measure-driven) radial Loewner chain. Such chains are closely related to the space $\lcS$, as we will soon see. 
 If $\lambda$ is a 
measure on $[0,\infty) \times \partial \D$ whose marginal on $[0,\infty)$ is Lebesgue measure, we define the radial Loewner equation driven by $\lambda$ via
\begin{equation}\label{eq:loew_int}
	g_t(z)=\int_{[0,t]\times \partial \D} g_s(z)\frac{u+g_s(z)}{u-g_s(z)} \, d\lambda(s,u) ; \quad \quad g_0(z)=z
\end{equation}
for $z\in \D$ and $t\ge 0$. It is known (see for example \cite[Proposition 6.1]{QLE}) that for any such $\lambda$, \eqref{eq:loew_int} has a unique solution $g_t(z)$ for each $z\in \D$, defined until time $t_z:=\sup\{t\ge 0: g_t(z)\in \D\}$. Moreover, if one defines $\lcTB_t:=\{z\in \D: t_z<t\}$, then $\lcTB=\{\lcTB_t\, , \, t\ge 0\}$ is an element of $\lcS$,  and $g_t$ from \eqref{eq:loew_int} is equal to $g_t[\lcTB]=(f_t[\lcTB])^{-1}$ for each $t$. We call $\lcTB$ the radial Loewner chain driven by $\lambda$. 

Note that if one restricts to measure of the form $\lambda(A,dt)=\delta_{W(t)}(A) \, dt$ with $W:[0,\infty)\to \partial \D$ piecewise continuous, this defines the more classical notion of a radial Loewner chain. In this case we can rewrite the radial Loewner equation as \begin{equation}
	\label{eqn:rad_loewner}
	\partial_t g_t(z)= g_t(z) \corr{\frac{W_t+g_t(z)}{W_t-g_t(z)}}; \;\; z\in \D,\, t\le t_z:=\inf\{s: g_s(z)=W_s\}
\end{equation} 
and we refer to the corresponding Loewner chain as the radial Loewner evolution with driving function $W$. In fact, this is the case that we will be interested in when defining radial $\SLE_\kp(\kp-6)$ for $\kp>4$.

\begin{remark}\label{rmk:dconv_cconv}Let us further remark that if $(\lambda^n)$ are a sequence of driving measures as above, such that $\lambda^n$ converges weakly \corr{(i.e.\, with respect to the weak topology on measures)} to some $\lambda$ on $[0,T]\times \partial \D$ for every $T$, then the corresponding Loewner chains $(\lcTB^n),\lcTB$ are such that $\lcTB^n\to \lcTB$ in $\lcS$ (\cite[Proposition 6.1]{QLE}). \corr{In particular, one can check that if $\lambda^n(A,dt)=\delta_{W^n(t)}(A) \, dt$ and $\lambda(A,dt)=\delta_{W(t)}(A)\, dt$ for some piecewise continuous functions $W^n:[0,\infty)\to \partial \D$, and $W:[0,\infty)\to \partial \D$ %with respect to the Skorokhod topology, 
	then the corresponding Loewner chains converge in $\lcS$
	% Indeed, if 
	if for any $T>0$ fixed and $F:[0,T]\times \partial \D\to \R$ bounded and continuous, we have  %$\tilde{F}^n$ defined by $\tilde{F}^n(t)= F(W^n(t),t)$ converges to $\tilde{F}$ defined by $\tilde{F}(t)=F(W(t),t)$ in the Skorokhod topology as $n\to \infty$. This implies that 
	\begin{equation}\label{eq:lambdanlambda}\lambda^n(F)=\int_0^T\int_{\partial \D} F(u,t) \delta_{W^n(t)}(u)dt = \int_0^T F(W^n(t),t) \, dt \to \lambda(F)=\int_0^T F(W(t),t) \, dt\end{equation} as $n\to \infty$.}% and so $\lambda^n$ converges weakly to $\lambda$ on $[0,T]\times \partial \D$. 
\end{remark}

\begin{remark}\label{rmk:stopped_loewner}
In what follows we will sometimes need to consider evolving domains $\{\lcTB_t\, ; \, t\in [0,S]\}$ that satisfy the conditions to be an element of $\lcS$ up to some finite time $S$. In this case we may extend the definition of $\lcTB_t$ for $t\ge S$ by setting $\lcTB_t=f_S(e^{-(t-S)}\D)$, where $f_S:\D\to \lcTB_S$ is the unique conformal map sending $0\to 0$ and with $f_S'(0)=e^{-S}$.With this extension, $\lcTB=\{\lcTB_t \, ; \, t\ge 0\}$ defines an element of $\lcS$. 
	
	If we have a sequence of such objects, then we say that they converge to a limiting object in $\lcS$ if and only if these extensions converge. We will use this terminology without further comment in the rest of the article.
\end{remark}
\subsubsection{Radial $\SLE_\kp(\kp-6)$} \label{sec:slek}

Let $\kp\in (4,8)$, and recall the relationship \eqref{eq:parameters} between $\kp\in (4,8)$ and $\eps\in (0,2-\sqrt{2})$. Although the use of $\eps$ is somewhat redundant at this point, we do so to avoid redefining certain notations later on. 

Let $B$ be a standard Brownian motion, and let $\theta^\eps_0=\{(\theta_0^\eps)_t\, ; \, t>0\}$ be the unique $B$-measurable process taking values in $[0,2\pi]$, with $(\theta^\eps_0)_0=x\in [0,2\pi]$, that is instantaneously reflecting at $\{0,2\pi\}$, and that solves the SDE
\begin{equation}\label{eq:sde_theta} d(\theta^\eps_0)_t = \sqrt{\kp}dB_t+ \frac{\kp-4}{2}\cot\left(\frac{(\theta^\eps_0)_t}{2}\right) \, dt\end{equation}
on time intervals for which $(\theta^\eps_0)_t\ne \{0,2\pi\}$. The existence and pathwise uniqueness of this process is shown in \cite[Proposition 3.15 \& Proposition 4.2]{Sh09}. It follows from the strong Markov property of Brownian motion that $\theta^\eps_0$ has the strong Markov property. We let $\tau^\eps_0$ be the first hitting time of $2\pi$ by $\theta_0^\eps$.

Associated to $\theta^\eps_0$, we can define a process $W^\eps_0$, taking values on $\partial \D$, by setting
\begin{equation}\label{def:Wfromtheta} (W_0^\eps)_t = \exp\left(\im\, ((\theta_0^\eps)_t- \int_0^t \cot\left((\theta_0^\eps)_s/2\right) \, ds)\right) \quad t\ge 0.\end{equation} 
This indeed gives rise to a continuous function $W_0^\eps$ in time (see e.g.\ \cite{Sh09,MSW14}) and using this as the driving function in the 
 radial Loewner equation \eqref{eqn:rad_loewner}
 defines a radial $\SLE_\kp(\kp-6)$ in $\D$ from $1$ to $0$, with a force point at $e^{-ix}$ (recall that $(\theta_0^\eps)_0=x$). We denote this by $({\lcB}^\eps_0)=\{({\lcB}^\eps_0)_t\, ; \, t\ge 0\}$ which is an element of $\lcS$. In fact, there almost surely exists a continuous non self-intersecting curve $\eta^\eps_0:[0,\infty)\to \D$ such that $({\lcB}^\eps_0)_t$ is the connected component of $\D\setminus \eta^\eps_0[0,t]$ containing $0$ for all $t$ \cite{RS05,MSIG1}. 

Usually we will start with $x=0$, and then we say that the force point is at $1^-$: everything in the above discussion remains true in this case, see \cite{Sh09}. In this setting we refer to ${\lcB}^\eps_0$ and/or $\eta^\eps_0$ (interchangeably) as simply a radial $\SLE_{\kp}(\kp-6)$ targeted at $0$.

The time $\tau^\eps_0$ corresponds to the first time that $0$ is surrounded by a counterclockwise loop: see Figure \ref{fig:tau}.  To begin, we will just consider the SLE stopped at this time. We write $${\lcTB}^\eps_0=\{({\lcTB}^\eps_0)_t\, ; \, t\ge 0\}:=\{({\lcB}^\eps_0)_{\tau^\eps \wedge t}\, ; \, {t\ge 0}\}$$ for the corresponding element of $\lcS$ (see Remark \ref{rmk:stopped_loewner}).

\begin{figure}[h]
	\centering
	\includegraphics[width=\textwidth]{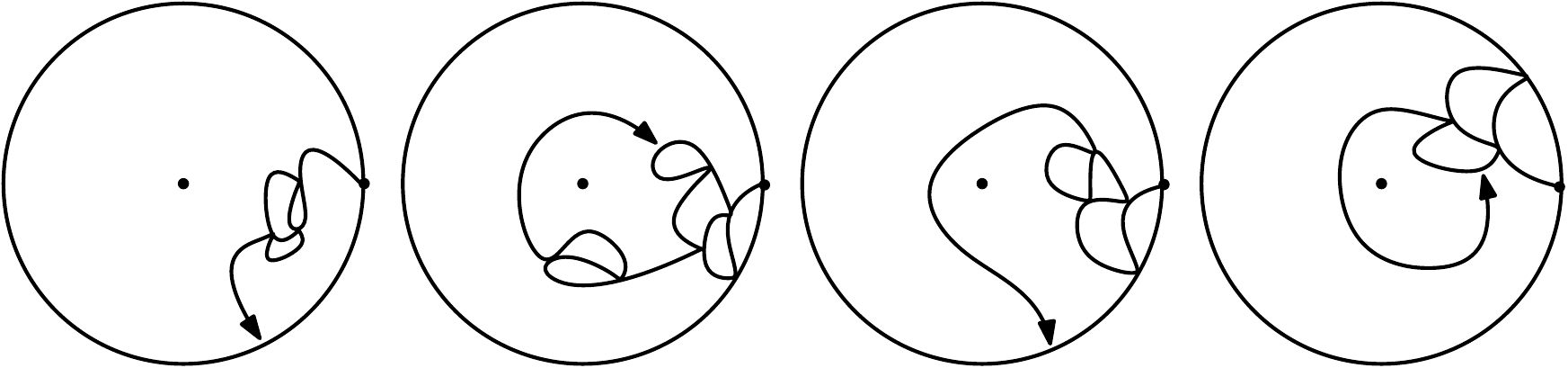}
	\caption{From left to right, the process $\theta^\eps_0$ does the following at the illustrated time: hits 0, hits 0, hits neither 0 nor $2\pi$, hits $2\pi$. The rightmost image is, therefore, an illustration of the time $\tau_0^\eps$.}\label{fig:tau}
\end{figure}

\subsubsection{An approximation to radial SLE$_{\kp}(\kp-6)$}
We will make use of the following approximations $({\lcTB}^{\eps,n}_0)_{n\in \N}$ to ${\lcTB}^\eps_0$ in $\lcS$ (in order to show convergence to {the} CLE$_4$ {exploration}). Fixing $\eps$, and taking the processes $\theta_0^\eps$ and $W_0^\eps$ as above, the idea is to remove intervals of time where $\theta_0^\eps$ is making tiny excursions away from $0$, and then define $\lcTB^{\eps,n}_0$ to be the radial Loewner chain whose driving function is equal to $W_0^\eps$, but with these times cut out. 

More precisely, we set $T_0^{\eps,n}:=0;$ and \corr{inductively define}
\begin{align*} {R}_1^{\eps,n} & =\inf \{t\ge T_0^{\eps,n}\hspace{-0.1cm}:\hspace{-0.1cm} (\theta^\eps_0)_t\ge 2^{-n}\};  \\ S_1^{\eps,n} & =\sup\{t\le R_1^{\eps,n}\hspace{-0.1cm}:\hspace{-0.1cm} (\theta^\eps_0)_t=0\}; \\ T_1^{\eps,n} & =\inf\{t\ge R_1^{\eps,n}\hspace{-0.1cm}:\hspace{-0.1cm} (\theta^\eps_0)_t=0\}; \\
R_2^{\eps,n} & =\inf \{t\ge T_1^{\eps,n}\hspace{-0.1cm}:\hspace{-0.1cm}(\theta^\eps_0)_t\ge 2^{-n}\} ;   \\ S_2^{\eps,n} & =\sup\{t\le R_2^{\eps,n}\hspace{-0.1cm}:\hspace{-0.1cm} (\theta^\eps_0)_t=0\} ; \\	T_2^{\eps,n} & =\inf\{t\ge R_2^{\eps,n}\hspace{-0.1cm}:\hspace{-0.1cm} (\theta^\eps_0)_t=0\}; \end{align*}
etc.\ so the intervals $[S_i^{\eps,n},T_i^{\eps,n}]$ for $i\ge 1$ are precisely the intervals on which $\theta^\eps_0$ is making an excursion away from $0$ whose maximum height exceeds $2^{-n}$. Call the $i$th one of these excursions $e_i^{\eps,n}$.
Also set $ \Lambda^{\eps,n}:= \sup\{j: S_j^{\eps,n}\le \tau^\eps_0\}$ and \begin{equation*}l_i^{\eps,n}:= T_i^{\eps,n}-S_i^{\eps,n} \text{ for } i<\Lambda^{\eps,n} \text{ ; } l_{\Lambda^{\eps,n}}^{\eps,n}=\tau^\eps_0-\corr{S}_{\Lambda^{\eps,n}}^{\eps,n} \text{ ; } L_i^{\eps,n}=\sum\nolimits_{1\le j \le i} l_j^{\eps,n} \text{ for } 1\le i \le \Lambda^{\eps,n}.\end{equation*}

Now we define $$(W_0^{\eps,n})_t=(W_0^\eps)_{S_i^{\eps,n}+(t-L_{i-1}^{\eps,n})}, \text{ for } t\in [L_{i-1}^{\eps,n},L_i^{\eps,n})\text{ and } 1\le i \le \Lambda^{\eps,n},$$ and set $\lcTB_0^{\eps,n}$ to be the radial Loewner chain with driving function $W_0^{\eps,n}$. This is defined up to time $\tau_0^{\eps,n}:=L_{\Lambda^{\eps,n}}^{\eps,n}$.

We will show in Section \ref{sec:conv_bt} that ${\lcTB}^{\eps,n}_0 \to {\lcTB}^\eps_0$ in $\lcS$ as $n\to \infty$ (see Lemma \ref{lem:ngoodapprox}).
\vspace{0.2cm}

\subsubsection{Uniform $\CLE_4$ exploration {targeted at the origin}}\label{sec:ucle4}

Now suppose that we replace $\kp$ with $4$, so that the solution $\theta_0$ of \eqref{eq:sde_theta} is simply a (speed $4$) Brownian motion reflected at $\{0,2\pi\}$. Then the integral in \eqref{def:Wfromtheta} does not converge, but it is finite for any single excursion of $\theta_0$.\footnote{That is, if $\lambda$ is the Brownian excursion measure then the integral is finite for $\lambda$-almost all excursions, see \cite[Section 2]{WW13}).} For any $n\in \N$ if we define $\tau^n_0$,  $\Lambda^{n}$ and $(S_i^{n},T_i^{n},l_i^{n},L_i^{n})_{i\ge 1}$ as in the sections above, we can therefore define a process ${\lcTB}^{n}_0$ in $\lcS$ via the following procedure: 
\begin{itemize}
	\item sample random variables $(X_i^{n})_{i\ge 1}$ uniformly and independently on $\partial \D$;
	\item define $(W_0^n)_t$ for $t\in [0,\tau_0^n)$ by setting 
	\begin{equation}\label{def:excu4} (W_0^n)_t= X_i^{n}\exp\left(\im ((\theta_0)_{t+S_i^{n}} - \int_{S_i^{n}}^{t+S_i^{n}} \cot((\theta_0)_s/2) \,ds)\right) \end{equation} for  $t\in [L_{i-1}^n,L_i)$ and $ 1\le i\le \Lambda^n$;
	\item let $\lcTB^n$ be the radial Loewner chain with driving function $W_0^n$.
\end{itemize}

With these definitions we have that ${\lcTB}^{n}_0\Rightarrow {\lcTB}_0$ in $\lcS$  as $n\to \infty$, where the limit process is the \emph{uniform CLE$_4$ exploration} introduced in \cite{WW13}, and run until the outermost CLE$_4$ loop surrounding $0$ is discovered. 

More precisely, the uniform CLE$_4$ exploration towards $0$ in $\D$ can be defined as follows. One starts with a Poisson point process $\{(\gamma_j, t_j)\, ; \, j\in J\}$ with intensity given by $M$ times Lebesgue measure, where $M$ is the SLE$_4$ bubble measure rooted uniformly over the unit circle: see \cite[Section 2.3.2]{SWWDCD}. In particular, for each $j$, $\gamma_j$ is a simple continuous loop rooted at some point in $\partial \D$. We define $\mathrm{int}(\gamma_j)$ to be the connected component of $\D\setminus \gamma_j$ that intersects $\partial \D$ only at the root, and set $\tau=\inf \{t: t=t_j \text{ with } 0 \in \mathrm{int}(\gamma_j)\}$ so that for all $t_j<\tau$, $\mathrm{int}(\gamma_j)$ does not contain the origin. Therefore, for each such $j$ we can associate a unique conformal map $f_j$ from $\D$ to the connected component of $\D\setminus \gamma_j$ containing $0$ to $\D$, such that $f_j(0)=0$ and $f_j'(0)>0$. For any $t\le \tau$ it is then possible to define  (for example by considering only loops with some minimum size and then letting this size tend to $0$, see again \cite{WW13,SWWDCD}) $f_t$ to be the composition $\circ_{t_j< t} f_{t_j}$, where the composition is done in reverse chronological order of the $t_j$s. The process
\begin{equation}\label{pppcle}\{\lcTB'_t \, ; \, t\le \tau\}:=\{f_t(\D) \, ; \, t\le \tau\}\end{equation}
is then a process of simply connected subdomains of $\D$ containing $0$, which is decreasing in the sense that $\lcTB'_t\subseteq \lcTB'_s$ for all $0\le s\le t \le \tau$. This is the description of the uniform $\CLE_4$ exploration towards $0$ most commonly found in the literature. Note that with this definition, time is parameterized according to the underlying Poisson point process, and entire loops are ``discovered instantaneously''. 

Since we are considering processes in $\lcS$, we need to reparameterize $\lcTB'$ by $-\log \CR$ seen from the origin. By definition, for each $j\in J$, $\gamma_j$ is a simple loop rooted at a point in $\partial \D$ that does not surround $0$. If we declare the loop to be traversed counterclockwise, we can view it as a curve $c_j:[0,f_j'(0)]\to \D$ parameterized so that $\CR(0;\D\setminus c_j)=e^{-t}$ for all $t$ (the choice of direction means that $\mathrm{int}(\gamma_j)$ is surrounded by the left-hand side of $c_j$). We then define $\lcTB$ to be the unique process in $\lcS$ such that for each $j\in J$ with $t_j\le \tau$, and all $t\in [-\log f_{t_j}'(0),- \log f_{t_j}'(0)-\log f_j'(0)]$, $\lcTB_t$ is the connected component of $f_{t_j}(\D\setminus c_j[0,t-\log f_{t_j}'(0)])$  containing $0$. In other words, $\lcTB$ is a reparameterization of $\lcTB'$ by $-\log \CR$ seen from $0$, where instead of loops being discovered instantaneously, they are traced continuously in a counterclockwise direction. The process is defined until time $\tau_0:=-\log \CR(0;f_{\tau}(D\setminus \gamma_{\tau} ))$, at which point the origin is surrounded by a loop (the law of this loop is that of the outermost loop surrounding the origin in a nested CLE$_4$ in $\D$).

With this definition, the same argument as in \cite[Section 4]{WW13} shows that ${\lcTB}^{n}_0\Rightarrow {\lcTB}_0$ in $\lcS$  as $n\to \infty$. Moreover, this convergence in law holds jointly with the convergence $\tau_0^n\Rightarrow \tau_0$ (in particular, $\tau_0$ has the law of the first time that a reflected Brownian motion started from $0$ hits $\pi$, as was already observed in \cite{SSW09}).

The $\CLE_4$ exploration can be continued after this first loop exploration time $\tau_0$ by iteration. More precisely, given the process up to time $\tau_0$, one next samples an independent $\CLE_4$ exploration in the interior of the discovered loop containing $0$, but now with loops traced clockwise instead of counterclockwise. When the next level loop containing $0$ is discovered, the procedure is repeated, but going back to counterclockwise tracing. Continuing in this way, we define the whole uniform CLE$_4$ exploration targeted at $0$: ${\lcB}_0=\{({\lcB}_0)_t \, ; \, t\ge 0\}$. Note that by definition $\lcTB_0$ is then just the process $\lcB_0$, stopped at time $\tau_0$.
\begin{remark}
	\label{rmk:slek_markov}
	The ``clockwise/counterclockwise'' switching defined above is consistent with what happens in the the $\SLE_\kp(\kp-6)$ picture when $\kp>4$. Indeed, it follows from the Markov property of $\theta_0^\eps$ (in the $\kp>4$ case) that after time $\tau_0^\eps$, the evolution of $\theta$ until it next hits $0$ is independent of the past and   equal in law to $(2\pi-\theta_0^\eps(t))_{t\in [0,\tau_0^\eps]}$. This implies that the future of the curve after time $\tau_0^\eps$ has the law of an $\SLE_\kp(\kp-6)$ in the connected component of the remaining domain containing $0$, but now with force point starting infinitesimally counterclockwise from the tip, until $0$ is surrounded by a clockwise loop.
	This procedure alternates, just as in the $\kp=4$ case.
\end{remark}

\subsubsection{{Exploration of the (nested) CLE}}\label{sec:sletocle}

In the previous subsections, we have seen how to construct $\SLE_\kp(\kp-6)$ processes, denoted by ${\lcB}^\eps_0$ ($\eps=\eps(\kappa')$) from $1$ to $0$ in $\D$, and that these are generated by curves $\eta^\eps$. We have also seen how to construct a uniform $\CLE_4$ exploration, ${\lcB}_0$, targeted at $0$ in $\D$. The $0$ in the subscripts here is to indicate that $0$ is a special \emph{target point}. But we can also define the law of an $\SLE_\kp(\kp-6)$, or a $\CLE_4$ exploration process, targeted at any point $z$ in the unit disk. To do this we simply take the law of $\phi(\lcB^\eps_0)$ or $\phi(\lcB_0)$, where $\phi:\D\to \D$ is the unique conformal map sending $0$ to $z$ and $1$ to $1$. We will denote these processes by $({\lcB}^\eps_z),{\lcB}_z$, where the $({\lcB}^\eps_z)$ are also clearly generated by curves $\eta^\eps_z$ for $\eps>0$. By definition, the time parameterization for $\lcB_z^\eps$ is such that $-\log \CR(z; (\lcB_z^\eps)_t)=t$ for all $t, z, \eps$ (similarly for $\lcB_z$).

In fact, both $\SLE_\kp(\kp-6)$ and the uniform $\CLE_4$ exploration satisfy a special \emph{target invariance} property: see for example \cite{SW05} {for $\SLE_\kp(\kp-6)$} and \cite[Lemma 8]{WW13} for CLE$_4$. This means that they can be targeted at a countable dense set of point in $\D$ simultaneously, in such a way that for any distinct $z,w\in \D$, the processes targeted at $z$ and $w$ agree (modulo time reparameterization) until the first time that $z$ and $w$ lie in different connected components of the yet-to-be-explored domain. We will choose our dense set of points to be $\mathcal{Q}:=\mathbb{Q}^2\cap \D$, and for $\eps>0$ refer to the coupled process $({\lcB}^\eps_z)_{z\in \mathcal{Q}}$ (or $(\eta^\eps_z)_{z\in \mathcal{Q}}$) as the \emph{branching $\SLE_\kp$} in $\D$. Similarly we refer to the coupled process $({\lcB}_z)_{z\in \mathcal{Q}}$ as the \emph{branching $\CLE_4$ exploration} in $\D$.

Note that in this setting we can associate a process $\theta_z^\eps$ to each $z \in \cQ$: we consider the image of $\lcB_z^\eps$ under the unique conformal map from $\D\to\D$ sending $z\mapsto 0$ and $1\mapsto1$, 
and define $\theta_z^\eps$ to be the unique process such that this new radial Loewner chain is related to $\theta_z^\eps$ via equations \eqref{def:Wfromtheta} and \eqref{eqn:rad_loewner}. Note that $\theta_z^\eps$ has the same law as $\theta_0^\eps$ for each fixed $z$ (by definition), but the above procedure produces a coupling of $\{\theta_z^\eps\, ; \, z\in \cQ\}$.

{We will make use of the following property connecting chordal and radial SLE (that is closely related to target invariance).}

\begin{lemma}[Theorem 3, \cite{SW05}] \label{lem:radial_chordal}
	Consider the radial $\SLE_{\kappa'}(\kappa'-6)$ with force point at $e^{-\im x}$ for $x\in (0,2\pi)$, stopped at the first time that $\e^{-\im x}$ and $0$ are separated. Then its law coincides (up to a time change) with that of a chordal SLE$_\kp$ from $1$ to $\e^{\im x}$ in $\D$, stopped at the equivalent time.
\end{lemma}

We remark that from $(\eta^\eps_z)_{z\in \mathcal{Q}}$, we can a.s.\ define a curve $\eta^\eps_a$ for any fixed $a\in \overline{\D}$, by taking the a.s.\ limit (with respect to the supremum norm on compacts of time) of the curves $\eta^\eps_{a_k}$, where $a_k\in \mathcal{Q}$ is a  sequence tending to $a$ as $k\to \infty$. This curve has the law of an $\SLE_{\kappa'}(\kp-6)$ from $1$ to $a$ in $\D$ \cite[Section 2.1]{MSW14}. Let us caution at this point that such a limiting construction does not work simultaneously for all $a$. Indeed, there are a.s.\ certain exceptional points $a$, the set of which a.s.\ has Lebesgue measure zero, for which the limit of $\eta^\eps_{a_k}$ does not exist for some sequence $a_k\rta a$. See Figure \ref{fig:loopdef}. 
\medskip

Let us now explain how, for each $\kp\in(4,8)$, we can use the branching $\SLE_\kp$ to define a (nested) $\CLE_\kp$. The \emph{conformal loop ensemble} $\CLE_{\kappa'}$ in $\D$ is a collection of non-crossing (nested) loops in the disk, \cite{SW12}, whose law is invariant under M\"{o}bius transforms $\D\to \D$. The ensemble can therefore be defined in any simply connected domain by conformal invariance, and the resulting family of laws is conjectured (in some special cases proved, e.g.\ \cite{CN08,Smi10,BH19,GMS19,KS19}) to be a universal scaling limit for collections of interfaces in critical statistical physics models.

\begin{figure}
	\centering
	\includegraphics[width=.9\textwidth]{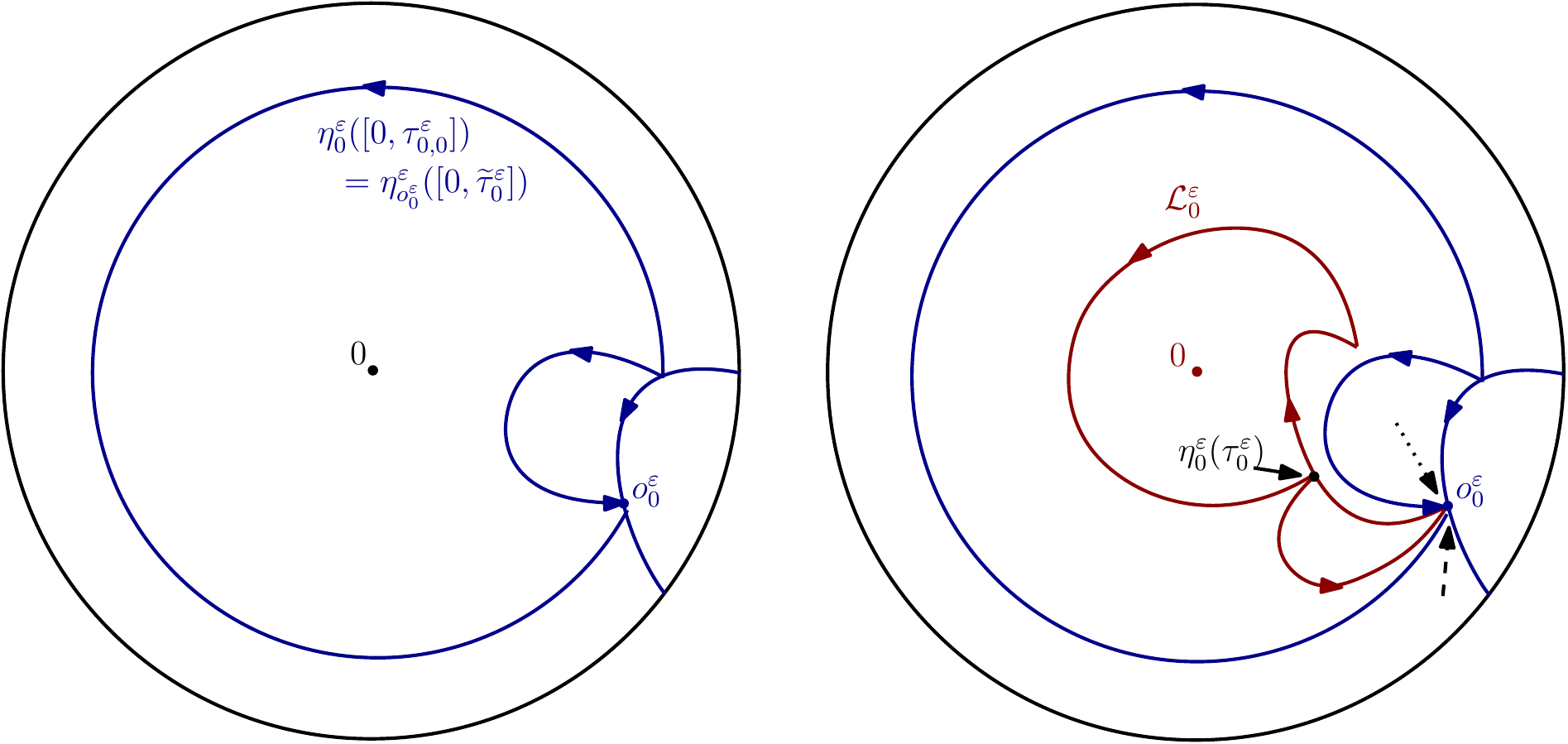}
	\caption{\textbf{On the left:} the curve $\eta_0^\eps$ (in blue) is run up to time $\tau_{0,0}^\eps$ (the last time that $\theta_0^\eps$ hits $0$ before hitting $2\pi$). The point $\eta_0^\eps(\tau_{0,0}^{\eps})$ is defined to be $o_0^\eps$ and we have that $\eta_0^\eps([0,\tau_{0,0}^{\eps}])=\eta_{o_0^\eps}^\eps([0,\wt{\tau}_0^\eps])$ for some time $\wt{\tau}_0^\eps$. \textbf{On the right:} the outermost CLE$_\kp$ loop $\cL_0^\eps$ containing $0$ (marked in red) is defined to be $\eta_{o_0^\eps}^\eps([\wt{\tau}_0^\eps,\infty])$. 
Note that we have a choice about how to define $\eta^{\eps}_{o^\eps_0}$: if we take it to be a limit of $\eta^\eps_{a_k}$ where $a_k\to o_0^\eps$ along the dotted line, this will be different to if $a_k\to o_0^\eps$ along the dashed line. We choose the definition that makes ${o^\eps_0}$ into a double point for $\eta^{\eps}_{o^\eps_0}$.}
	\label{fig:loopdef}
\end{figure} 
For $z\in \mathcal{Q}$, the procedure to define $\Loop^\eps_{z}$, the outermost $\CLE_\kp$ loop containing $z$, goes as follows:
\begin{itemize}
	\item 
	Let $\tau^\eps_{z}$ be the first time that $\theta^\eps_z$ hits $2\pi$, and let $\tau^\eps_{0,z}$ be the last time before this that $\theta^\eps_z$ is equal to $0$.\label{def:tauz}
	\item Let $o^\eps_z=\eta_z^\eps(\tau^\eps_{0,z})$. In fact, the point $o^\eps_z$ is one of the exceptional points for which the limit of $\eta^\eps_{a_k}$ does not exist for all sequences $a_k\rta o^\eps_z$, so it is not immediately clear how to define $\eta^\eps_{o_z^\eps}$, see Figure \ref{fig:loopdef}. However, the limit \emph{is} well defined if we insist that the sequence $a_k\to o_z^\eps$ is such that $0$ and $a_k$ are separated by $\eta_z^\eps$ at time $\tau_z^\eps$ for each $k$. 
	\item Define $\eta^\eps_{o_z^\eps}$ to be the limit of the curves $\eta^\eps_{a_k}$ as $k\to \infty$. In particular the condition on the sequence $a_k$ means that $o_z^\eps$ is a.s.\ a
	 double point of $\eta^\eps_{o_z^\eps}$. 
	  With this definition of $\eta^\eps_{o_z^\eps}$, it follows that \[\eta^\eps_z([0,\tau^\eps_{0,z}])=\eta^\eps_{o^\eps_z}([0,\wt{\tau}^{\eps}_z]) \text{ a.s.\  for some } \wt{\tau}^{\eps}_z\ge 0.\]
	\item Set $\Loop^\eps_z:=\eta^\eps_{o^\eps_z}([\wt{\tau}^{\eps}_z,\infty))$.   
\end{itemize}

We write $\bub^\eps_z$ for the connected component of $\D\setminus \Loop^\eps_z$ containing $z$: note that this is equal to $({\lcB}^\eps_z)_{\tau^\eps_z}$. We will call this the (outermost) $\CLE_{\kp}$ \emph{interior bubble} containing $z$.

We define the sequence of nested $\CLE_\kp$ loops $(\Loop^\eps_{z,i})$ for $i\ge 1$ by iteration (so $\Loop^\eps_z=:\Loop^\eps_{z,1}$), and denote the corresponding sequence of nested domains (interior bubbles) containing $z$ by $(\bub^\eps_{z,i})_{i\ge 1}$. More precisely, the $i$th loop is defined inside $\bub^\eps_{z,i-1}$ in the same way that the first loop is defined inside $\D$, after mapping $\bub^\eps_{z,i-1}$ conformally to $\D$ and considering the curve $\eta^\eps_z([\tau_z^\eps,\infty))$ rather than $\eta^\eps_z$.

The uniform $\CLE_4$ exploration defines a nested $\CLE_4$ in a similar but less complicated manner: see \cite{WW13}. For any $z\in \cQ$, to define $\Loop_z$ (the outermost $\CLE_4$ loop containing $z$) we consider the Loewner chain ${\lcTB}_z$  and define the times $\tau_z$ and $\tau_{0,z}$ (according to $\theta_z$) as in the $\kp>4$ case. Then between times $\tau_{0,z}$ and $\tau_z$ the Loewner chain ${\lcTB}_z$ is tracing a simple loop - starting and ending at a point $o_z$. This loop is what we define to be $\Loop_z$. 
We define $\bub_z$ to be the interior of $\Loop_z$: note that this is also equal to $({\lcB}_z)_{\tau_z}.$ Finally, we define the nested collection of $\CLE_4$ loops containing $z$ and their interiors by iteration, denoting these by $(\bub_{z,i},\Loop_{z,i})_{i\ge 1}$ (so $\bub_{z,1}:=\bub_z$ and $\Loop_{z,1}:=\Loop_z$).

\subsubsection{Space-filling SLE} \label{sec:sf_sle}

Now, for $\kappa'\in (4,8)$ we can also use the branching SLE$_\kp$, $(\eta^\eps_z)_{z\in \cQ}$, to define a space-filling curve $\eta^\eps$ known as space-filling SLE$_\kp$. This was first introduced in \cite{MSIG4,DMS14}\corr{; see also \cite[Appendix A.3]{BG20} for the precise definition of the space-filling \emph{loop} that we will use.} {The presentation here closely follows \cite{GHS19}.}

In our definition, the branches of $(\eta^\eps_z)_{z\in \cQ}$ 
are all $\SLE_\kp(\kp-6)$ processes started from the point $1$, and with force points initially located infinitesimally clockwise from $1$. This means that the associated space-filling SLE$_\kp$ will be a so-called \emph{counterclockwise space-filling $\SLE_\kp$ loop} from $1$ to $1$ in $\D$.\footnote{Variants of this process, e.g.\ chordal/whole-plane versions, a clockwise version, and version with another starting point, can be defined by modifying the definition of the branching SLE, see e.g.\ \cite{GHS19,AG19}.} 

Given an instance $(\eta^\eps_z)_{z\in \cQ}$ of a branching SLE$_\kp$, to define the associated space-filling SLE$_\kp$, we start by defining an ordering on the points of $\cQ$. For this we use a coloring procedure.  First, we color the boundary of $\D$ blue. 
Then, for each $z\in \cQ$, we can consider the branch $\eta^\eps_z$ of the branching SLE$_\kp$ targeted towards $z$. We color the left hand side of $\eta^\eps_z$ red, and the right hand side of $\eta^\eps_z$ blue. Whenever $\eta^\eps_z$ disconnects one region of $\D$ from another, we can then label the resulting connected components as \emph{monocolored} or \emph{bicolored}, depending on whether the boundaries of these components are made up of one or two colors, respectively. 

\begin{figure}\label{fig:spacefilling}
	\centering
	\includegraphics[width=.5\textwidth]{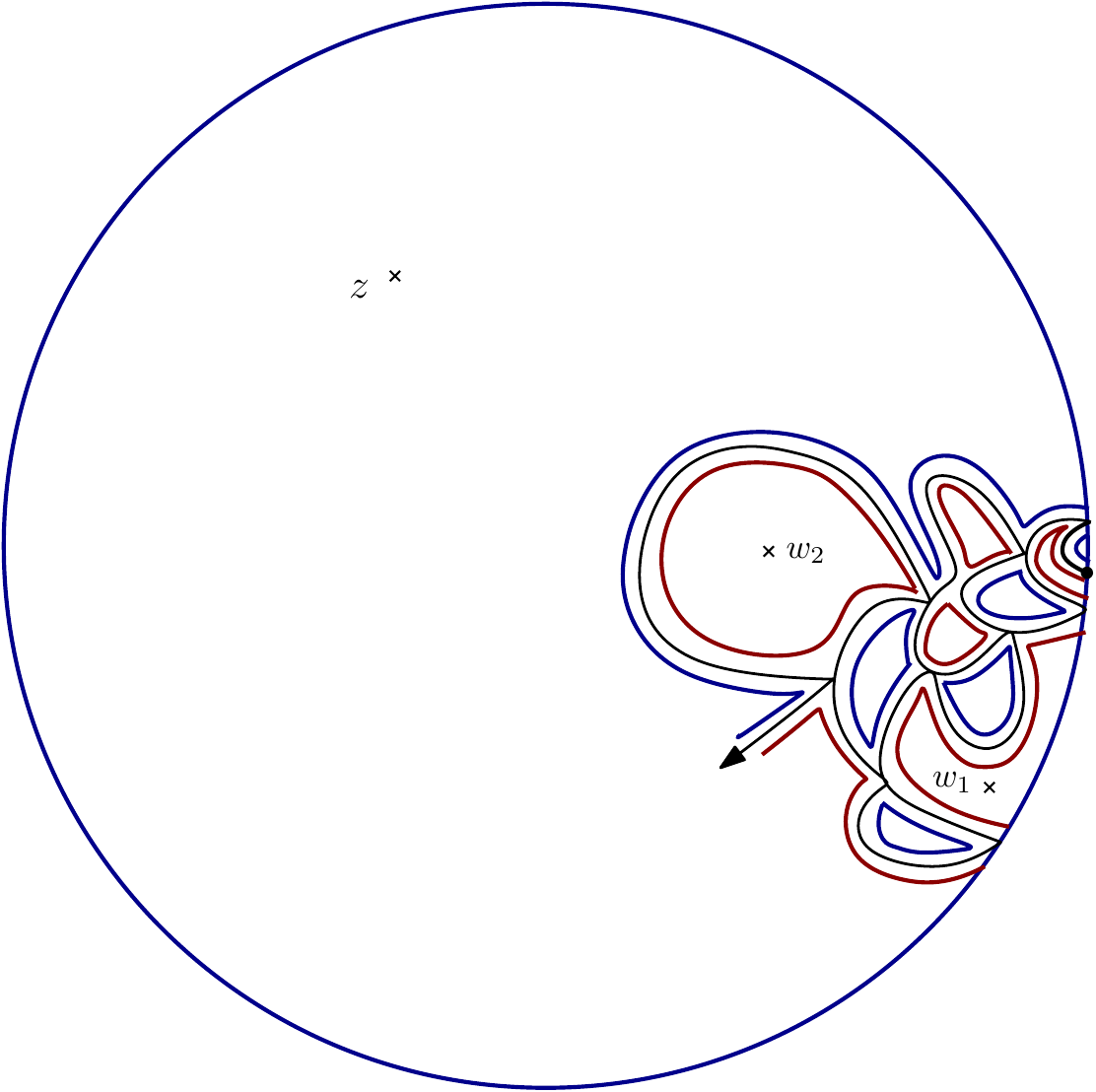}
	\caption{Constructing the ordering from the space-filling $\SLE_\kp$. When $z$ and $w_1$ are separated, the connected component containing $z$ has entirely blue boundary, while the connected component containing $w_1$ has red and blue on its boundary $\Rightarrow$ $z$ comes before $w_1$ in the ordering. By contrast, when $z$ and $w_2$ are separated, $w_2$ is in a monocolored component and $z$ is not{, which implies that} $z$ comes after $w_2$ in the ordering. So $w_1\prec z \prec w_2$ in this example.}
\end{figure}

For $z$ and $w$ distinct elements of $\cQ$, we know (by definition of the branching SLE) that $\eta^\eps_z$ and $\eta^\eps_w$ will agree until the first time that $z$ and $w$ are separated. When this occurs, it is not hard to see that precisely one of $z$ or $w$ will be in a newly created monocolored component. If this is $z$ we declare that $z\prec w$, and otherwise that that $w\prec z$. In this way, we define a consistent ordering $\prec$ on $\cQ$. See Figure \ref{fig:spacefilling}.

It was shown in \cite{MSIG4} that there is a unique continuous space-filling curve $\eta^\eps$, parametrized by Lebesgue area, that  visits the points of $\cQ$ in this order. This is the counterclockwise space-filling SLE$_\kp$ loop (we will tend to parametrize it differently in what follows, but will discuss this later). We make the following remarks. 

\begin{itemize}\setlength\itemsep{0em}
	\item We can think of $\eta^\eps$ as a version of ordinary $\SLE_{\kappa'}$ that iteratively fills in bubbles, or disconnected components, as it creates them. The ordering means that it will fill in monocolored components first, and come back to bicolored components only later.
	\item The word counterclockwise in the definition refers to the fact that the boundary of $\partial \D$ is covered up by $\eta^\eps$ in a counterclockwise order.
\end{itemize}

\subsection{Convergence of the SLE$_\kp(\kp-6)$ branches} \label{sec:conv_bt}
In this subsection and the next, we will show that for any $z\in \mathcal{Q}$, we have the joint convergence, in law as $\kp\downarrow 4$ of:
\begin{itemize}
	\item the $\SLE_{\kappa'}(\kp-6)$ branch towards $z$ to the $\CLE_4$ exploration branch towards $z$; and 
	\item the nested $\CLE_{\kappa'}$ loops surrounding $z$ to the nested $\CLE_4$ loops surrounding $z$.
\end{itemize}
The present subsection is devoted to proving the first statement. 
	
Let us assume without loss of generality that our target point $z$ is the origin. We first consider the radial $\SLE_\kp(\kp-6)$ branch targeting $0$, $\lcTB_0^\eps$, up until the first time $\tau_0^\eps$ that $0$ is surrounded by a counterclockwise loop. The basic result is as follows.
\begin{proposition}\label{prop:sletocleconv}
	$({\lcTB}_0^\eps,\tau_0^\eps)\Rightarrow ({\lcTB}_0,
	\tau_0)$ in $\lcS\times \R$ as $\eps\downarrow 0$.
\end{proposition}

By Remark \ref{rmk:slek_markov} and the iterative definition of the $\CLE_4$ exploration towards $0$, the convergence for all time follows immediately from the above:

\begin{proposition} \label{prop:convfullbranch}
	${\lcB}^\eps_0\Rightarrow {\lcB}_0$ in $\lcS$ as $\eps\downarrow 0$.
\end{proposition}

Our proof of Proposition \ref{prop:sletocleconv} will go through the approximations $\lcTB_0^{\eps,n}$ and $\lcTB_0^n$. Namely, we will show that for any \emph{fixed} level $n$ of approximation, $\lcTB_0^{\eps,n}\to \lcTB_0^n$ as $\eps\downarrow 0$, equivalently $\kp\downarrow 4$. Broadly speaking this holds since the macroscopic excursions of the underlying processes $\theta_0^\eps$ converge, and in between these macroscopic excursions we can show that the location of the tip of the curve distributes itself uniformly on the boundary of the unexplored domain. We combine this with the fact that the approximations $\lcTB_0^{\eps,n}$ converge to $\lcTB^\eps_0$ as $n\to \infty$, \emph{uniformly} in $\eps$, to obtain the result. 

The heuristic explanation for the mixing of the curve tip on the boundary is that the force point in the definition of an $\SLE_\kp(\kp-6)$ causes the curve to ``whizz'' around the boundary more and more quickly as $\kp\downarrow 4$. This means that in any fixed amount of time (e.g., between macroscopic excursions), it will forget its initial position and become uniformly distributed in the limit. {Making this heuristic rigorous is the main technical step of this subsection, and is achieved in Subsection \ref{subsec:whiz}.}

\subsubsection{Excursion measures converge as $\kp\downarrow 4$}
The first step towards proving Proposition \ref{prop:sletocleconv} is to describe the sense in which the underlying process $\theta^\eps_0$ for the $\SLE_{\kappa'}(\kappa'-6)$ branch converges to the process $\theta_0$ for the CLE$_4$ exploration. It is convenient to formulate this in the language of excursion theory; see Lemma \ref{lem:mn_conv} below.

To begin we observe, and record in the following remark, that when $\theta^\eps_0$ is very small, it behaves much like a Bessel process of a certain dimension.   
\begin{remark}
	\label{rmk:theta_bessel_compare}
	Suppose that $(\theta_0^\eps)_0=0$. By Girsanov's theorem, if the law of $\{(\theta^\eps_0)_t \, ; \, t\ge 0\} $ is weighted by the martingale $$\exp(Z_t^\eps-\frac{\langle Z^\eps \rangle_t}{2})\; ; \; Z^\eps_t:=\frac{\kappa'-4}{\sqrt{\kappa'}} \int_0^t (\frac{1}{(\theta^\eps_0)_s}-\frac{1}{2}\cot(\frac{(\theta^\eps_0)_s}{2})) \, dB_s ,$$ the resulting law of 
	$\{(\theta^\eps_0)_t\, ; \, {t\le \tau^\eps_0} \}$ is that of $\sqrt{\kp}$ times a Bessel process of dimension $\delta(\kp)=3-8/\kp$. \corr{	Note that for $y\in [0,2\pi)$, $(1/y-
	(1/2)\cot(y/2))$ is positive and increasing, and that for $y\in [0,\pi]$,  $y/12\le (1/y-(1/2)\cot(y/2)) \le y/6$, so in particular the integral in the definition of $Z_t^\ep$ is well-defined.} 
\end{remark}

Now, observe that by the Markov property of $\theta^\eps_0$, we can define its associated (infinite) excursion measure on excursions from $0$. {We define $m^\eps$ to be the image of this measure under the operation of stopping excursions if and when they reach height $2\pi$.}

For $n\ge 0$, we write $m^\eps_n$ for $m^\eps$ restricted to excursions 
with maximum height exceeding $2^{-n}$, and normalized to be a probability measure. It then follows from the strong Markov property that the excursions of $\theta_0^\eps$ during the intervals $[S_i^{\eps,n},T_i^{\eps,n}]$ are independent samples from $m_n^\eps$, and $\Lambda^{\eps,n}$ is the index of the first of these samples that actually reaches height $2\pi$. We also write $m^\eps_*$ for the measure $m^\eps$ restricted to excursions that reach $2\pi$, again normalized to be a probability measure.

Finally, we consider the excursion measure on excursions from $0$ for Brownian motion.  We denote the image of this measure, after stopping excursions when they hit $2\pi$, by $m$. Analogously to above, we write $m_n$ for $m$ conditioned on the excursion exceeding height $2^{-n}$. We write $m_\star$ for $m$ conditioned on the excursion reaching height $2\pi$.

The measures $m,(m^\eps)_\eps$ are supported on the excursion space 
\[ E = \{ e\in C(\R_+,[0,2\pi])\, ; \, e(0)=0, \zeta(e):=\sup\{s>0: e(s)\in (0,2\pi)\}\in (0,\infty)\}\]
on which we define the  distance \[d_E(e,e')=\sup_{t\ge 0} |e(t)-e'(t)| + |\zeta(e)-\zeta(e')|.\]

\begin{lemma}\label{lem:mn_conv}
	For any $n\ge 0$, $m_n^\eps \to m_n$ in law as $\eps\to 0$, with respect to $d_E$. The same holds with $(m_\star^\eps,m_\star)$  in place of $(m_n^\eps, m_n)$.
\end{lemma}

\begin{proof}
	For $a>0$, set $E^a = \{ e\in C(\R_+,[0,2\pi-a])\, ; \, e(0)=0, \zeta^a(e):=\sup\{s>0: e(s)\in (0,2\pi-a)\}\in (0,\infty)\}$, and equip it with the metric $d_{E^a}(e,e')=\sup_{t\ge 0} |e(t)-e'(t)| + |\zeta^a(e)-\zeta^a(e')|$.
	Set $\delta=\delta(\kp(\eps))$, recalling the definition $\delta(\kp)=8-3/\kp$.  {We first state and prove the analogous result for Bessel processes.}
	\begin{lemma}\label{bconv}
		Let $b^\eps$ be a sample from the Bessel-$\delta$ excursion measure away from $0$, conditioned on exceeding height $2^{-n}$, and stopped on the subsequent first hitting of $0$ or $2\pi-a$. Let $b$ be a sample from the Brownian excursion measure with the same conditioning and stopping.\footnote{Of course this depends on $a$, but we drop this from the notation for simplicity.} Then for any $a>0$, $b^\eps\Rightarrow b$ as $\eps\downarrow 0$, in the space $(E^a,d_{E^a})$.
	\end{lemma} 
	\begin{proofof}{Lemma \ref{bconv}} \corr{For} any $\eps\in (0,2-\sqrt{2})$, $b^\eps$ can be sampled (see 
		\cite[Section 3]{DMS14}) by:
		\begin{itemize}
			\item first sampling $X^\eps$ from the probability measure on $[2^{-n},\infty)$ with density proportional to $x^{\delta-3} dx$;
			\item then running a Bessel-$(4-\delta)$ process from $0$ to $X^\eps$;
			\item stopping this process at $2\pi-a$ if $X^\eps\ge 2\pi-a$; or 
			\item placing it back to back with the time reversal of an independent Bessel-$(4-\delta)$ from $0$ to $X^\eps$ if $X^\eps<2\pi-a$.
		\end{itemize}  Since the time for a Bessel-$(4-\delta)$ to leave $[0,a']$ converges to $0$ as $a'\to 0$ uniformly in $\delta<3/2$, and for any $a'<2^{-n}$, a Bessel-$(4-\delta)$ from $a'$ to $y$ converges in law to a Bessel$-3$ from $a'$ to $y$ as $\kp\downarrow 4$, uniformly in $y\in [2^{-n},2\pi]$, this shows that $b^\eps\Rightarrow b$ in $(E^a,d_{E^a})$.\end{proofof}
	\medskip 
	
{Now we continue the proof of Lemma \ref{lem:mn_conv}.}	Recalling the Radon--Nikodym derivative of Remark \ref{rmk:theta_bessel_compare} (note that $\kp-4\to 0$ as $\eps\downarrow 0$), we conclude that if $e^\eps$ and $e$ are sampled from $m_n^\eps$ and $m_n$ respectively, and stopped upon hitting $\{0,2\pi-a\}$ for the first time after hitting $2^{-n}$, then $e^\eps\to e$ in law as $\eps\downarrow 0$, in the space $(E^a,d_{E^a})$. 
	
	To complete the proof, it therefore suffices to show (now without stopping $e^\eps$ or $e$) that \[\zeta(e^\eps)-\zeta^a(e^\eps)\to 0 \;\;\; \text{ and } \;\;\; \sup_{t\in (\zeta^a(e^\eps),\zeta(e^\eps))} |e^\eps(t)-2\pi|\to 0\] as $a\to 0$, uniformly in $\eps$ (small enough). But by symmetry, if $\zeta^a(e^\eps)<\zeta(e^\eps)$ then $2\pi-e^\eps$ from time $\zeta^a(e^\eps)$ onwards has the law of $\theta^\eps$ started from $a$ and stopped upon hitting $0$ or $2\pi$. As $a\to 0$ the probability that this process remains in $[0,\pi]$ tends to $1$ uniformly in $\eps$, and then we can use the same Radon--Nikodym considerations to deduce the result. The final statement of Lemma \ref{lem:mn_conv} can be justified in exactly the same manner.
\end{proof}

\subsubsection{Strategy for the proof of Proposition \ref{prop:sletocleconv}}
With Lemma \ref{lem:mn_conv} in hand the strategy to prove Proposition \ref{prop:sletocleconv} is to establish the following two lemmas: 

\begin{lemma}\label{lem:ngoodapprox} Let $F$ be a continuous bounded function on $\lcS\times [0,\infty)$. 
	Then $\mathbb{E}[F({\lcTB}^{\eps,n}_0,\tau^{\eps,n}_0)]\to \mathbb{E}[F({\lcTB}^\eps_0,\tau^\eps_0)]$ 
	as $n\to \infty$, uniformly in $\kp\in (4,8)$, equivalently $\eps\in (0,2-\sqrt{2})$. 
\end{lemma}

\begin{proof} Fix $\eps$ as above, and let us assume that the processes ${\lcTB}^{\eps,n}_0$ as $n$ varies and ${\lcTB}^\eps_0$ are coupled together in the natural way: using the same underlying $\theta^\eps_0$ and $W^\eps_0$. \corr{By Remark \ref{rmk:dconv_cconv}, in particular \eqref{eq:lambdanlambda}, it suffices to prove that 	\begin{equation}\label{eq:tente}
			\tau^{\eps,n}_0\to \tau^{\eps}_0
		\end{equation}
		in probability as $n\to \infty$, uniformly in $\eps$. In other words, to show that the time spent by $\theta_0^\eps$ in excursions of maximum height less than $2^{-n}$ (before first hitting $2\pi$) goes to $0$ uniformly in $\eps$ as $n\to \infty$.}

\corr{To do this, let us consider the total (i.e., cumulative) duration $C^{\eps,n}$ of such excursions of $\theta_0^\eps$, before the the first time $\sigma^{\eps}$ that $\theta_0^\eps$ reaches $\pi$. The reason for restricting to this time interval is to make use of the final observation in Remark \ref{rmk:theta_bessel_compare}: that the integrand in the definition of $Z^\eps$ is deterministically bounded up to time $\sigma^\eps$. This will allow us to transfer the question to one about Bessel processes. And, indeed, since the number of times that $\theta_0^\eps$ will reach $\pi$ before time $\tau_0^\eps$ is a geometric random variable with success probability uniformly bounded away from $0$ (due to Lemma \ref{lem:mn_conv}),  it is enough to show that $C^{\eps,n}$ tends to $0$ in probability as $n\to \infty$, uniformly in $\eps$.}

\corr{For this, we first notice that by Remark \ref{rmk:theta_bessel_compare},  for any $a,S>0$ we can write 
	\begin{equation*}
		\mathbb{P}(C^{\eps,n}>a)\le \mathbb{P}(\sigma^\eps>S)+\mathbb{Q}^\eps(\exp(-Z_{\sigma^\eps}^\eps+\tfrac{1}{2}\langle Z^\eps \rangle_{\sigma^\eps}) \I_{\{C^{\eps,n}>a\}}\I_{\{\sigma^\eps\le S\}})
	\end{equation*} where $Z^\eps$ is as defined in Remark \ref{rmk:theta_bessel_compare} and under $\mathbb{Q}^\eps$, $\theta_0^\eps$ has the law of  $\sqrt{\kp}$ times a Bessel process of dimension $\delta(\kp)=3-8/\kp$. Since $\mathbb{P}(\sigma^\eps>S)\to 0$ as $S\to \infty$, uniformly in $\eps$ (this is proved for example in \cite{SSW09}), it suffices to show that for any fixed $S$, the second term in the above equation tends to $0$ uniformly in $\eps$ as $n\to \infty$. 

To this end, we begin by using Cauchy--Schwarz to obtain the upper bound
\begin{equation*}
	\mathbb{Q}^\eps(\exp(-Z_{\sigma^\eps}^\eps+\tfrac{1}{2}\langle Z^\eps \rangle_{\sigma^\eps} \I_{\{C^{\eps,n}>a\}}\I_{\{\sigma^\eps\le S\}})\big)^2\le 
\mathbb{Q}^\eps(\exp(-2Z_{\sigma^\eps}^\eps+\langle Z^\eps \rangle_{\sigma^\eps}) \I_{\{\sigma^\eps\le S\}}) \mathbb{Q}^\eps( \I_{\{C^{\eps,n}>a\}}).
\end{equation*}
Then, because we are on the event that $\sigma^\eps\le S$, and the integrand in the definition of $Z^\eps$ is deterministically bounded up to time $\sigma^\eps$, we have that 
$\mathbb{Q}^\eps(\exp(-2Z_{\sigma^\eps}^\eps+\langle Z^\eps \rangle_{\sigma^\eps}) \I_{\{\sigma^\eps\le S\}}) \le c$ for some constant $c=c(S)$ not depending on $\eps$.  So it remains to show that the $\mathbb{Q}^\eps$ expectation of $C^{\eps,n}$, %but now associated with a dimension $\delta(\kp)$ Bessel process instead of $\theta_0^\eps$, we have that the expectation of $C^{\kp,n}$ goes 
goes to $0$ uniformly in $\eps$ as $n\to \infty$.} %uniformly in $\kp$. }

\corr{Recall that  under $\mathbb{Q}^\eps$, $\theta_0^\eps$ has the law of  $\sqrt{\kp}$ times a Bessel process of dimension $\delta(\kp)=3-8/\kp$. Now, by  \cite[Theorem 1]{PitmanYor} we can construct a dimension $\delta(\kp)$ Bessel process by concatenating excursions from a Poisson point process $\Lambda$ with intensity $\int_0^{\infty} x^{\delta-3} \nu_\delta^x \, dx$ times Lebesgue measure on $E\times \R$,  where $\nu_\delta^x$ is a probability measure on Bessel excursions with maximum height $x$ for each $x>0$. Moreover, by Brownian scaling, $\nu_\delta^x(e)=\nu_\delta^1(e_x)$, $e_x(s)=x^{-1}e(x^{2}s)$ for $0\le s \le \zeta(e_x)=x^{-2}\zeta(e)$. (For proofs of these results, see for example \cite{PitmanYor}). 

Now, if we let $T=\inf\{t:(e,t)\in \Lambda \text{ and } \sup e(s) \ge \pi\}$, then  conditionally on $T$, we can write $C^{\kp, n}$ as the sum of the excursion lifetimes $\zeta(e)$ over points $(e,t)$ in a (conditionally independent) Poisson point process with intensity $$\int_0^{2^{-n}} x^{\delta-3} \nu_\delta^x \, dx \times \mathrm{Leb}([0,T]).$$ Note that by definition of the Poisson point process, $T$ is an exponential random variable with associated parameter $\int_\pi^\infty x^{\delta-3} \, dx$, and so has uniformly bounded expectation in $\kp$. Since Brownian scaling also implies that $\nu_\delta^x(\zeta(e)) =x^2\nu_\delta^{1}(\zeta(e_x))$ for excursions $e$,
Campbell's formula yields that the expectation of $C^{\kp,n}$ is of order $2^{-n\delta}$. This indeed converges uniformly to $0$ in $\delta\ge 1$  (equivalently $\kp,\eps$), which completes the proof.} 
\end{proof}
 
\begin{lemma}\label{lem:levelnconv} 
	For any fixed $n\in \N$, $({\lcTB}^{\eps,n}_0,\tau^{\eps,n}_0)$ converges to $({\lcTB}^n_0,\tau^n_0)$ in law as $\eps \downarrow 0$, with respect to the Carath\'{e}odory $\times$ Euclidean topology.
\end{lemma}

\begin{proofof}{Proposition \ref{prop:sletocleconv}} This follows by combining Lemma \ref{lem:ngoodapprox} and  
Lemma \ref{lem:levelnconv}, plus the fact that $(\lcTB_0^n,\tau_0^n)\Rightarrow (\lcTB_0,\tau_0)$ as $n\to \infty$. \end{proofof}

\subsubsection{Convergence at a fixed level of approximation as $\kp\downarrow 4$}\label{subsec:whiz}
The remainder of this section will now be devoted to proving Lemma \ref{lem:levelnconv}.
This is slightly trickier, and so we will break down its proof further into Lemmas \ref{lem:bigexsame} and \ref{lem:uniform} below.

 Let us first set-up for the statements of these lemmas. For $\kp\in (4,8)$ (equiv. $\eps\in (0,2-\sqrt{2})$) we set $X_i^{\eps,n}=(W_0^\eps)_{S_i^{\eps,n}}$ for $1\le i\le \Lambda^{\eps,n}$ and then write  $$\mathbf{X}^{\eps,n}=(X_1^{\eps,n},X_2^{\eps,n},\cdots, X_{\Lambda^{\eps,n}}^{\eps,n}).$$ For the $\CLE_4$ case, we write 
$$\mathbf{X}^{n}=(X_1^{n},X_2^{n},\cdots, X_{\Lambda^{n}}^{n})$$
where the $X^n$ are as defined in Section \ref{sec:ucle4}.
Also recall the definition of the excursions
$(e_i^{\eps,n})_{1\le i \le \Lambda^{\eps,n}}$ of $\theta^\eps$ above height $2^{-n}$. Define the corresponding excursions $(e_i^n)_{i\le \Lambda^n}$ for the uniform $\CLE_4$ exploration, and denote
$$\mathbf{e}^{\eps,n}=(e_1^{\eps,n},e_2^{\eps,n},\cdots, e_{\Lambda^{\eps,n}}^{\eps,n}), \quad \mathbf{e}^{n}=(e_1^{n},e_2^{n},\cdots, e_{\Lambda^{n}}^{n}).$$

Thus, $\mathbf{X}^{\eps,n}, \mathbf{X}^n$ live in the space of sequences of finite length, taking values in $\partial \D$. We equip this space with topology such that $\mathbf{X}^{(n)}\to \mathbf{X}$ as $n\to \infty$ iff the vector length of $\mathbf{X}^{(n)}$ is equal to the length of $\mathbf{X}$ for all $n\ge N_0$ large enough, and such that every component of $\mathbf{X}^{(n)}$ 
\corr{(for $n\ge N_0$)} converges to the corresponding component of $\mathbf{X}$ with respect to the Euclidean distance. Similarly, $\mathbf{e}^{\eps,n}, \mathbf{e}^n$ live in the space of sequences of finite length, taking values in the space $E$ of excursions away from $\{0,2\pi\}$. 

We equip this sequence space with topology such that $\mathbf{e}^{(k)}\to \mathbf{e}$ as $k\to \infty$ iff the vector length of $\mathbf{e}^{(k)}$ is equal to the vector length of $\mathbf{e}$ for all $k$ large enough, together with component-wise convergence with respect to $d_E$. 
\begin{lemma} \label{lem:bigexsame} 
	For any $n\in \N$, $(\mathbf{e}^{\eps,n},\tau^{\eps,n})\Rightarrow (\mathbf{e}^n,\tau^n)$ as $\eps\to 0$.
\end{lemma}

\begin{proof} This is a direct consequence of Lemma \ref{lem:mn_conv} and the definition of $\tau^{\eps,n},\tau^n$. \end{proof}

\begin{lemma}\label{lem:uniform} 
	For any $n\in \N$, $\mathbf{X}^{\eps,n}\to \mathbf{X}^n$ in law as $\eps\to 0$.
\end{lemma}

This second lemma will take a bit more work to prove. However, we can immediately see how the two together imply Lemma \ref{lem:levelnconv}:\\

\begin{proofof}{Lemma \ref{lem:levelnconv}}
	Lemmas \ref{lem:bigexsame} and \ref{lem:uniform} imply that the driving functions of $\lcTB^{\eps,n}_0$ converge in law to the driving function of $\lcTB^n_0$ with respect to the Skorokhod topology. This implies the result by Remark \ref{rmk:dconv_cconv}.
\end{proofof}\\

Our new goal is therefore to prove Lemma \ref{lem:uniform}. The main ingredient is the following (recall that $S_1^{\eps,n}$ is the start time of the first excursion of $\theta_0^\eps$ away from $0$ that reaches height $2^{-n}$).

\begin{lemma}\label{lem:uniform_equation}
	For any $u\ne 0$ and $n\in \N$ fixed, 
	\begin{equation}\label{eqn:ftuniform}
\corr{\mathbb{E}[\, X_1^{\eps,n}\, ]} =	\mathbb{E}[\,\exp(\im u \int_0^{S_1^{\eps,n}}\cot((\theta^\eps_0)_s/2) \, ds)\,]\to 0 \text{ as } \eps \downarrow 0.
	\end{equation}
\end{lemma} 

For the proof of Lemma \ref{lem:uniform_equation}, we are going to make use of Remark \ref{rmk:theta_bessel_compare}. That is, the fact that $\theta^\eps_0$ behaves very much like $\sqrt{\kp}$ times a Bessel process of dimension $\delta=3-8/\kp\in (1,2)$. The Bessel process is much more convenient to work with (in terms of exact calculations), because of its scaling properties. Indeed, for Bessel processes we have the following lemma:

{\begin{lemma}\label{lem:bes_uniform}
		Let $\wt \theta^\eps$ be $\sqrt{\kp}=\sqrt{\kp(\eps)}$ times a Bessel process of dimension $3-8/\kp$ (started from $0$) and $\wt S^{\eps,m}$ be the start time of the first excursion in which it exceeds $2^{-m}$. Then for $u\ne 0$, $$| \mathbb{E}[\exp\big(2\im u \int_0^{\wt{S}^{\eps,m}} (\wt \theta^\eps_s)^{-1} \, ds\big) ]|\to 0$$ as $\eps\downarrow 0$ for any $m$ large enough. 
\end{lemma} }

\noindent (The assumption that $m$ is sufficiently large here is made simply for convenience of proof.)\\

\begin{proof} 
	By changing the value of $u$ appropriately, we can instead take $\wt \theta^{\eps}$ to be a Bessel process of dimension $\delta(\kp)=3-8/\kp$ (i.e., we forget about the multiplicative factor of $\sqrt{\kp}$). Note that $\delta(\kp)\in (1,2)$ for $\kp<8$ and $\delta(\kp) \downarrow 1$ as $\kp\downarrow 4$. By standard It\^{o} excursion theory, $\wt \theta^\eps$ can be formed by gluing together the excursions of a Poisson point process $\Lambda$
	with intensity $\nu_{\delta(\ka)}\times \text{Leb}_{[0,\infty)}$, where $\nu_\delta$ is the Bessel-$\delta$ excursion measure. \corr{As mentioned previously, it is a classical result that we can decompose $\nu_\delta(\cdot)=\int_0^\infty x^{\delta-3}\nu_\delta^x(\cdot) \, dx$ (there is a multiplicative constant that we can set to one without loss of generality) where $\nu_\delta^x$ is a probability measure on excursions with maximum height exactly $x$ for each $x>0$ and that moreover by Brownian scaling, $\nu_\delta^x(e)=\nu_\delta^1(e_x)$, $e_x(s)=x^{-1}e(x^{2}s)$ for $0\le s \le \zeta(e_x)=x^{-2}\zeta(e)$.} 
	
	Let \begin{equation}\label{eq:Tkm} T^\kp_m \, \corr{\overset{(d)}{=}} \, \text{Exp}\left(\frac{(2^{-m})^{\delta-2}}{2-\delta} \right)\end{equation} be the smallest $t$ such that $(e,t)$ is in the Poisson process for some $e$ with $\sup(e)> 2^{-m}$. \corr{Then conditionally on $T_m^{\kp}$, the collection of points $(e,t)$ in the Poisson process with $t\le T_m^{\kp}$ is simply a Poisson process $\Lambda{(T_m^\kp)}$ with intensity $\int_0^{2^{-m}} x^{\delta-3}\nu_\delta^x \times \mathrm{Leb}([0,T_m^{\kp}])$. So, if for any given excursion $e\in E$, we define $$f(e)=\int_0^{\zeta(e)}\frac{1}{e(s)} \, ds$$
	(setting $f(e)=\infty$ if the interval diverges), we have 
		\begin{equation}\label{eq:campbell} \mathbb{E}(\e^{2 \im u \int_0^{\wt{S}^{\eps,m}} (\wt \theta_s^\eps)^{-1}\, ds} \, | \, T_m^\kp ) = \mathbb{E}(\e^{2\im u \sum_{(e,t)\in \Lambda{(T_m^\kp)}} f(e)} \, | \, T_m^\kp)=\exp\big(T_m^\kp \int_0^{2^{-m}} x^{\delta-3}\nu_\delta^x(1-\e^{2\im u f(e)}) \big) \end{equation}
		where in the final equality we have applied Campbell's formula for the Poisson point process $\Lambda{(T_m^\kp)}$.}

\corr{The real part of $1-\e^{2 \im u f(e)}$ is bounded above by $2 u^2 f(e)^2$. Then using the Brownian scaling property of $\nu_\delta^x$ explained before, we can bound $\nu_\delta^x(\Re (1-\e^{2 \im u f(e)}))$ by $u^2 x^2\nu_\delta^1(f^2)$. Using the fact that $\nu_\delta^1(f^2) < \infty$, which can be obtained from a direct calculation, it follows that $\int_0^{2^{-m}} x^{\delta-3}\nu_\delta^x(\Re (1-\e^{2 \im u f(e)})) \, dx< (2-\delta)^{-1} 2^{-m(\delta-2)}$ for all $m\ge M_0 = M_0(u)$, where $M_0<\infty$ does not depend on $\delta<3/2$ (say). This allows us to take expectations over $T_m^{\kp}$ in \eqref{eq:campbell} (recall the distribution of $T_m^\kp$ from \eqref{eq:Tkm}) to obtain that
	\begin{align}\label{eq:boundcampbell} \left|\mathbb{E}(\e^{2 \im u \int_0^{\wt{S}^{\eps,m}} (\wt \theta_s^\eps)^{-1}\, ds})\right| & = \left|1-2^{m(\delta-2)}(2-\delta) \int_0^{2^{-m}} x^{\delta-3} \nu_\delta^x((1-\cos(2uf(e))+\im \sin(2uf(e)))) \, dx \right|^{-1} \nonumber \\
		& \le \left|2^{m(\delta-2)}(2-\delta) \int_0^{2^{-m}} x^{\delta-3} \nu_\delta^x(\sin(2u f(e))) \, dx  \right|^{-1} \nonumber \\
		& \le \left|(2-\delta) \int_0^{1} y^{\delta-3} \nu_\delta^{2^{-m}y}(\sin(2u f(e))) \, dy  \right|^{-1}
	\end{align}
for all $m\ge M_0$ and $\delta\in (1,3/2)$.}
	
	\corr{We now fix $u\ne 0$ and $m\ge M_0$ for the rest of the proof. Our aim is to show that the final expression in \eqref{eq:boundcampbell} above converges to $0$ as $\delta\downarrow 1$ (equivalently $\eps\downarrow 0$). To do this, we use the Brownian scaling property of $\nu_\delta^x$ again to write $\nu_{\delta}^{2^{-m}y}(\sin(2uf(e)))=\nu_\delta^1 (\sin(2^{-m+1}uyf(e)))$ for each $y$. We also observe that $$y^{-1}\nu_\delta^1(\sin(2^{-m+1}uyf(e)))\to \nu_\delta^1(2^{-m+1}uf(e))$$ as $y\downarrow 0$, which follows by dominated convergence since  $\sin(z)/z\to 1$ as $z\downarrow 0$. Moreover (by Lemma \ref{lem:mn_conv}, say) the convergence is uniform in $\delta$. This means that for some $Y_{u,m}\in(0,1)$ and $ k_{u,m}<\infty$ depending only on $u$ and $m$, we have that
	$$|\nu_\delta^1(\sin(2^{-m+1}uyf(e)))\ge  k_{u,m} y \; \text{ for all } y\ge Y_{u,m}.$$
	It follows that 
	\begin{align*}\left|(2-\delta) \int_0^{1} y^{\delta-3} \nu_\delta^{2^{-m}y}(\sin(2u f(e))) \, dy  \right|
&	\ge  (2-\delta)k_{u,m}\int_0^{Y_{u,m}} y^{\delta-2} \, dy -(2-\delta)\int_{Y_{u,m}}^1 y^{\delta-3} \, dy \\
& \ge  \frac{k_{u,m}Y_{u,m}^{\delta-1}}{\delta-1}-(1-Y_{u,m}^{\delta-2}). \end{align*}
	for all $\delta\in (1,3/2)$. Since this expression converges to $\infty$ as $\delta\downarrow 1$, and the final term in \eqref{eq:boundcampbell} is its reciprocal, the proof is complete.}
\end{proof}\\

With this in hand, the proof of Lemma \ref{lem:uniform_equation} follows in a straightforward manner.\\

\begin{proofof}{Lemma \ref{lem:uniform_equation}}
	In order to do a Bessel process comparison and make use of Lemma \ref{lem:bes_uniform}, we need to replace the fixed $n$ in \eqref{eqn:ftuniform} by some $m$ which is very large (so we are only dealing with time intervals where $\theta^\eps_0$ is tiny). However, this is not a problem, since for $m\ge n$ we can write 
	\[ \int_0^{S_1^{\eps,n}} \cot((\theta^\eps_0)_s/2) \, ds = \int_0^{S_1^{\eps,m}} \cot((\theta^\eps_0)_s/2) \, ds + \int_{S_1^{\eps,m}}^{S_1^{\eps,n}} \cot((\theta^\eps_0)_s/2) \, ds 
	\]
	where  the two integrals are independent. This means that 
	$|\mathbb{E}[\,\exp(i u \int_0^{S_1^{\eps,n}}\cot((\theta^\eps_0)_s/2) \, ds)\,]|$ is actually increasing in $n$ for any fixed $\eps$, so proving \eqref{eqn:ftuniform}
	for $m>n$ also proves it for $n$.
	
	So we can write, for any $m\ge n$
	\[ |\mathbb{E}[\,\exp(\im u \int_0^{S_1^{\eps,n}}\cot((\theta^\eps_0)/2) \, ds)\,]|  \le |\mathbb{E}[\,\exp(\im u \int_0^{S_1^{\eps,m}}\cot((\theta^\eps_0)_s/2) \, ds)\,]| \] which is, by the triangle inequality, less than
\begin{equation*}\label{eqn:boundcomparebessel}
	\corr{ \left| \mathbb{E}[\exp\big(2\im u \int_0^{\wt{S}^{\eps,m}} (\wt \theta^\eps_s)^{-1} \, ds\big) ]\right| +	\left|\mathbb{E}[\,\exp\big(\im  u \int_0^{S_1^{\eps,m}}\cot((\theta^\eps_0)_s/2) \, ds\big)\,]-\mathbb{E}[\exp\big(2\im u \int_0^{\wt S^{\eps,m}} (\wt \theta^{\eps} _s)^{-1} \, ds\big) ]\right|. }
	\end{equation*}
	Now, \corr{using that $(1/y-
		(1/2)\cot(y/2))\downarrow 0$ as $y\downarrow 0$, and an argument almost identical to the first half of the proof of Lemma \ref{lem:ngoodapprox}}, the second term above 
	converges to $0$ as $m\to \infty$, uniformly in $\eps$. Since Lemma \ref{lem:bes_uniform} says that the first term converges to 0 as $\eps\to 0$ for any $m$ large enough, this completes the proof.  \end{proofof}\\

\begin{proofof}{Lemma \ref{lem:uniform}}  Equation \eqref{eqn:ftuniform} implies  that the law of $X_1^{\eps,n}$ converges to the uniform distribution on the unit circle as $\kp\downarrow 4$. 
	The full result then follows by the Markov property of $\theta^\eps_0$.
\end{proofof}

\subsubsection{Summary}
So, we have now tied up all the loose ends from the proof of Proposition \ref{prop:sletocleconv}. Recall that this proposition asserted the convergence in law of a single $\SLE_{\kappa'}(\kp-6)$ branch in $\D$, targeted at $0$, to the corresponding uniform CLE$_4$ exploration branch.  
Let us conclude this subsection by noting that the same result holds when we change the target point.

 For $z\in \D$ not necessarily equal to $0$, we define 
$\lcS_z$ to be the space of evolving domains whose image after applying the conformal map $f(w)=(w-z)/(1-\bar{z}w)$ from $\D\to \D$, $z\mapsto 0$, lies in $\lcS$.

From the convergence in Proposition \ref{prop:convfullbranch}, plus the target invariance of radial $\SLE_\kp(\kp-6)$ and the uniform CLE$_4$ exploration, it is immediate that:

\begin{corollary}\label{cor:convfullbranch} For any $z\in \cQ$, $({\lcB}_z^\eps,\tau_z^\eps)\Rightarrow ({\lcB}_z,\tau_z)$ in $\lcS_z\times \R$ as $\eps\to 0$.
\end{corollary}

\corr{Recall that $\tau_{0,z}^\eps$  is the last time that $\theta_z^\eps$  hits $0$ before first hitting $2\pi$ and $[\tau_{0,z},\tau_z]$ is the time interval during which $\lcB_z$ traces the outermost CLE$_4$ loop surrounding $z$. Notice that $\tau_z^\eps-\tau_{0,z}^\eps$ is equal to the length of the excursion $
	\mathrm{e}_{\Lambda^{\eps,n}}^{\eps,n}$ and similarly $\tau_z-\tau_{0,z}$ is the length of the excusion $\mathrm{e}_{\Lambda^n}$ (for every $n$), so that by Lemma \ref{lem:bigexsame} the following extension holds.}

\begin{corollary}
	\label{rmk:convfullbranch}
For any fixed $z\in \cQ$ $$({\lcB}_z^\eps, \tau_z^\eps,\tau_{0,z}^\eps
	)\Rightarrow ({\lcB}_z,\tau_z, \tau_{0,z}
	)$$ as $\eps\to 0$.
\end{corollary}

\subsection{Convergence of the CLE$_\kp$ loops} \label{sec:conv_loops}

Recall that for $z\in \cQ$, $\Loop_z^\eps$ (resp. $\Loop_z$) denotes the outermost $\CLE_{\kappa'}$ loop (resp. CLE$_4$ loop) containing $z$ and $\bub_z^\eps$ (resp. $\bub_z$) denotes the connected component of the complement of $\Loop_z^\eps$ (resp. $\Loop_z$) containing $z$. By definition we have 
\begin{equation}\label{eq:bub_loewner} \bub_z^\eps = ({\lcB}^\eps_z)_{\tau^\eps_z} \text{ and } \bub_z=({\lcB}_z)_{\tau_z},
	\end{equation}
where $\{(\lcB_z^\eps)_t\, ; \, t\ge 0\}$ and $\{(\lcB_z)_t\, ;  \, t\ge0\}$ {are processes in $\cD_z$} describing radial $\SLE_\kp(\kp-6)$ processes and a uniform $\CLE_4$ exploration, respectively, towards $z$. See Section \ref{sec:sletocle} for more details.

{In this subsection we will prove the convergence of $\Loop_z^\eps\Rightarrow \Loop_z$ with respect to the Hausdorff distance. That this might be non-obvious is illustrated by the following difference: in the limit $\partial \lcB_z = \Loop_z$, whereas this is not at all the case for $\eps > 0$. Nevertheless, we have:}

\begin{proposition}\label{prop:singleloopconv}For any $z\in \cQ$ 
	\label{prop:cleloopconv} $$({\lcB}^\eps_z, \Loop^\eps_z, \bub^\eps_z) \Rightarrow ({\lcB}_z, \Loop_z, \bub_z) $$ as $\eps\downarrow 0$, with respect to the product topology generated by ($\lcS_z$ $\times$ Hausdorff $\times$ \cart viewed from $z$) convergence. 
\end{proposition}

Given \eqref{eq:bub_loewner}, and that we already know the convergence of $\lcB_z^\eps$ as $\eps\downarrow 0$, the proof of Proposition \ref{prop:cleloopconv} boils down to the following lemma. 
\begin{lemma} 
	\label{lem:keyforcleloopconv}
	Suppose that $({\lcB}_0, \Loop, \bub_0)$ is a subsequential limit in law of $({\lcB}_0^\eps, \Loop_0^\eps, \bub_0^\eps)$ as $\eps\downarrow 0$ (with the topology of Proposition \ref{prop:cleloopconv}). Then we have $\Loop=\Loop_0$ a.s. 
\end{lemma}

\begin{proofof}{Proposition \ref{prop:cleloopconv} given Lemma \ref{lem:keyforcleloopconv}}
	By conformal invariance we may assume that $z=0$. 
	Observe that by Corollary \ref{cor:convfullbranch}, we already know that $({\lcB}_0^\eps, \bub_0^\eps)\Rightarrow ({\lcB}_0, \bub_0)$  as $\eps\to 0$, with respect to the product ($\lcS$ $\times$ \cart) topology. Indeed, if one takes a sequence $\eps_n$ converging to $0$, and a coupling of $({\lcB}_0^{\eps_n},\tau_0^{\eps_n})_{n\in \N}$ and $({\lcB}_0,\tau_0)$ so that $({\lcB}_0^{\eps_n},\tau_0^{\eps_n})\to ({\lcB}_0,\tau_0)$ a.s.\ as $n\to \infty$, it is clear due to \eqref{eq:bub_loewner} that each $\bub_0^{\eps_n}$ also converges to $\bub_0$ a.s. Also note that $(\Loop_0^\eps)$ is tight in $\eps$ with respect to the Hausdorff topology, since all the sets in question are almost surely contained in $\overline{\D}$. Thus $({\lcTB}_0^\eps, \bub_0^\eps, \Loop_0^\eps)$ is tight in the desired topology, and the limit is uniquely characterized by the above observation and Lemma \ref{lem:keyforcleloopconv}. This yields the proposition.
\end{proofof}

\subsubsection{Strategy for the proof of Lemma \ref{lem:keyforcleloopconv}}

At this point, we know the convergence in law of $({\lcB}_0^\eps, \bub_0^\eps)\to ({\lcB}_0, \bub_0)$ as $\eps\downarrow 0$, and we know that $\bub_0^\eps$ is the connected component of $\D\setminus \Loop_0^\eps$ containing $0$ for every $\eps$. Given a subsequential limit $({\lcB}_0, \bub_0, \Loop)$ in law of  $({\lcB}_0^\eps,\bub_0^\eps, \Loop_0^\eps)$, the difficulty in concluding that $\cL=\cL_0$ lies in the fact that \cart convergence (which is what we have for $\bub_0^\eps$) does not ``see'' bottlenecks: see Figure \ref{fig:cart_prob}.

To proceed with the proof, we first show that any part of the supposed limit $\Loop$ that does not coincide with $\Loop_0$ must lie \emph{outside} of $\bub_0$.
\begin{lemma}\label{cor:loopconvinclusion}
	With the set up of Lemma \ref{lem:keyforcleloopconv}, we have $\Loop\subseteq \C\setminus \bub_0$ almost surely.
\end{lemma}

Once we have this ``one-sided" result, it suffices to prove that the laws of $\Loop$ and $\Loop_0$ coincide:

\begin{lemma}\label{lem:convlooplaw}
	Suppose that $\cL$ is as in Lemma \ref{lem:keyforcleloopconv}. Then the law of $\cL$ is equal to the law of $\cL_0$.
\end{lemma}

The first lemma follows almost immediately from the \cart convergence of $\bub_0^\eps\to \bub_0$ (see the next subsection). To prove the second lemma, we use the fact that $\CLE_\kappa$ for $\kappa\in (0,8)$ is \emph{inversion invariant}: more correctly, a \emph{whole plane} version of $\CLE_\kappa$ is invariant under the mapping $z\mapsto 1/z$. Roughly speaking, this means that for whole plane CLE, we can use inversion invariance to obtain the complementary result to Lemma \ref{cor:loopconvinclusion}, and deduce Hausdorff convergence in law of the analogous loops. We then have to do a little work, using the relation between whole plane CLE and CLE in the disk (a Markov property), to translate this back to the disk setting and obtain Lemma \ref{lem:convlooplaw}. 
\begin{figure}
	\centering
	\includegraphics[width=0.5\textwidth]{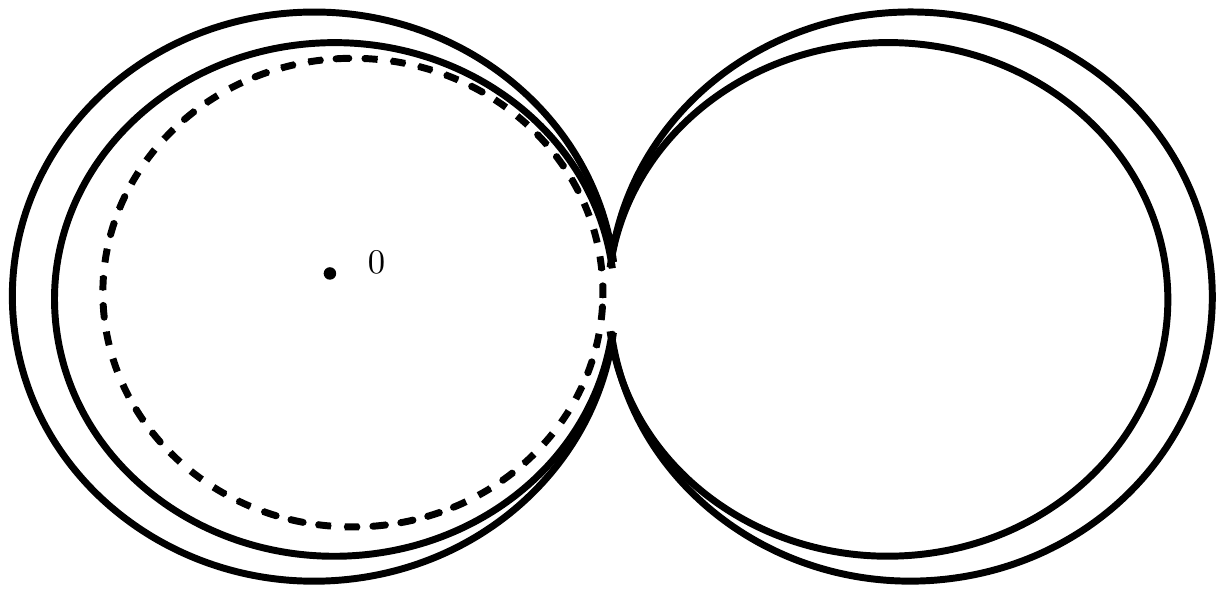}
	\caption{The sequence of domains enclosed by the thick black curves will converge in the \cart sense (viewed from $0$), but \emph{not} in the Hausdorff sense, to the dotted domain. This is the type of behavior that must be ruled out to deduce convergence of CLE loops (in the Hausdorff sense) from convergence of the radial SLE (in the \cart sense).}\label{fig:cart_prob}
\end{figure}

\subsubsection{Preliminaries on \cart convergence}
We first record the following standard lemma concerning \cart convergence, that will be useful in what follows.

\begin{lemma}[\cart kernel theorem]
	Suppose that $(U_n)_{n\ge 1}$ is a sequence of simply connected domains containing $0$, and for each $n$, write $V_n$ for the connected component of the interior of $\cap_{k\ge n} U_k$ containing $0$. Define the \emph{kernel} of $(U_n)_{n\ge 1}$ to be $\cup_n V_n$ if this is non-empty, otherwise declare it to be $\{0\}$. 
	
	Suppose that $(U_n)_{n\ge 1}$ and $U$ are simply connected domains containing $0$. Then $U_n\to U$ with respect to the \cart topology (viewed from 0) if and only if every subsequence of the $U_n$ has kernel $U$.\label{lem:kernel}
\end{lemma}

One immediate consequence of this is the following:

\begin{corollary}\label{cor:cart_inclusion}
	Suppose that $(K_n,D_n)\Rightarrow (K,D)$ as $n\to \infty$ for the product (Hausdorff $\times$ \cart topology), where for each fixed $n$, the coupling of $K_n$ and $D_n$ is such that $D_n$ is a simply connected domain with $0\in D_n$, and $K_n$ is a compact subset of $\C$ with $K_n\subseteq \C\setminus D_n$ almost surely. Then $K\subseteq \C\setminus D$ almost surely.
\end{corollary}
\begin{proof} By Skorokhod embedding, \corr{we may assume without loss of generality that $(K_{n},D_{n})\to (K,D)$} almost surely as $n\to \infty$.%and prove that $K\subseteq \C\setminus D$ almost surely. 
	
	 For $j\in \N$ write $V_{j}$ for the connected component of $\mathrm{int}(\cap_{k\ge j}D_{k})$ containing $0$. By assumption, $K_{n}\subseteq \C\setminus D_{n}$ for every $n$ almost surely, which means that $K_{n}\subseteq \C\setminus V_j$ for all $n\ge j$ almost surely. Since $K_{n}$ converges to $K$ in the Hausdorff topology, we have $K\subseteq \C\setminus V_j$ for each $j$, {which implies that} 
	$K\subseteq \C\setminus \cup_j V_j$ almost surely. Finally, because $D_{n}\to D$ in the \cart topology, the \cart kernel theorem gives that $\cup_j V_j=D$ almost surely. Hence $K\subseteq \C \setminus D$ almost surely, as desired.
\end{proof}

\noindent In particular:

\begin{proofof}{Lemma \ref{cor:loopconvinclusion}} This is a direct consequence of Corollary \ref{cor:cart_inclusion}.
\end{proofof}\\

Now, if $U_n\subseteq \C$ are such that $1/U_n:=\{z: 1/z\in U_n\}$ is a simply connected domain containing $0$ for each $n$, we say that $U_n\to U$ with respect to the \cart topology seen from $\infty$, iff $1/U_n\to 1/U$ with respect to the \cart topology seen from $0$. It is clear from this definition and the above arguments (or similar) that the following properties hold.

\begin{lemma}\label{lem:cartfrominfinityprops} Suppose that $U_n\in \C$ are simply connected domains such that $1/U_n$ is simply connected containing $0$ for each $n$. Then
	\begin{itemize}
		\item if $(U_n,K_n)\Rightarrow (U,K)$ jointly with respect the product (\cart seen from $\infty \times$ Hausdorff) topology, for some compact sets $K_n$ with $K_n\subseteq \C\setminus U_n$ for each $n$, then \corr{$K\subseteq \C\setminus U$} almost surely;
		\item if $(U_n, D_n)\Rightarrow (U,D)$ jointly with respect the product (\cart seen from $\infty \times$ \cart seen from $0$) topology, for some simply connected domains $\corr{\D}\supseteq D_n\ni 0$ with $D_n\subseteq \C\setminus \corr{U_n}$ for each $n$, then $D\subseteq \C\setminus \corr{U}$ almost surely.
	\end{itemize}
	
\end{lemma}

\corr{\begin{proof}
		The first bullet point follows from Corollary \ref{cor:cart_inclusion} by considering $1/U_n,1/U$ and $1/K_n,1/K$. For the second, let us assume by Skorohod embedding that $(U_n,D_n)\to (U,D)$ almost surely in the claimed topology. Then the compact sets $\partial D_n:=\bar{D_n}\setminus D_n\subset \bar{\D}$ are tight for the Hausdorff topology, and hence have some subsequential limit $\partial$. (The argument of) Corollary \ref{cor:cart_inclusion} implies that $\partial \subset \C\setminus U$ and $\partial \subset \C\setminus D$ almost surely. Since $U$ is an open simply connected domain containing $\infty$ and $D$ is an open simply connected domain containing $0$, this implies that $D\subset \C\setminus U$ almost surely.
		\end{proof}}
\subsubsection{Whole plane CLE and conclusion of the proofs}

As mentioned previously, we would now like to use some kind of symmetry argument to prove Lemma \ref{lem:convlooplaw}. However, the symmetry we wish to exploit is not present for CLE in the unit disk, and so we have to go through an argument using \emph{whole plane} CLE instead. Whole plane CLE was first introduced in \cite{KW16} and is, roughly speaking, the local limit of CLE in (any) sequence of domains with size tending to $\infty$. The key symmetry property of whole plane CLE$_\kp$ that we will make use of is its invariance under applying the inversion map $z\mapsto 1/z$ (\cite{KW16,GMQ18}). More precisely:

\begin{lemma}\label{lem:wpcle_props}
	Let $\Gamma^\kp$ be a whole plane $\CLE_\kp$ with $\kp\in [4,8)$. 
	\begin{itemize}
		\item \emph{(Inversion invariance)} The image of $\Gamma^\kp$ under  $z\mapsto 1/z$ has the same law as $\Gamma^\kp$.
		\item \emph{(Markov property)} Consider the collection of loops in $\Gamma^\kp$ that lie entirely inside $\D$ and surround $0$.  Write $I_1^{\eps}$ (with $\eps=\eps(\kp)$ as usual) for the connected component containing $0$ of the complement of the 
		{outermost} loop in this collection. Write $\fl_2^\eps$ for the second {outermost}
		loop in this collection.  Then the image of $\fl_2^\eps$ under the conformal map $I_1^\eps\to \D$ sending $z$ to $0$ with positive derivative at $0$ has the same law as the outermost loop surrounding $0$ for a $\CLE_\kp$ in $\D$. 
	\end{itemize} 
\end{lemma}

\begin{proof} The inversion invariance is shown in \cite[Theorem 1.1]{KW16} for $\kp=4$ and \cite[Theorem 1.1]{GMQ18} for $\kp\in (4,8)$. The Markov property follows from \cite[Lemma 2.9]{GMQ18} when $\kp>4$ and \cite[Theorem 1]{KW16} when $\kp=4$. \end{proof}\\

Let us now state the convergence result that we will prove for whole plane CLE$_\kp$ as $\kp\downarrow 4$, and show how it implies Lemma \ref{lem:convlooplaw}.

For $\eps>0$, we extend the above definitions and write $\frk{l}_1^\eps, \frk{l}_2^\eps$ for the largest and second largest whole plane $\CLE_{\kp}$ loops containing $0$, that are entirely contained in the unit disk. We let $I_i^\eps$ be the connected component of $\C\setminus \frk{l}_i^\eps$ containing $0$ for $i=1,2$ and let $E_i^\eps$ be the connected component containing $\infty$. When $\eps=0$ we write $\frk{l}_1,\frk{l}_2$ for the corresponding loops of a whole plane $\CLE_4$, and $I_1,E_1, I_2,E_2$ for the corresponding domains containing $0$ and $\infty$. Note that in this case we have $\overline{I_i}=\C\setminus E_i$ and $\overline{E_i}=\C\setminus I_i$ for $i=1,2$.
\begin{lemma}\label{lem:wholeplane}
	$(I_1^\eps, E_1^\eps, I_2^\eps, E_2^\eps)\Rightarrow (I_1, E_1, I_2, E_2)$ as $\eps\to 0$, with respect to the product \\ \cart  (seen from $(0,\infty,0,\infty)$ in the four coordinates) topology.
\end{lemma}

\begin{proofof}{Lemma \ref{lem:convlooplaw} given Lemma \ref{lem:wholeplane}} Suppose that $(I_1^\eps, \frk{l}_1^\eps)$ converges in law to $(I_1,\frk{l})$ along some subsequence, with respect to the product (\cart seen from 0 $\times$ Hausdorff) topology. By the above lemma, we can extend this convergence to the joint convergence of $(I_1^\eps, \frk{l}_1^\eps, E_2^\eps, I_2^\eps)\to (I_1,\frk{l},E_2, I_2)$. But then Corollary \ref{cor:cart_inclusion} and Lemma \ref{lem:cartfrominfinityprops} imply that $\frk{l}\subseteq \C\setminus I_2=\overline{E_2}$ and $\frk{l}\subseteq \C\setminus E_2= \overline{I_2}$ almost surely. This implies that $\frk{l}\subseteq \frk{l}_2=\partial(E_2)=\partial(I_2)$ almost surely. Moreover, it is not hard to see (using the definition of Hausdorff convergence) that $\frk{l_2}\setminus \frk{l}=\emptyset$, else $\frk{l}_2^\eps$ would not disconnect $0$ from $\infty$ for small $\eps$. So $\frk{l}=\frk{l}_2$ almost surely. 
	
	Now consider, for each $\eps$, the unique conformal map $g_1^\eps:I_1^\eps \to \D$ that sends $0\to 0$ and has $(g_1^\eps)'(0)>0$. Then the above considerations imply that if $g_1^\eps(\frk{l}_2^\eps)$ converges in law along some subsequence, with respect to the Hausdorff topology, then the limit must have the law of $g_1(\frk{l}_2)$, where $g_1:I_1\to \D$ is defined in the same way as $g_1^\eps$ but with $I^\eps_1$ replaced by $I_1$.  Since the law of $g_1^\eps(\frk{l}_2^\eps)$ is the same as that of $\Loop_0^\eps$ for every $\eps$ and the law of $g_1(\frk{l}_2)$ has the law of $\Loop_0$, this proves Lemma \ref{lem:convlooplaw}.\end{proofof}\\

\begin{proofof}{Lemma \ref{lem:keyforcleloopconv} and Proposition \ref{prop:cleloopconv}}
	Combining Lemmas \ref{cor:loopconvinclusion} and \ref{lem:convlooplaw} yields Lemma \ref{lem:keyforcleloopconv}. As explained previously, this implies Proposition \ref{prop:cleloopconv}.
\end{proofof}\\

So, we are left only to prove Lemma \ref{lem:wholeplane}, concerning whole plane $\CLE$. We will build up to this with a sequence of lemmas: first proving convergence of nested $\CLE$ loops in very large domains, then transferring this to whole plane CLE, and finally appealing to inversion invariance to obtain the result.

\begin{lemma}\label{lem:nestedcleconv}
	Fix $R>1$. For $\kp\in (4,8)$ and a $\CLE_{\kappa'}$ in $R\D$, denote by $(l_i^\eps)_{i\ge 1}$ the sequence of nested loops containing $0$, starting with the second smallest loop to \emph{fully} enclose the unit disk (set equal to the boundary of $R\D$ if only one or no loops in $R\D$ actually surround $\D$) {and such that $l_i^\eps$ surrounds $l_{i+1}^\eps$ for all $i$}. Write $(b_i^\eps)_{i\ge 1}$ for the connected components containing $0$ of the complements of the $(l_i^\eps)_{i\ge 1}$. Then $(b_i^\eps)_{i\ge 1}$ converges in law to its CLE$_4$ counterpart as $\eps \to 0$, with respect to the product \cart topology viewed from $0$.
\end{lemma}

\begin{proof}
	By Corollary \ref{cor:convfullbranch} and scale invariance of CLE, together with the iterative nature of the construction of nested loops, we already know that the sequence of nested loops in $R\D$ containing $0$, starting from the outermost one, converges as $\eps\to 0$, with respect to the product \cart topology viewed from $0$. Taking a coupling where this convergence holds a.s., it suffices to prove that the index of the smallest loop containing the unit disk also converges a.s. This is a straightforward consequence of the kernel theorem - Lemma \ref{lem:kernel} - plus the fact that the smallest $\CLE_4$ loop in $R\D$ that contains $\D$ actually contains $(1+r)\D$ for some strictly positive $r$ a.s.
\end{proof}

\begin{lemma} \label{lem:nestedwpcleconv}
	The statement of the above lemma holds true if we replace the CLEs in $R\D$ with whole plane versions.
\end{lemma}

\begin{proof}
	For fixed $\kappa\in [4,8)$, let $\Gamma^\C$, $\Gamma^{R\D}$ denote whole plane CLE$_\kp$  and $\CLE_{\kappa'}$ on $R\D$ respectively. The key to this lemma is Theorem 9.1 in \cite{MWW15}, which states (in particular) that $\Gamma^{R\D}$ rapidly converges to $\Gamma^\C$ in the following sense. For some $C,\alpha > 0$, $\Gamma^{R\D}$ and $\Gamma^\C$ can be coupled so that for any $r>0$ and $R>r$, with probability at least $1-C(R/r)^{-\alpha}$, there is a conformal map $\varphi$ from some $D\supset (R/r)^{1/4}\D$ to $D'\supset (R/r)^{1/4}\D$, which maps the nested loops of $\Gamma^{R\D}$ - starting with the smallest containing $r\D$ - to the corresponding nested loops of $\Gamma^\C$, and has low distortion in the sense that $|\varphi'(z)-1|\le C(R/r)^{-\alpha}$ on $R^{1/4}\D$.
	
	In fact, it is straightforward to see that $C$ and $\alpha$ (which in principle depend on $\kappa$) may be chosen uniformly for $\kappa\in [4,6]$ (say). Indeed, {it follows from the proof in \cite{MWW15} that} they depend only on the law of the log conformal radius of the outermost loop containing $0$ for a $\CLE_\kp$ in $\D$, and this varies continuously in $\kappa$, {\cite{SSW09}}. Hence, the result follows by letting $R\to \infty$ in Lemma \ref{lem:nestedcleconv} {and noting that the second smallest loop containing $\D$ is contained in $r\D$ with arbitrarily high probability as $r\to \infty$, uniformly in $\kappa$.}
\end{proof}
\medskip

\begin{proofof}{Lemma \ref{lem:wholeplane}}
	Lemmas \ref{lem:nestedwpcleconv} and \ref{lem:wpcle_props} (inversion invariance) imply that $(I_1^\eps,I_2^\eps)\Rightarrow (I_1,I_2)$ and $(E_1^\eps, E_2^\eps)\Rightarrow (E_1, E_2)$ as $\eps\to 0$. This ensures that $(I_1^\eps, E_1^\eps, I_2^\eps, E_2^\eps)$ is tight in $\eps$, so we need only prove that if $(I_1, \hat{E}_1, I_2, \hat{E}_2)$ is a subsequential limit of $(I_1^\eps, E_1^\eps, I_2^\eps, E_2^\eps)$, then $\hat{E}_1=E_1=\mathrm{int}(\C\setminus I_1)$ and $\hat{E}_2=E_2=\mathrm{int}(\C\setminus I_2)$ almost surely. Note that $(\hat{E}_1,\hat{E}_2)$ has the same law as $(E_1, E_2)$, and since $I_1^\eps\subseteq \C\setminus E_1^\eps$ for all $\eps$, \corr{Lemma \ref{lem:cartfrominfinityprops}} implies that $I_1\subseteq \C\setminus \hat{E}_1$. In other words $ \hat{E}_1\subseteq E_1$ almost surely. Then because $\hat{E}_1$ and $E_1$ have the same law, we may deduce that they are equal almost surely. Similarly we see that $\hat{E}_2=E_2$ almost surely.
\end{proofof}

\subsubsection{Conclusion}
Recall that for $z\in \D$, $(\bub_{z,i}^\eps,\Loop_{z,i}^\eps)_{i\ge 1}$ (resp.  $(\bub_{z,i},\Loop_{z,i})_{i\ge 1}$) denotes the sequence of nested $\CLE_{\kappa'}$ (resp. $\CLE_4$) bubbles and loops containing $z$. 
By the Markov property and iterative nature of the construction, it is immediate from Proposition \ref{prop:cleloopconv} that: 

\begin{corollary}
	\label{cor:cleloopconv} For fixed $z\in \cQ$
	$$({\lcB}^\eps_z, (\Loop^\eps_{z,i})_{i\ge 1}, (\bub^\eps_{z,i})_{i\ge 1}) \Rightarrow ({\lcB}_z, (\Loop_{z,i})_{i\ge 1}, (\bub_{z,i})_{i\ge 1})$$ as $\eps\downarrow 0$, with respect to the product topology generated by ($\lcS_z$ $\times$ $\prod$ Hausdorff $\times$ $\prod$ \cart viewed from $z$) convergence.
\end{corollary}

\section{The uniform space-filling SLE$_4$} \label{sec:conv_order}

In this section we show that the ordering on points (with rational co-ordinates) in the disk, induced by space-filling $\SLE_\kp$ with $\kp>4$, converges to a limiting ordering as $\kp\downarrow 4$. We call this the uniform space-filling SLE$_4$.\footnote{This name is partially inspired from the fact that the process is constructed via a uniform CLE$_4$ exploration, and partly since, every time the domain of exploration is split into two components, the components are ordered uniformly at random.} Nonetheless, we can describe explicitly the law of this ordering, which for any two fixed points comes down to the toss of a fair coin. {As for $\kappa' > 4$, there would be other ways to define a space-filling SLE$_4$ process, by considering different explorations of CLE$_4$.}

Let us now recall some notation in order to properly state the result.  For $\eps\in (0,2-\sqrt{2})$ and $z,w\in \cQ$, we define $\cO_{z,w}^\eps$ to be the indicator function of the event that the space-filling $\SLE_{\kappa'}$ $\eta^\eps$ 
hits $z$ before $w$ (see Section \ref{sec:sf_sle}). By convention we set this equal to $1$ when $z=w$. 

To describe the limit as $\kp\downarrow 0$, we define $\cO=(\cO_{z,w})_{z,w\in \cQ}$  to be a collection of random variables, coupled with $(\lcB_z)_{z\in\cQ}$ such that \emph{conditionally given $(\lcB_z)_{z\in\cQ}$}: 
\begin{itemize}
	\itemsep0em
	\item $\cO_{z,z}=1$ for all $z\in \cQ$ a.s.; 
	\item  $\cO_{z,w}$ is a Bernoulli($\frac{1}{2}$) random variable for all $z,w\in\cQ$ with $z\ne w$;  
	\item $\cO_{z,w}=1-\cO_{w,z}$ for all $z,w\in\cQ$ with $z\ne w$ a.s.;
	\item for all $z,w_1,w_2\in \cQ$ with $z\ne w_1, w_2$, if $\lcB_z$ separates $z$ from $w_2$ at the same time as it separates $z$ from $w_1$ then $\cO_{z,w_1}=\cO_{z,w_2}$, otherwise $\cO_{z,w_1}$ and $\cO_{z,w_2}$ are independent. 
\end{itemize}
\begin{lemma}
	\label{lem:defo}There is a unique joint law on $((\lcB_z)_{z\in \cQ},\cO)$ satisfying the above requirements, and such that the marginal law of $(\lcB_z)_{z\in \cQ}$ is that of a branching uniform $\CLE_4$ exploration. With this law, $\cO$ a.s.\ defines an order on any finite subset of $\cQ$ by declaring that $z\preceq w$ iff  $\cO_{z,w}=1$.
\end{lemma} 
We will prove the lemma in just a moment. The main result of this section is the following.

\begin{proposition}\label{prop:convbranchingsleorder}
	$(({\lcB}^{\eps}_z)_{z\in \mathcal{Q}},(\cO^\eps_{z,w})_{z,w\in \cQ})$ converges to $(({\lcB}_z)_{z\in \cQ},(\cO_{z,w})_{z,w\in \cQ})$, in law as $\eps\downarrow 0$, with respect to the product topology  $\left(\prod_{\cQ} \lcS_z \,\times \, \prod_{\cQ\times \cQ} \text{discrete}\right)$, where $(\cO_{z,w})_{z,w\in \cQ}$ is as defined in Lemma \ref{lem:defo}. 
\end{proposition}

\begin{proofof}{Lemma \ref{lem:defo}}
	The main observation is that if a joint law $((\lcB_z)_{z\in \cQ},\cO)$ as in the lemma exists, then for all $z,w,y\in \cQ$ we a.s.\ have \begin{equation}\label{order_prop}\{\cO_{z,w}=1\}\cap \{\cO_{w,y}=1\}\Rightarrow \{\cO_{z,y}=1\}.\end{equation}
	To verify this, we assume that $z,w,y$ are distinct (else the statement is trivial) with $\cO_{z,w}=1$ and $\cO_{w,y}=1$. Since $\cO_{w,z}=1-\cO_{z,w}=0$ this implies that $y$ and $z$ are \emph{not} separated from $w$ by $\lcB_w$ at the same time. If $\lcB_w$ separates $z$ from $w$ strictly before separating $y$ from $w$, then $\lcB_z$ separates $y$ and $w$ from $z$ at the same time, so $\cO_{z,y}=\cO_{z,w}=1$. If $\lcB_w$ separates $y$ from $w$ strictly before separating $z$ from $w$, then $\lcB_y$ separates $z$ and $w$ from $y$ at the same time, so $\cO_{z,y}=1-\cO_{y,z}=1-\cO_{y,w}=\cO_{w,y}=1$. In either case it must be that $\cO_{z,y}=1$.
	
We now show why this implies that for any $\{z_1, \cdots, z_k\}$ with $z_i\in \cQ$ distinct, there exists a unique a conditional law on $(\cO_{z_i,z_j})_{1\le i,j\le k}$ given $(\lcB_z)_{z\in \cQ}$, satisfying the requirements of the lemma. We argue by induction on the number of points. Indeed, suppose it is true with $1\le k \le n-1$ for some $n$ and take $\{z_1,\cdots, z_n\}$ in $\cQ$ distinct. We construct the conditional law of $(\cO_{z_i,z_j})_{1\le i,j\le n}$ given $(\lcB_z)_{z\in \cQ}$ as follows.
	\begin{itemize}
		\item To define $(\cO_{z_1,z_i})_{1\le i \le n}$:
		\begin{itemize}\item partition the indices $\{2,\cdots, n\}$ into equivalence classes $\{C_1,\dots, C_K\}$ such that $i\sim j$ iff $\lcB_{z_1}$ separates $z_1$ from $z_i$ and $z_j$ at the same time; \item for each equivalence class sample an independent Bernoulli$(1/2)$ random variable; \item set $\cO_{z_1,z_i}$ to be the random variable associated with class $[i]$ for every $i$. \end{itemize}
		\item Given $(\cO_{z_1,z_i})_{1\le i \le n}$ and $(\lcB_z)_{z\in \cQ}$, define $\cO_{z_i,z_j}$ with $[i]\ne [j]$ by setting it equal to  $\cO_{z_1,z_j}$ if $z_i$ and $z_1$ are separated from $z_j$ at the same time, or $\cO_{z_1,z_i}$ if $z_j$ and $z_1$ are separated from $z_i$ at the same time.
		\item For each $1\le l\le K$ consider the connected component $U_l\subset \D$ in the branching $\CLE_4$ exploration that contains the points $z_i$ with $[i]=C_l$  when they are separated from $z_1$. The $\CLE_4$ explorations inside these components are mutually independent, independent of the $\CLE_4$ exploration before this separation time, and each has the same law as $(\lcB_z)_{z\in \cQ}$ after mapping to the unit disk. Thus since each equivalence class contains strictly less than $n$ points, using the induction hypothesis, we can define $(\cO_{z_i,z_j})_{i\ne j, [i]=[j]=C_l}$ for $1\le l\le K$ such that:
		\begin{itemize}
			\item the collections for different $l$ are mutually independent;
			\item $(\cO_{z_i,z_j})_{i\ne j, [i]=[j]=C_l}$ for each $l$ is independent of the CLE$_4$ exploration outside of $U_l$, and after conformally mapping everything to the unit disk, is coupled the exploration inside $U_l$ as in the statement of Lemma \ref{lem:defo}.
		\end{itemize}
\end{itemize}
Using the induction hypothesis, it is straightforward to see that this defines a conditional law on $(\cO_{z_i,z_j})_{1\le i\ne j\le n}$ given $(\lcB_z)_{z\in \cQ}$ that satisfies the conditions of the Lemma. Moreover, note that the first two bullet points above, together with \eqref{order_prop}, define the law of $(\cO_{z_1,z_j})_{1\le j\le n}$ and $(\cO_{z_i,z_j})_{[i]\ne [j]}$ (satisfying the requirements) uniquely. Combining with the uniqueness in the induction hypothesis, it follows easily that the conditional law of $(\cO_{z_i,z_j})_{1\le i\ne j\le n}$ given $(\lcB_z)_{z\in \cQ}$ (satisfying the requirements) is unique.

Consequently, given $(\lcB_z)_{z\in \cQ}$, there exists a unique conditional law on the product space $\{0,1\}^{\cQ\times \cQ}$ equipped with the product $\sigma$-algebra, such that if $\cO=(\cO_{z,w})_{z,w\in \cQ}$ has this law then it satisfies the conditions above Lemma \ref{lem:defo}.

This concludes the existence and uniqueness statement of the lemma. 
The property \eqref{order_prop}  implies that $\cO$ does a.s.\ define an order on any finite subset of $\cQ$. 
\end{proofof}\\

In the coming subsections we will prove Proposition \ref{prop:convbranchingsleorder}. Since tightness of all the random variables in question is immediate (either by definition or from our previous work) it suffices to characterize any limiting law. We begin in Section \ref{sec:twopoints} by showing this {for the order of two points}; see just below for an outline of the strategy.
Then,  we will prove that the \emph{time} at which they are separated by the $\SLE_\kp(\kp-6)$ converges (for the $-\log \CR$ parameterization with respect to either of the points). This is important for characterizing joint limits, when there are three or more points being considered. It also turns out to be non-trivial, due to pathological behavior that cannot be ruled out when one only knows convergence of the SLE branches in the spaces $\lcS_z$.  We conclude the proof in a third subsection, and finally combine this with the results of Section \ref{sec:conv_clesle} to summarize the ``Euclidean'' part of this article in Proposition \ref{prop:cle-conv}.

\subsection{Convergence of order for two points}\label{sec:twopoints}
In this section we show that for two distinct points $z,w\in \D$, the law of the order in which they are visited by the space-filling SLE$_\kp$ $\eta^\eps$, converges to the result of a fair coin toss as $\kp\downarrow 4$. That is, $\cO_{z,w}^\eps$ converges to a Bernoulli$(1/2)$ random variable as $\eps\downarrow 0$. The rough outline of the proof is as follows

Recall that $\eta^\eps$ is determined by an $\SLE_\kp(\kp-6)$ branching tree, in which $\eta^\eps_z$ denotes the SLE$_\kp(\kp-6)$ branch towards $z$ (parameterized according to minus log conformal radius as seen from $z$). If we consider the time $\sigma^\eps_{z,w}$ at which $\eta^\eps_z$ separates $z$ and $w$, then for every $\eps>0$, $\cO_{z,w}^\eps$ is actually measurable with respect to $\eta^\eps_z([0,\sigma^\eps_{z,w}])$. 
So what we are trying to show is that this measurability turns to independence in the $\eps\downarrow 0$ limit. This means that we will not get very far if we consider the conditional law of $\cO_{z,w}^\eps$ given $\eta^\eps_z([0,\sigma^\eps_{z,w}])$, so instead we have to look at times just before $\sigma_{z,w}^\eps$. Namely, we will consider the times $\sigma_{z,w,\delta}^\eps$ that $w$ is sent first sent to within distance $\delta$ of the boundary by the  Loewner maps associated with $\eta^\eps_z$. We will show that for any \emph{fixed} $\delta\in (0,1)$, the conditional probability that $\cO_{z,w}^\eps=1$, given $\eta_z^\eps([0,\sigma_{z,w,\delta}^\eps])$, converges to $1/2$ as $\eps\to 0$. Knowing this for every $\delta$ allows us to reach the desired conclusion. 

To show that these conditional probabilities do tend to $1/2$ for fixed $\delta$, we apply the Markov property at time $\sigma_{z,w,\delta}^\eps$. This tells us that after mapping $(\lcB^\eps_z)_{\sigma_{z,w,\delta}^\eps}$ to the unit disc, the remainder of $\eta^\eps_z$ evolves as a radial $\SLE_\kp(\kp-6)$ with a force point somewhere on the unit circle. And we know the law of this curve: initially it evolves as a chordal $\SLE_\kp$ targeted at the force point, and after the force point is swallowed, it evolves as a radial $\SLE_\kp(\kp-6)$ in the to be discovered domain with force point starting adjacent to the tip. So we need to show that for such a process, the behavior is ``symmetric'' in an appropriate sense. 
In fact, we have to deal with two scenarios, according to whether the images of $z$ and $w$ are separated or not when the force point is swallowed. If they are separated, our argument becomes a symmetry argument for chordal $\SLE_\kp$. If they are not, our argument becomes a symmetry argument for space-filling $\SLE_\kp$.  
For a more detailed outline of the strategy, and the bulk of the proof, see Lemma \ref{lem:1/2}. 

At this point, let us just record the required symmetry property of space-filling $\SLE_\kp$ in the following lemma. 
\begin{lemma}\label{lem:pre-1/2} Let $\eta^\eps$ be a space-filling $\SLE_{\kp(\eps)}$ in $\D$, as above. Then for any $x\in \D$:
	$$\mathbb{P}(\eta^\eps \text{ hits } 0 \text{ before } x)
	\to \frac{1}{2} \text{ as } \eps\to 0.$$
\end{lemma}
\begin{proof}
For this we use a conformal invariance argument. Namely, we notice that by conformal invariance of $\eta^\eps$, applying the map $$z\mapsto \frac{1-\bar{x}}{1-x}\frac{z-x}{1-\bar{x}z}$$ from $\D$ to $\D$ that sends $1$ to $1$ and $x$ to $0$, we have 
$$ \mathbb{P}[\eta^\eps \text{ hits } 0 \text{ before } x] = \mathbb{P}[\eta^\eps \text{ hits } \hat{x} \text{ before } 0]= 1- \mathbb{P}[\eta^\eps \text{ hits } 0\text{ before } \hat{x} ]$$
where $\hat{x}=-x(1-\bar{x})(1-x)^{-1}$ is the image of $0$ under the conformal map, and $|\hat{x}|=|x|$. Hence it suffices to show that	$$ \mathbb{P}[\eta^\eps \text{ hits } 0 \text{ before } x] - \mathbb{P}[\eta^\eps \text{ hits } 0 \text{ before } \hat{x}]\to 0$$ as $\eps\to 0$. By rotational invariance, if we write $\eta^\eps_{\theta}$ for a space-filling SLE$_\kp$ starting at $\e^{i\theta}$, then it is enough to show that
$$ \mathbb{P}[\eta^\eps_{\theta} \text{ hits } 0 \text{ before } |x|]-\mathbb{P}[\eta^\eps_{0} \text{ hits } 0 \text{ before } |x|]\to 0$$ as $\eps\to 0$, for any $\theta\in[0,2\pi]$. 

However, this is easily justified, because we can couple an $\SLE_{\kappa'}(\kp-6)$ from $1$ to $0$ and another from $\e^{i\theta}$ to $0$, so that they successfully couple (i.e., coincide for all later times) before $0$ is separated from $|x|$ with arbitrarily high probability (uniformly in $\theta$) as $\kp\downarrow 4$. This follows from Lemma \ref{lem:uniform_equation}, target invariance of the SLE$_\kp(\kp-6)$ and \eqref{eq:tente}; i.e., because in an arbitrarily small amount of time as $\kp\downarrow 4$, the $\SLE_{\kappa'}(\kp-6)$ will have swallowed every point on $\partial \D$. 
\end{proof}\\

Now we proceed with the set-up for the main result of this section (Proposition \ref{prop:order_simple} below). Recall that ${\lcB}_z\in {\lcS}$ is the sequence of domains formed by the branch of the uniform CLE$_4$ exploration towards $z$ in $\D$. For $w\ne z$, we write $\sigma_{z,w}$ for the first time that ${\lcB}_z$ separates $z$ from $w$ and let $\mathcal{O}_{z,w}$ be a Bernoulli random variable (taking values $\{0,1\}$ each with probability $1/2$) that is independent of $\{({\lcB}_z)_t \, ; \, t\in [0,\sigma_{z,w}]\}$. 

We define elements $${\lcTB}^\eps_{z,w}=\{ ({\lcB}^\eps_{z})_{t\wedge \sigma_{z,w}^\eps}\, ; \, t\ge 0\} \text{ and } {\lcTB}_{z,w}=\{ ({\lcB}_{z})_{t\wedge \sigma_{z,w}}\, ; \, t\ge 0\}$$ of $\lcS$. These are, respectively, the domain sequences formed by the $\SLE_\kp(\kp-6)$ and the uniform $\CLE_4$ exploration branches towards $z$, stopped when $z$ and $w$ become separated. By definition, they are parameterized such that $-\log \CR(0;({\lcTB}_{z,w}^\eps)_{t})=t\wedge \sigma_{z,w}^\eps$ for all $t$.
\begin{proposition}
	\label{prop:order_simple}
	Fix $z\ne w\in \cQ$. Then if $({\lcB},\cO)$ is a subsequential limit in law of $({\lcB}^\eps_z,\mathcal{O}^\eps_{z,w})$ (with respect to the product $\lcS_z$ $ \times$ discrete topology), $({\lcB},\cO)$ must satisfy the following property. If ${{\lcTB}}$ is equal to ${\lcB}$ stopped at the first time that $w$ is separated from $z$, then  $$({{\lcTB}},\cO) \overset{(law)}{=}({\lcTB}_{z,w},\mathcal{O}_{z,w}).$$
\end{proposition}

\noindent Note that this  \emph{does not} yet imply that the times at which $z$ and $w$ are separated converge.\\

To set up for the proof of this proposition, we define for $\eps,\delta>0$, ${\sigma}^\eps_{z,w,\delta}$ to be the first time $t$ that, under the conformal map $g_t[{\lcTB}_z^\eps]$, the image of $w$ is at distance $\delta$ from $\partial \D$. See Figure \ref{fig:almost_sep} for an illustration.  
Define $\sigma_{z,w,\delta}$ in the same way for $\eps=0$. Write ${\lcTB}_{z,w,\delta}^\eps$ and ${\lcTB}_{z,w,\delta}$ for the same things as ${\lcTB}_{z,w}^\eps$ and ${\lcTB}_{z,w}$, but with the time now cut off at $\sigma_{z,w,\delta}^\eps$ and $\sigma_{z,w,\delta}$ respectively. 

\begin{lemma}\label{lem:Dzw}
	\begin{enumerate}[(a)]
		\item $({\lcTB}_{z,w,\delta}^\eps,\sigma_{z,w,\delta}^\eps) \Rightarrow ({\lcTB}_{z,w,\delta},\sigma_{z,w,\delta})$ as $\eps\to 0$ for every fixed $\delta>0$. 
		\item $({\lcTB}_{z,w,\delta},\sigma_{z,w,\delta})\Rightarrow ({\lcTB}_{z,w},\sigma_{z,w})$ as $\delta\to 0$
	\end{enumerate}
\end{lemma}

\begin{proof}
	For (a) we use that ${\lcB}^\eps_{z}\Rightarrow {\lcB}_z$ in $\lcS_z$. Taking a coupling  $({\lcB}_z,({\lcB}^{\eps}_z)_{\eps>0})$ such that this convergence is almost sure, it is clear from the definition of convergence in $\lcS_z$ that, under this coupling, $({\lcTB}_{z,w,\delta}^\eps,\sigma_{z,w,\delta}^\eps)\to ({\lcTB}_{z,w,\delta},\sigma_{z,w,\delta})$ almost surely for every $\delta>0$.
	Statement (b) holds because $\sigma_{z,w,\delta}\to \sigma_{z,w}$ a.s.\ as $\delta\to 0$. Indeed, $\sigma_{z,w,\delta}$ is almost surely increasing in $\delta$ and bounded above by $\sigma_{z,w}$ so must have a limit $\sigma^*\le \sigma_{z,w}$ as $\delta\to 0$. On the other hand, $w$ cannot be mapped anywhere at positive distance from the boundary under $g_{\sigma^*}[{\lcB}_z]$, so it must be that $\sigma^*\ge \sigma_{z,w}$. 
\end{proof}\\

Thus we can reduce the proof of Proposition \ref{prop:order_simple} to the following lemma. 

\begin{lemma}\label{lem:cond_exp_ssl}
	For any continuous bounded function $F$ with respect to $\lcS_z$, and any fixed $\delta>0$, we have that
	\[  \mathbb{E}[\cO^\eps_{z,w} F({\lcTB}^\eps_{z,w,\delta})] \to  \frac{1}{2} \mathbb{E}[F({\lcTB}_{z,w,\delta})]\]
	as $\eps\to 0$.
\end{lemma}

\begin{proofof}{Proposition \ref{prop:order_simple} given Lemma \ref{lem:cond_exp_ssl}}
	Consider a subsequential limit as in Proposition \ref{prop:order_simple}. Write $\wt{{\lcTB}}_\delta$ for ${\lcB}$ stopped at the first time that $w$ is sent within distance $\delta$ of $\partial \D$ under the Loewner flow. Then it is clear (by taking a coupling where the convergence holds a.s.) that $(\wt{{\lcTB}}_\delta, \cO)$ is equal to the limit in law of $({\lcTB}^\eps_{z,w,\delta}, \cO_{z,w}^\eps)$ as $\eps\to 0$ along the subsequence.
	
	On the other hand, Lemma \ref{lem:cond_exp_ssl} implies that the law of such a limit is that of ${\lcTB}_{z,w,\delta}$ together with an independent Bernoulli random variable. Indeed, any continuous bounded function with respect to the product topology on on $\lcS_z \times \{0,1\}$ is of the form $({\lcTB},x)\to \I_{\{x=1\}} F({\lcTB})+\I_{\{x=0\}} G({\lcTB})$ for $F,G$ bounded and continuous with respect to $\lcS_z$. Moreover, $\I_{\{x=0\}}G=G-\I_{\{x=1\}}G$ and we already know that $\mathbb{E}[G({\lcTB}_{z,w,\delta}^\eps)] \to \mathbb{E}[G({\lcTB}_{z,w,\delta})]$ as $\eps \to 0$.
	
So $(\tilde{\lcTB}_\delta,\cO)$ has the law of $\lcB_{z,w,\delta}$ plus an independent Bernoulli random variable for each $\delta>0$. Combining with (b) of Lemma \ref{lem:Dzw} yields the proposition.
\end{proofof} \\

The proof of Lemma \ref{lem:cond_exp_ssl} will take up the remainder of this subsection. An important ingredient is the following result of \cite{KS17}, about the convergence of $\SLE_\kp$ to $\SLE_4$ as $\kp\downarrow 4$. 

\begin{theorem}[Theorem 1.10 of \cite{KS17}]\label{KS}
Chordal $\SLE_{\kappa'}$ between two boundary points in the disk converges in law  to chordal $\SLE_4$ as $\kp\downarrow 4$. This is with respect to supremum norm on curves viewed up to time reparameterization.
\end{theorem}

\begin{proofof}{Lemma \ref{lem:cond_exp_ssl}} 
	Since $F$ is bounded, subsequential limits of  $\mathbb{E}[\cO^\eps_{z,w} F({\lcTB}_{z,w,\delta}^\eps)]$ always exist. Therefore, we only need to show that such a limit must be equal to $(1/2) \mathbb{E}[F({\lcTB}_{z,w,\delta})]$. 
	
\begin{figure}
	\centering
	\includegraphics[width=\textwidth]{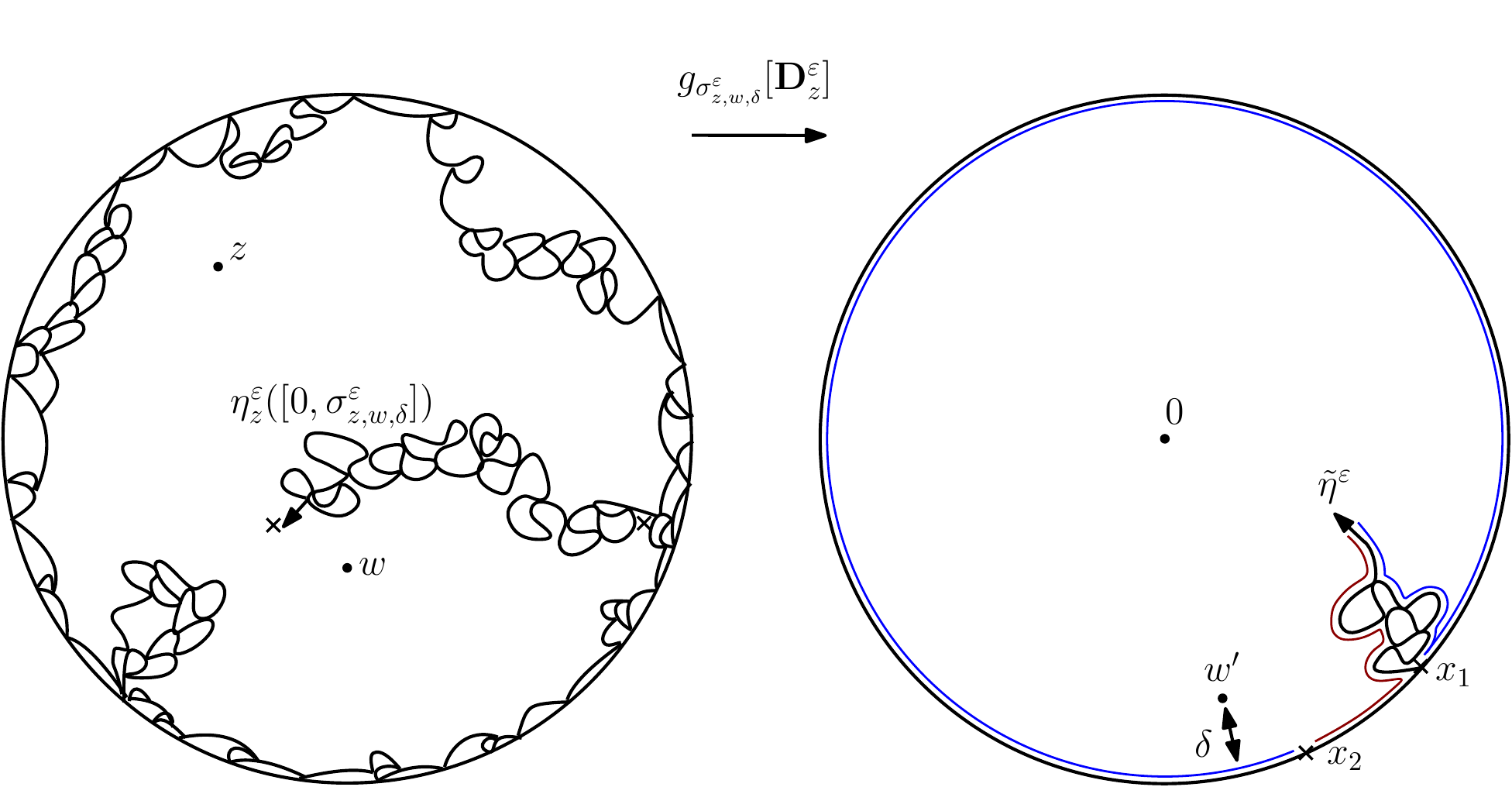}
	\caption{\emph{The SLE$_\kp(\kp-6)$ branch $\eta_z^\eps$, run up to time $\sigma_{z,w,\delta}^\eps$. This is the first time that under the Loewner map, $w$ is sent within distance $\delta$ of the boundary. The future of the curve has image  $\wt{\eta}^\eps$ under this map, and is an SLE$_\kp(\kp-6)$ starting from $x_1=\eta_z^\eps(\sigma_{z,w,\delta}^\eps)$  with a force point at $x_2\in \partial \D$. $z$ is visited before $w$ by the original space-filling $\SLE_\kp$ iff when $\wt{\eta}^\eps$ separates $0$ and $w'$ (the image of $w$), the component containing $0$ is ``monocolored''.}}
	\label{fig:almost_sep}
\end{figure}
	For this, we apply the map $g_{\sigma^\eps_{z,w,\delta}}[{\lcB}_z^\eps]$: recall that this is the unique conformal map from $({\lcB}_z^\eps)_{\sigma_{z,w,\delta}^\eps}$ to $\D$ that sends $z$ to $0$ and has positive real derivative at $z$, see Figure \ref{fig:almost_sep}. We then use the Markov property of $\SLE_{\kappa'}(\kp-6)$. This tells us that conditionally on ${\lcTB}^\eps_{z,w,\delta}$, the image of  $\eta^\eps_z$ under this map is that of an $\SLE_{\kappa'}(\kp-6)$ started at some $x_1\in \partial \D$ with a force point at $x_2\in \partial \D$ (where $x_1,x_2$ are measurable with respect to ${\lcTB}^\eps_{z,w,\delta}$). Let us call this curve $\wt{\eta}^\eps$. Let $w'$ be the image of $w$ under $g_{\sigma^\eps_{z,w,\delta}}[{\lcB}_z^\eps]$, which is also measurable with respect to $D^\eps_{z,w,\delta}$ and has $|w'|=1-\delta$ a.s.\
	Then the conditional expectation of $\cO^\eps_{z,w}$ given ${\lcTB}_{z,w,\delta}^\eps$  can be written as a probability for $\wt{\eta}^\eps$. Namely, it is just the probability that when $\wt{\eta}^\eps$ first separates $w'$ and $0$, the component containing $0$  either has boundary made up of entirely of the left hand side of $\wt{\eta}^\eps$ and the clockwise arc from $x_1$ to $x_2$, or the right hand side of $\wt{\eta}^\eps$ and the complementary counterclockwise arc. 
	We denote this event for $\wt{\eta}^\eps$ by $\cA^\eps$.  
	
	Therefore, by dominated convergence, Lemma \ref{lem:cond_exp_ssl} follows from Lemma \ref{lem:1/2} stated and proved below.
\end{proofof}

\begin{lemma}\label{lem:1/2}
Let $\wt{\eta}^\eps$ be	 an $\SLE_{\kappa'}(\kp-6)$ started at some $x_1\in \partial \D$ with a force point at $x_2\in \partial \D$.  
Fix $w'\in\D$. Let $\cA^\eps$   be the  event that when $\wt{\eta}^\eps$ first separates $w'$ and $0$, the component containing $0$  either has boundary made up of entirely of the left hand side of $\wt{\eta}^\eps$ and the clockwise arc from $x_1$ to $x_2$, or the right hand side of $\wt{\eta}^\eps$ and the complementary counterclockwise arc.  \begin{equation}  \label{eq:1/2} \mathbb{P}(\cA^\eps) \to \frac{1}{2} \text{ as } \eps\to 0 \; \textrm{(equivalently as $\kp\downarrow 4$)}.\end{equation} 
\end{lemma}
Another way to describe the event $\cA^\eps$ is the following. If the clockwise boundary arc from $x_1$ to $x_2$ together with the left hand side of $\wt{\eta}^\eps$ is colored red, and the counterclockwise boundary arc together with the right hand side of $\wt{\eta}^\eps$ is colored blue (as in Figures \ref{fig:almost_sep} and \ref{fig:almost_sep_2}) then $\cA^\eps$ is the event that when $0$ and $w'$ are separated, the component containing $0$ is ``monocolored''. \\

\emph{Outline for the proof of Lemma \ref{lem:1/2}.}
Note that until the first time that $0$ is separated from $x_2$, $\wt{\eta}^\eps$ has the law (up to time reparameterization) of a chordal $\SLE_{\kappa'}$ from $x_1$ to $x_2$ in $\D$: see Lemma \ref{lem:radial_chordal}. Importantly, we know by Theorem \ref{KS} that this converges to chordal $\SLE_4$ as $\kp \downarrow 4$.
	
	This is the main ingredient going into the proof, for which the heuristic is as follows. If $\wt{\eta}^\eps$ is very close to a chordal SLE$_4$, then after some small initial time it should not hit the boundary of $\D$ again until getting very close to $x_2$. At this point either $w'$ and $0$ will be on the ``same side of the curve" (scenario on the right of Figure \ref{fig:almost_sep_2}) or they will be on ``different sides'' (scenario on the left of Figure \ref{fig:almost_sep_2}).
	
	\begin{itemize}
		\item In the latter case (left of Figure \ref{fig:almost_sep_2}), note that $\wt{\eta}$ is very unlikely to return anywhere near to $0$ or $w'$ before swallowing the force point at $x_2$. Hence, whether or not $\cA^\eps$ occurs depends only on whether the curve goes on to hit the boundary ``just to the left'' of $x_2$, or ``just to the right''. Indeed, hitting on one side will correspond to $0$ being in a monocolored red bubble when it is separated from $w'$, meaning that $\cA^\eps$ will occur, while hitting on the other side will correspond to $w'$ being in a monocolored blue bubble, and it will not. By the Markov property and symmetry, we will argue that each of these happen with (conditional) probability close to $1/2$.
		\item In the former case (right of Figure \ref{fig:almost_sep_2}), $\wt{\eta}$ will go on to swallow the force point $x_2$ before separating $0$ and $w'$, with high probability as $\kp\downarrow 4$. Once this has occurred, $\wt{\eta}^\eps$ will continue to evolve in the cut-off component containing $0$ and $w'$, as an $\SLE_\kp(\kp-6)$ with force point initially adjacent to  the tip. But then by mapping to the unit disk again, the conditional probability of $\cA^\eps$ becomes the probability that a space-filling $\SLE_\kp$ visits one particular point before another. This converges to $1/2$ as $\kp\downarrow 4$ by Lemma \ref{lem:pre-1/2}.
	\end{itemize}
	
	\begin{figure}
		\includegraphics[width=\textwidth]{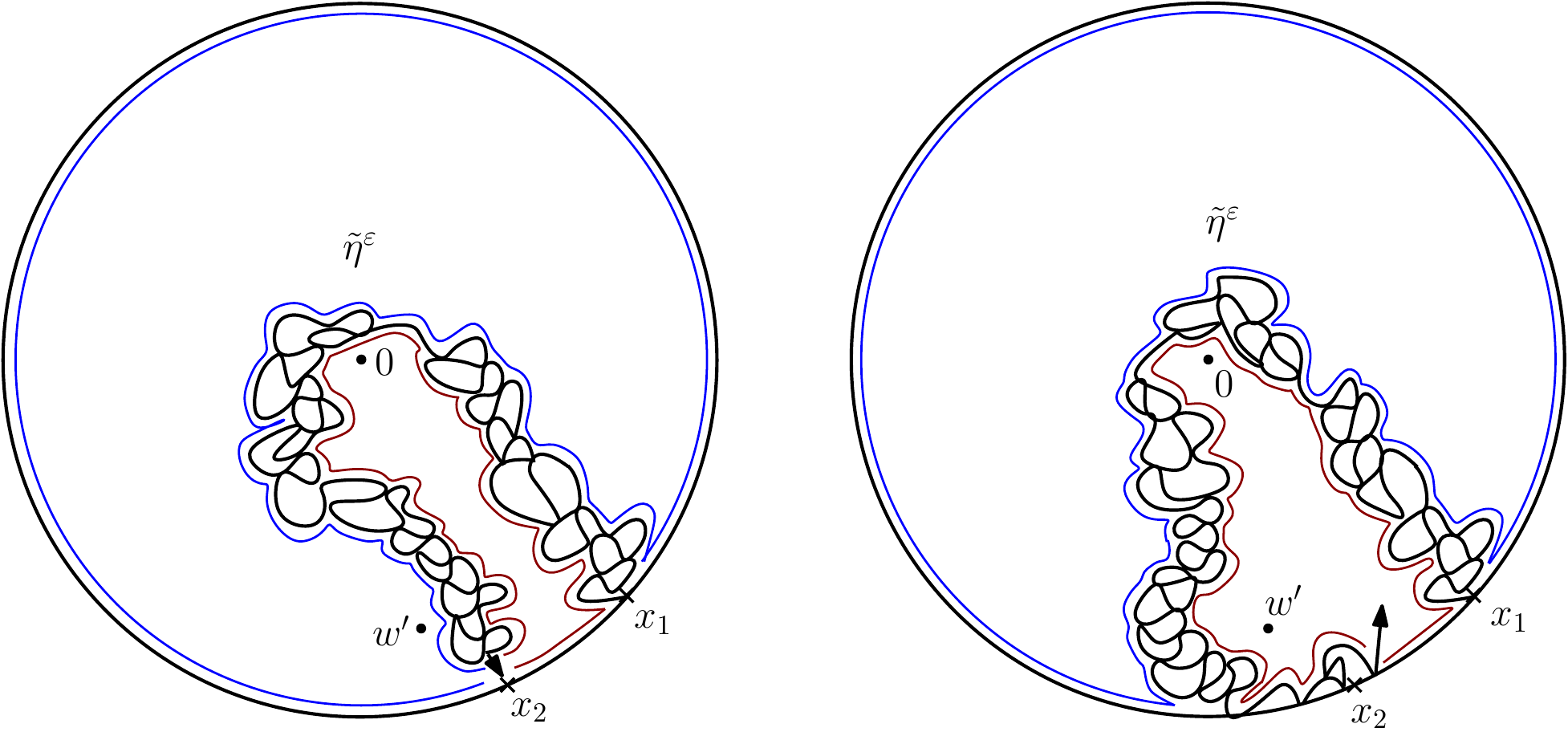}
		\caption{\emph{{Illustration of Lemma~\ref{lem:1/2}.} The two scenarios  that can occur when the force point $x_2$ is swallowed by $\wt{\eta}^\eps$. On the left, $0$ and $w'$ are on opposite sides of the curve (there is also an analogous scenario when $0$ is on the ``blue side'' and $w'$ is on the ``red side''). If this happens, we are interested whether $\wt{\eta}^\eps$ hits the blue or the red part of $\partial \D$ first. On the right, they are on the same side of the curve and we are interested in what happens after  $x_2$ is swallowed.}} \label{fig:almost_sep_2}
		\centering
	\end{figure}

	\begin{proofof}{Lemma \ref{lem:1/2}} Let us now proceed with the details. For $u>0$ small, let $\wt{\eta}^\eps_u$ be $\wt{\eta}^\eps$ run until the first entry time $T^\eps_u$ of $\D \cap B_{x_2}(u)$.  By Theorem \ref{KS}, the probability that $\wt{\eta}^\eps$ separates $0$ or $w'$ from $x_2$ before time $T^\eps_u$ tends to $0$ as $\eps\to 0$ for any fixed $u<|x_2-x_1|$. We write $E_{u,\text{b}}^\eps$ for this event. 
	
	We also {fix} a $u'>0$, chosen such that $x_1,0$ and $w'$ are contained in the closure of $\D \setminus B_{x_2}(u')$. Again from the convergence to SLE$_4$ we can deduce that
	
	\begin{equation}\label{eq:star} \mathbb{P}\left(\wt{\eta}^\eps \text{ revisits } \D \setminus B_{x_2}(u') \text{ after time }T^\eps_u \right)\to 0 \text{ as } u\to 0, \text{ uniformly in } \eps . \end{equation}
	The point of this is that $\wt{\eta}^\eps$ cannot ``change between the configurations in Figure \ref{fig:almost_sep_2}'' without going back into $\D\setminus B_{x_2}(u')$.
	Write:
	\begin{itemize} 
		\itemsep0em
		\item  $E_{u,\text{l}}^\eps$ for the intersection of $(E_{u,\text{b}}^\eps)^c$ and the event that \corr{$\wt{\eta}^\eps_u\cup \overline{B_{x_2}(u)}$} separates $0$ and $w'$ in $\D$, with $0$ on the left of $\wt{\eta}^\eps_u$;\item  $E_{u,\text{r}}^\eps$ for the same thing but with left replaced by right; \item  $E_{u,\text{s}}^\eps$ for the intersection of $(E_{u,\text{b}}^\eps)^c$ and the event that \corr{$\wt{\eta}_u^\eps\cup \overline{B_{x_2}(u)}$} does not separate $0$ and $w'$ in $\D$. 
	\end{itemize}
	Then we can decompose 
	\begin{align*}
		\mathbb{P}(\cA^\eps) & = &  \mathbb{E}[\mathbb{P}(\cA^\eps \, | \, E_{u,\text{b}}^\eps )\I_{E_{u,\text{b}}^\eps}+\mathbb{P}(\cA^\eps \, | \, E_{u,\text{l}}^\eps )\I_{E_{u,\text{l}}^\eps}+\mathbb{P}(\cA^\eps \, | \,E_{u,\text{r}}^\eps )\I_{E_{u,\text{r}}^\eps}+\mathbb{P}(\cA^\eps \, | \, E_{u,\text{s}}^\eps)\I_{E_{u,\text{s}}^\eps}] \nonumber \\
		& = & \underset{\textcircled{1}}{\mathbb{E}[{\cA}^\eps \I_{E_{u,\text{b}}^\eps}]} + \underset{\textcircled{2}}{\mathbb{E}[\mathbb{P}(\cA^\eps \, | \, E_{u,\text{l}}^\eps )\I_{E_{u,\text{l}}^\eps}]} + \underset{\textcircled{3}}{\mathbb{E}[\mathbb{P}(\cA^\eps \, | \,E_{u,\text{r}}^\eps )\I_{E_{u,\text{r}}^\eps}]} +\underset{\textcircled{4}}{\mathbb{E}[\mathbb{P}(\cA^\eps \, | \, E_{u,\text{s}}^\eps)\I_{E_{u,\text{s}}^\eps}]}
	\end{align*}
	By the observations of the previous paragraph, $\mathbb{P}(E_{u,\text{b}}^\eps)\to 0$ as $\eps\to 0$ for any fixed $u$, and therefore also \begin{equation} \label{1to 0}\textcircled{1} \to 0 \text{ as } \eps\to 0 \text{ for any fixed } u. \end{equation}
	
	Let us now describe what is going on with the terms $\textcircled{2},\textcircled{3}$ and $\textcircled{4}$. The term $\textcircled{2}$ corresponds to the left scenario of Figure \ref{fig:almost_sep_2}, and the term $\textcircled{3}$ corresponds to the same scenario, but when $0$ and $w'$ lie on opposite sides of the curve to those illustrated in the figure. We will show that: 
	\begin{equation}\label{left_scenario}
	 \lim_{u\to 0} \lim_{\eps\to 0} \, (\textcircled{2} + \textcircled{3}) =  \frac{1}{2} \mathbb{P}(\SLE_4 \text{ from } x_1 \text{ to } x_2 \text{ in } \D \text{ separates } w' \text{ and } 0)=: \frac{p}{2}
	\end{equation}
The term $\textcircled{4}$ corresponds to the scenario on the right of Figure \ref{fig:almost_sep_2}. We will show that: 
	\begin{equation}\label{right_scenario} \lim_{u\to 0}\lim_{\eps \to 0} \, \textcircled{4} = \frac{1}{2}(1-p)=\frac{1}{2}\mathbb{P}(\SLE_4 \text{ from } x_1 \text{ to } x_2 \text{ in } \D \text{ does not separate } w' \text{ and } 0).\end{equation} 
	Combining \eqref{left_scenario}, \eqref{right_scenario}, \eqref{1to 0} and the decomposition $\mathbb{P}(\cA^\eps)=\textcircled{1}+\textcircled{2}+\textcircled{3}+\textcircled{4}$ gives  \eqref{eq:1/2}, and thus completes the proof. So all that remains is to show \eqref{left_scenario} and \eqref{right_scenario}. 
	\smallskip
	
	\emph{Proof of \eqref{left_scenario}.} First, by \eqref{eq:star}, we can pick $u$ small enough such that the differences 
	\begin{align*} & \left(\textcircled{2} - \mathbb{E}[\mathbb{P}(\wt{\eta}^\eps\corr{|_{[T_u^\eps,\infty)}} \text{ hits the clockwise arc between } x_1 \text{ and } x_2 \text{ first }\, | \, E_{u,\text{l}}^\eps )\, \I_{E_{u,\text{l}}^\eps}] \right) \text{ and } \\ &
		\left(\textcircled{3} - \mathbb{E}[\mathbb{P}(\wt{\eta}^\eps\corr{|_{[T_u^\eps,\infty)}} \text{ hits the counterclockwise arc between } x_1 \text{ and } x_2 \text{ first } \, | \, E_{u,\text{r}}^\eps )\,\I_{E_{u,\text{r}}^\eps}]\right) \end{align*}
	are arbitrarily small, uniformly in $\eps$. All we are doing here is using the fact that if $u$ is small enough, $\wt\eta^\eps$ will not return anywhere close to $0$ or $w'$ after time $T_u^\eps$. This allows us to reduce the problem to estimating conditional probabilities for chordal $\SLE_\kp$.  To estimate these probabilities (the conditional probabilities in  the displayed equations above) we can use Theorem \ref{KS}, \corr{plus symmetry. In particular, Theorem \ref{KS} implies that for a chordal SLE$_\kp$ curve on $\HH$ from $0$ to $\infty$, the probability that it hits $[R,\infty)$ before $(-\infty,-L]$ for any fixed $L,R\in (0,\infty)$ can be made arbitrary close to the probability that it hits $[\max(L,R),\infty)$ before $(-\infty,-\max(L,R)]$ as $\kp\downarrow 4$. This is because SLE$_4$ does not hit the boundary apart from at the end points and the convergence is in the uniform topology. Since the probability that chordal SLE$_\kp$ in $\HH$ from $0$ to $\infty$ hits $[\max(L,R),\infty)$ before $(-\infty,-\max(L,R)]$ is $1/2$ for every $\kp$ by symmetry, we see that the probability of hitting $[R,\infty)$ before $(-\infty,-L]$ converges to $1/2$ as $\kp \downarrow 4$.}  
	
	We use this to observe, by conformally mapping to $\HH$ that  $$\mathbb{P}\left(\wt{\eta}^\eps\corr{|_{[T_u^\eps,\infty)}} \text{ hits the clockwise arc between } x_1 \text{ and } x_2 \text{ first  }  \, | \, \wt{\eta}^\eps([0,T_u^\eps])\right)\to \frac{1}{2} $$ almost surely as $ \eps\to 0$. Using this along with dominated convergence, we obtain \eqref{left_scenario}.
	\smallskip

	\emph{Proof of \eqref{right_scenario}.}
	Write $E^\eps$ for the event that $\wt{\eta}^\eps$ swallows the force point $x_2$ before separating $0$ and $w'$. Then we can rewrite \textcircled{4} as 
	\begin{equation}
		\mathbb{E}[\cA^\eps (\I_{E^\eps_{u,s}}-\I_{E^\eps})]+\mathbb{E}[\cA^\eps \I_{E^\eps}].
	\end{equation} 
	Applying \eqref{eq:star} shows that the first term tends to $0$ as $u\to 0$, uniformly in $\eps$. 
Let us now show that the second tends to $(1/2)(1-p)$ as $\eps\to 0$.

To do this, we condition on $\wt{\eta}^\eps$ run up to the time $T^\eps_0$ that the force point $x_2$ is swallowed. Conditioned on this initial segment we can use the Markov property of $\SLE_\kp(\kp-6)$ to describe the future evolution of $\wt{\eta}^\eps$. Indeed, it is simply that of a radial $\SLE_{\kappa'}(\kp-6)$ started from $\wt{\eta}^\eps(T_0^\eps)\in \partial \D$ and targeted towards $0$, with force point located infinitesimally close to the starting point. 
%Write $w''$ for the ($\wt{\eta}^\eps([0,T_0^\eps])$-measurable) image of $w'$ under the Loewner flow at time $T_0^\eps$. 
Viewing the evolution of $\wt{\eta}^\eps$ after time $T_0^\eps$ as one branch of a space-filling $\SLE_{\kappa'}$ we then have 
	\begin{equation*}\mathbb{E}[\cA^\eps \I_{E^\eps}]  =  \mathbb{E}[\mathbb{P}(\text{space-filling SLE}_\kp \text{ started from } \wt{\eta}^\eps(T_0^\eps) \text{ hits } 0 \text{ before } w') \I_{E^\eps}] \end{equation*} which we further decompose as 
	\[ \frac{1}{2} \mathbb{P}(E^\eps) + \mathbb{E}\left[\left(\mathbb{P}(\text{space-filling SLE}_\kp \text{ started from } \wt{\eta}^\eps(T_0^\eps) \text{ hits } 0 \text{ before } w')-1/2\right) \I_{E_\eps}\right]. \]
	Since the first term  above tends to $(1/2)(1-p)$ as $\eps\to 0$, it again suffices by dominated convergence (and by applying a rotation) to show that for any $x\in \D$:
	$$\mathbb{P}(\eta^\eps \text{ hits } 0 \text{ before } x)
	\to \frac{1}{2} \text{ as } \eps\to 0.$$ 
	
This is precisely the statement of Lemma \ref{lem:pre-1/2}.	Thus we conclude the proof of \eqref{right_scenario}, and therefore Lemma \ref{lem:1/2}.
\end{proofof}

\subsection{Convergence of separation times}
We now want to prove that for $z\ne w$ the \emph{actual separation times} $\sigma_{z,w}^\eps$ converge to the separation time $\sigma_{z,w}$ in law (jointly with the exploration) as $\eps\to 0$. The difficulty is as follows. Suppose we are on a probability space where $\eta^\eps_z$ converges a.s.\ to $\eta_z$. Then we can deduce (by Lemma \ref{lem:Dzw}) that any limit of $\sigma_{z,w}^\eps$ must be greater than or equal to $\sigma_{z,w}$.  But it still could be the case that $z$ and $w$ are ``almost separated'' at some sequence of times that converge to $\sigma_{z,w}$ as $\eps\downarrow 0$, but that the  $\eta_z^\eps$ then go on to do something else for a macroscopic amount of time before coming back to finally separate $z$ and $w$. Note that in this situation the $\eta_z^\eps$ would be creating ``bottlenecks'' at the almost-separation times, so it would not contradict Proposition \ref{prop:order_simple}).

The main result of this subsection is the following.
\begin{proposition}\label{prop:jointDsigma} For any $z\ne w \in \cQ$
	\begin{equation}
		\label{sep_time_conv}(\lcB_z^\eps,\sigma_{z,w}^\eps)\Rightarrow (\lcB_z,\sigma_{z,w})\end{equation} as $\eps\to 0$, with respect to \cart convergence in $\lcS$ in the first co-ordinate, and convergence in $\R$ in the second.
	
\end{proposition}

\begin{remark}\label{rmk:sep_time_tight}
	It is easy to see that $\sigma_{z,w}^\eps$ is tight in $\eps$ for any fixed $z\ne w\in \D$. For example, this follows from Corollary \ref{cor:cleloopconv}, which implies that minus the log conformal radius, seen from $z$, of the first $\CLE_{\kp}$ loop containing $z$ and not $w$, is tight. Since $\sigma_{z,w}^\eps$ is bounded above by this minus log conformal radius, tightness of $\sigma_{z,w}^\eps$ follows.
\end{remark}
There is one situation where convergence of the separation times is already easy to see from our work so far. Namely, when $z$ and $w$ are separated (in the limit) at a time when a CLE$_4$ loop has just been drawn. 
More precisely: 

\begin{lemma}\label{lem:septime_endofloop}
Suppose that $\eps_n\downarrow 0$ is such that $$(\lcB_{z}^{\eps_n},\lcB_w^{\eps_n}, 
\sigma_{z,w}^{\eps_n}, \sigma_{w,z}^{\eps_n},\cO_{z,w}^{\eps_n})\Rightarrow (\lcB_z, \lcB_w^*,
\sigma_{z,w}^*,\sigma_{w,z}^*,\cO^*) \text{ as } n\to \infty$$ (where at this point we know that $\lcB_z,\lcB_w^*$ have the same \emph{marginal} laws as $\lcB_z,\lcB_w$, but not necessarily the same joint law). Then on the event that $\lcB_z$ separates $w$ from $z$ at a time $\sigma_{z,w}$ when a $\CLE_4$ loop $\Loop$ is completed,  we have that almost surely: \begin{compactitem} \item $\sigma_{z,w}^*=\sigma_{z,w}$; \item $\lcB_w^*$ is equal to $\lcB_z$ (modulo time reparameterization), up to the time $\sigma_{w,z}$ that $z$ is separated from $w$; \item $\sigma_{w,z}^*=\sigma_{w,z}$; and \item 
conditionally on the above event occurring, $\cO^*$ is  independent of $\lcB_z,\lcB_w^*$ and has the law of a Bernoulli$(\frac{1}{2})$ random variable.
\end{compactitem} \end{lemma}

\begin{proof}  Without loss of generality, by switching the roles of $z$ and $w$ if necessary and by the Markov property of the explorations, it suffices to consider the case that $\cL=\cL_z$ is the outermost $\CLE_4$ loop (generated by $\lcB_z$) containing $z$. 
	
	By Skorokhod embedding together with Corollary \ref{rmk:convfullbranch} and Proposition \ref{prop:cleloopconv}, we may assume that we are working on a probability space where the convergence assumed in the lemma holds almost surely, jointly with the convergence $\Loop_z^{\eps_n} \to \Loop_z$ (in the Hausdorff sense), $\bub_z^{\eps_n}=(\lcB_{z}^{\eps_n})_{\tau_z^{\eps_n}}\to \bub_z=(\lcB_z)_{\tau_z}=\mathrm{int}(\Loop_z)$ (in the Carth\'{e}odory sense) and $(\tau_{0,z}^{\eps_n},\tau_z^{\eps_n}) \to (\tau_{0,z},\tau_z)$.  (Recall the definitions of these times from Section \ref{def:tauz}). We may also assume that the convergence $\sigma_{z,w,\delta}^{\eps_n}\to \sigma_{z,w,\delta}$ holds almost surely as $n\to \infty$ for all rational $\delta>0$.
	
	Now we restrict to the event $E$ that $\lcB_z$ separates $z$ from $w$ at time $\tau_z$, so that in particular 
	 $w$ is at positive distance from $\Loop_z\cup (\lcB_z)_{\tau_z}=\overline{(\lcB_z)_{\tau_z}}$. The Hausdorff convergence $\Loop_z^{\eps_n} \to \Loop_z$ thus implies that $w\in \D\setminus \bub_z^{\eps_n}$ for all $n$ large enough (i.e., $w$ is outside of the first $\CLE_{\kp(\eps_n)}$ loop containing $z$), and therefore that $\sigma_{z,w}^{\eps_n}\le\tau^{\eps_n}_z$ for all $n$ large enough (i.e., separation occurs no later than this loop closure time). Since $\sigma_{z,w}^*$ is defined to be the almost sure limit of $\sigma_{z,w}^{\eps_n}$ as $n\to \infty$, and we have assumed that $\tau_z^{\eps_n}\to \tau_z$ almost surely, this implies that $\sigma_{z,w}^*\le \tau_z$ almost surely on the event $E$. On the other hand, we know that $\sigma_{z,w}^{\eps_n}\ge \sigma_{z,w,\delta}^{\eps_n}$  and $\sigma_{z,w,\delta}^{\eps_n}\to \sigma_{z,w,\delta}$ as $n\to \infty$ for all rational positive $\delta$, so that $\sigma_{z,w}^*\ge \sigma_{z,w,\delta}$ for all $\delta$ and therefore $\sigma_{z,w}^*\ge \lim_{\delta\to } \sigma_{z,w,\delta}=\sigma_{z,w}=\tau_z$ almost surely.	Together this implies that $\sigma_{z,w}=\tau_z=\sigma_{z,w}^*$ on the event $E$.
	
	Next, observe that by the same argument as in the penultimate sentence above, we have $\sigma_{w,z}^*\ge \sigma_{w,z}$ with probability one. Moreover, we saw that on the event $E$, $w\in \D\setminus \bub_z^{\eps_n}$ for all $n$ large enough. But we also have that $\sigma_{z,w}^{\eps_n}\to \tau_z$, so that $\sigma_{z,w}^{\eps_n}>\tau_{0,z}^{\eps_n}$  and therefore $w\in (\lcB^{\eps_n}_{z,w})_{\tau_{0,z}^{\eps_n}}\setminus \bub_z^{\eps_n}$ for all $n$ large enough. Hence, $$ \sigma_{w,z}^*=\lim_n \sigma_{w,z}^{\eps_n} \le \lim_n -\log \CR(w,(\lcB^{\eps_n}_{z,w})_{\tau_{0,z}^{\eps_n}}\setminus \bub_z^{\eps_n})=-\log \CR(w,(\lcB_{z})_{\corr{\tau_{0,z}}}\setminus \bub_z)=\sigma_{w,z}.$$ Combining the two inequalities above gives the third bullet point of the lemma, and since  $\lcB_{w,z}^{\eps_n}$ and $\lcB_{z,w}^{\eps_n}$ agree up to time parameterization until $z$ and $w$ are separated for every $n$, we also obtain the second bullet point.

	For the final bullet point,  if we write $\lcB_{z,w}$ for $\lcB_z$ stopped at time $\sigma_{z,w}$, we already know from the previous subsection that the law of $\cO^*$ given  $\lcB_{z,w}$ is fair Bernoulli. Moreover, since $\cO^{\eps_n}_{z,w}$ and $(g_{\sigma_{z,w}^{\eps_n}}[\lcB_{z}^{\eps_n}]((\lcB_z^{\eps_n})_{s+\sigma_{z,w}^{\eps_n}}) \, ; \, s\ge 0)$ are independent for every $n$, it follows that $\cO^*$ is independent of $(g_{\sigma_{z,w}^*}[\lcB_{z}]((\lcB_z)_{s+\sigma_{z,w}^*})\, ; \,s\ge 0)$. So in general (i.e., without restricting to the event $E$) we can say that, given $(g_{\sigma_{z,w}^*}[\lcB_{z}]((\lcB_z)_{s+\sigma_{z,w}^*}) \, ; \, s\ge 0)$ and $((\lcB_z)_t \, ; \, t\le \sigma_{z,w})$, $\cO^*$ has the conditional law of a Bernoulli$(1/2)$ random variable. Since the event $E$ (that $\sigma_{z,w}=\tau_z$) is measurable with respect to $((\lcB_z)_t \, ; \, t\le \sigma_{z,w})$, and we have already seen that $\sigma_{z,w}=\sigma_{z,w}^*$ on this event, we deduce the final statement of the lemma.
\end{proof}\\

 \begin{proofof}{Proposition \ref{prop:jointDsigma}}
By tightness (Remark \ref{rmk:sep_time_tight}), and since we already know the convergence in law  of $(\lcB_z^{\eps}, (\sigma^\eps_{z,w,\delta})_{\delta>0})$ to $(\lcB_z,(\sigma_{z,w,\delta})_{\delta>0})$ , it suffices to prove that any joint subsequential limit in law of $(\lcB_z,(\sigma_{z,w,\delta})_{\delta>0},\sigma^*_{z,w}) \text{ of } (\lcB_z^{\eps}, (\sigma^\eps_{z,w,\delta})_{\delta>0}, \sigma_{z,w}^\eps)$ has $\sigma^*_{z,w}=\sigma_{z,w}$ almost surely. So let us assume that we have such a subsequential limit (along some sequence $\eps_n\downarrow 0$) and that we are working on a probability space where the convergence holds almost surely. As remarked previously, since $\sigma_{z,w}^{\eps_n}\ge \sigma_{z,w,\delta}^{\eps_n}$ for each $\delta>0$ and $\lim_\delta\lim_n\sigma_{z,w,\delta}^{\eps_n}=\lim_\delta \sigma_{z,w,\delta}=\sigma_{z,w}$, we already know that $\sigma_{z,w}^*\ge \sigma_{z,w}$ almost surely. So we just need to prove that 
$\mathbb{P}(\sigma_{z,w}+s\le\sigma_{z,w}^*)=0$, or alternatively, that $\lim_{\delta\to 0} \mathbb{P}(\sigma_{z,w,\delta}+s\le \sigma_{z,w}^*)=0$
for any $s>0$ fixed. Since $\sigma_{z,w,\delta}$ and $\sigma_{z,w}^*$ are the almost sure limits of $\sigma^{\eps_n}_{z,w,\delta}$ and $\sigma_{z,w}^{\eps_n}$ as $n\to \infty$, it is sufficient to prove that for each $s>0$
$$\limsup_{\delta\to 0}\limsup_{\eps\to 0} \mathbb{P}(\sigma_{z,w,\delta}^{\eps}+s\le \sigma_{z,w}^{\eps})=0.$$
The strategy of the proof is to use Lemma \ref{lem:septime_endofloop} to say that (when $\delta$ and $\eps$ are small), $\eta_z^\eps$ will separate lots of $\CLE_\kp$ loops from $z$ during the time interval $[\sigma_{z,w,\delta}^\eps,\sigma_{z,w,\delta}^\eps+s]$. Then we will argue that this is very unlikely to happen during the time interval $[\sigma_{z,w,\delta}^\eps,\sigma_{z,w}^\eps]$, which means that $\sigma_{z,w}^\eps<\sigma_{z,w,\delta}^\eps+s$ with high probability.

More precisely, let us assume from now on that $s>0$ is fixed, and write $\cS_r$ for the collection of faces (squares) of $r\Z^2$ that intersect $\D$. We write $\wt S_{\delta,r}^\eps$ for the event that  there exists $S\in \cS_r$ that is separated by $\eta_z^\eps$ from $z$ during the interval $[\sigma_{z,w,\delta}^\eps,\sigma_{z,w,\delta}^\eps +s]$ \emph{and} such that $z$ is visited by the space-filling $\SLE_\kp$ before $S$. We write $S_{\delta,r}^\eps$ for the same event but with the interval [$\sigma_{z,w,\delta}^\eps, \sigma_{z,w}^\eps]$ instead.
So if the event $\{\sigma_{z,w,\delta}^\eps+s\le \sigma_{z,w}^\eps\}$ occurs, then either $S_{\delta,r}^\eps$ occurs or $\wt S_{\delta,r}^\eps$ does not. Hence, for any $r>0$:
$$\limsup_{\delta\to 0}\limsup_{\eps\to 0} \mathbb{P}(\sigma_{z,w,\delta}^{\eps}+s\le \sigma_{z,w}^{\eps})\le \limsup_{\delta\to 0} \limsup_{\eps\to 0} \mathbb{P}(S_{\delta,r}^\eps)+ \limsup_{\delta\to 0}\limsup_{\eps\to 0}\mathbb{P}((\wt S_{\delta,r}^\eps)^c).$$ \corr{We will show that 
	\begin{equation} \label{eq:kexp}
	\liminf_{\delta\downarrow 0}\liminf_{\eps\downarrow 0}\mathbb{P}(\wt S_{\delta,r}^\eps)\to 1 \text{ as } r\to 0, 
	\end{equation}and that for any $r>0$,
\begin{equation}
\label{eqn:bottleneck}
\lim_{\delta\downarrow 0} \lim_{\eps\downarrow 0} \mathbb{P}(S_{\delta,r}^\eps) =0.
\end{equation} 

Let us start with \eqref{eq:kexp}. First, Lemma \ref{lem:septime_endofloop} tells us that since many $S\in \cS_r$ will be separated from $z$ by the CLE$_4$ exploration during the time interval $[\sigma_{z,w},\sigma_{z,w}+s]$ as $r\downarrow 0$, the same will be true for the space filling SLE$_\kp$ on the time interval $[\sigma_{z,w,\delta}^\eps, \sigma_{z,w,\delta}^\eps+s)$ when $\eps, \delta$ are small. More precisely, for any fixed $k\in \mathbb{N}$, $\delta>0$, the lemma implies that $$\liminf_{\eps\downarrow 0} \mathbb{P}(\eta_z^\eps([\sigma_{z,w,\delta}^\eps,\sigma_{z,w,\delta}^\eps+s]) \text{ separates } k \text{ squares in } \cS_r \text{ from } z) \ge p_{\delta,k,r} $$where $p_{\delta,k,r}$ is the probability that $\lcB_z$ disconnects at least $k$ squares in $\cS_r$ from $z$ by distinct $\CLE_4$ loops during the time interval $[\sigma_{z,w,\delta},\sigma_{z,w,\delta}+s]$. Moreover, since $\sigma_{z,w,\delta}\to \sigma_{z,w}$ as $\delta\to 0$ almost surely, $\liminf_{\delta\downarrow 0} p_{\delta,k,r}$ is equal to the probability $p_{k,r}$ that $\lcB_z$ disconnects at least $k$ squares in $\cS_r$ from $z$ by distinct $\CLE_4$ loops during the time interval $[\sigma_{z,w},\sigma_{z,w}+s]$. Note that since $s$ is positive (and fixed), $p_{k,r}\to 1$ as $r\to 0$ for any fixed $k$. 

This is almost exactly what we need. However, recall that although $\tilde{S}_{\delta,r}^\eps$ only requires one $S\in \cS_r$ to be disconnected from $z$ by $\eta_z^\eps([\sigma_{z,w,\delta}^\eps,\sigma_{z,w,\delta}^\eps+s])$, it also requires that this $z$ is visited by the space filling SLE$_\kp$ before $S$. This is why we ask for $k$ squares to be separated, because then by Lemma \ref{lem:septime_endofloop}, whether they are visited before or after $z$ converges to a sequence of independent coin tosses. Namely, for any $k\in \mathbb{N}$, \begin{align*} \liminf_{\delta\downarrow 0}\liminf_{\eps\downarrow 0}\mathbb{P}(\wt S_{\delta,r}^\eps) & \ge (1-2^{-k}) \liminf_{\delta\downarrow 0}\liminf_{\eps\downarrow 0} \mathbb{P}(\eta_z^\eps([\sigma_{z,w,\delta}^\eps,\sigma_{z,w,\delta}^\eps+s]) \text{ separates } k \text{ squares in } \cS_r \text{ from } z) \\
	& \ge (1-2^{-k})\liminf_{\delta\downarrow 0} p_{\delta,k,r} \\
	& \ge (1-2^{-k}) p_{k,r}.
\end{align*}
The $\liminf$ as $r\to 0$ of the left-hand side above is therefore greater than or equal to $(1-2^{-k})\lim_{r\to 0}p_{k,r}=(1-2^{-k})$ for every $k$. Since $k$ was arbitrary this concludes the proof of \eqref{eq:kexp}.

Hence, to conclude the proof of the proposition, it suffices to justify \eqref{eqn:bottleneck}.}
Although this is a statement purely about SLE, it turns out to be somewhat easier to prove using the connection with Liouville quantum gravity in \cite{DMS14}. Thus we postpone the proof of \eqref{eqn:bottleneck} to Section \ref{sec:order_proof}, at which point we will have introduced the necessary objects and stated the relevant theorem of \cite{DMS14}. Let us emphasise that this proof will rely only on \cite{DMS14} and basic properties of Liouville quantum gravity (and could be read immediately by someone already familiar with the theory) so it is safe from now on to treat Proposition \ref{prop:jointDsigma} as being proved. 
\end{proofof}\\

\subsection{Convergence of the partial order: proof of Proposition \ref{prop:convbranchingsleorder}}

Recall that Proposition \ref{prop:convbranchingsleorder}, stated at the very beginning of Section \ref{sec:conv_order}, asserts the joint convergence of the branching $\SLE_\kp$ and the collection of order variables to the limit $$((\lcB_z)_{z\in \cQ}, (\cO_{z,w})_{z,w\in \cQ})$$ defined in Lemma \ref{lem:defo}. Completing the proof is now simply a case of putting together our previous results.

  \begin{proofof}{Proposition \ref{prop:convbranchingsleorder}}
	The following three claims are the main ingredients.
\begin{itemize}
	\item[\textbf{Claim 1:}] $({\lcB}_z^\eps)_{z\in \cQ}\Rightarrow ({\lcB}_z)_{z\in \cQ}$. 
	
	\noindent \textit{Proof:} This follows from Corollary \ref{cor:convfullbranch}, Proposition \ref{prop:jointDsigma} and the fact that for every $\eps$ and $z,w\in \cQ$, ${\lcB}_z^\eps$ and ${\lcB}_w^\eps$ agree (up to time change) until $z$ and $w$ are separated, and then evolve independently. 
	
	\item[\textbf{Claim 2:}] For any $z, w\in \cQ$, \corr{$({\lcB}_z^\eps, {\lcB}_w^\eps, \cO_{z,w}^\eps)\Rightarrow ({\lcB}_z, {\lcB}_w, \cO_{z,w})$}.
	
	\noindent  \textit{Proof:}  \corr{As usual, due to tightness, it is enough to show that any subsequential limit $({\lcB}_z^*, {\lcB}_w^*, \cO^*)$ of $({\lcB}_z^\eps, {\lcB}_w^\eps, \cO_{z,w}^\eps)$, along a sequence $\eps_n\downarrow 0$, has the correct joint distribution. In fact, we may assume that $$({\lcB}_z^{\eps_n}, {\lcB}_w^{\eps_n}, \sigma_{z,w}^{\eps_n},\sigma_{w,z}^{\eps_n},\cO_{z,w}^{\eps_n})\Rightarrow ({\lcB}_z^*, {\lcB}_w^*,\sigma_{z,w}^*,\sigma_{w,z}^*, \cO^*)$$ and verify the same statement, where  by Proposition \ref{prop:jointDsigma} and Claim 1 above, we already know that $$(\lcB_z^*,\lcB_w^*,\sigma_{z,w}^*,\sigma_{w,z}^*)\overset{(d)}{=} (\lcB_z,\lcB_w,\sigma_{z,w},\sigma_{w,z})$$ (in particular, $\lcB_z^*$ and $\lcB_w^*$  agree up to time reparametrization until $z$ and $w$ are separated at times $\sigma_{z,w}^*$, $\sigma_{w,z}^*$). 
		
		Now, Proposition \ref{prop:order_simple} implies that, given  ${\lcB}_z^*$ and ${\lcB}_w^*$ stopped at times $\sigma_{z,w}^*,\sigma_{w,z}^*$ respectively, the conditional law of $\cO^*$ is fair Bernoulli. On the other hand, since $$\cO^{\eps_n}_{z,w}\, , \, (g_{\sigma_{z,w}^{\eps_n}}[\lcB_{z}^{\eps_n}]((\lcB_z^{\eps_n})_{s+\sigma_{z,w}^{\eps_n}}) \, ; \, s\ge 0) \text{ and } (g_{\sigma_{w,z}^{\eps_n}}[\lcB_{w}^{\eps_n}]((\lcB_w^{\eps_n})_{s+\sigma_{w,z}^{\eps_n}}) \, ; \, s\ge 0)$$ are mutually independent for every $n$, it follows that $\cO^*$ is independent of $$(g_{\sigma_{z,w}^*}[\lcB^*_{z}]((\lcB_z)_{s+\sigma_{z,w}^*})\, ; \,s\ge 0)\, , \, (g_{\sigma_{w,z}^*}[\lcB^*_{w}]((\lcB_w)_{s+\sigma_{w,z}^*})\, ; \,s\ge 0).$$ This provides the claim. }
%	This follows from Propositions \ref{prop:order_simple} and \ref{prop:jointDsigma}. Proposition \ref{prop:jointDsigma} is needed to verify that $\cO_{z,w}$ is also independent of the evolution of ${\lcB}_z$ after time $\sigma_{z,w}$ (which follows once we have Proposition \ref{prop:jointDsigma} since the same holds true at every level $\eps$ - see the proof of Lemma \ref{lem:septime_endofloop} for the detailed argument).
	
	\item[\textbf{Claim 3:}] For any $z,w\in \cQ$, $(({\lcB}_y^\eps)_{y\in \cQ}, \cO^\eps_{z,w})\Rightarrow (({\lcB}_y)_{y\in \cQ}, \cO_{z,w})$.
	
	\noindent  \textit{Proof:} \corr{The same argument as for Claim 2 above extends directly to this slightly more general setting (we omit the details).}
\end{itemize}
	
	With Claim 1 in hand (and the argument proving Lemma \ref{lem:defo}) all we need to show is that for any subsequential limit in law $(({\lcB}_z)_{z\in \cQ}, (\cO^*_{z,w})_{z,w\in \cQ})$ of $(({\lcB}_z^\eps)_{z\in \cQ}, (\cO^\eps_{z,w})_{z,w\in \cQ})$  as $\eps\to 0$, the \emph{conditional law} of $(\cO^*_{z,w})_{z,w\in \cQ}$ given $({\lcB}_z)_{z\in \cQ}$ satisfies the bullet points above Lemma \ref{lem:defo}. 
	That is: (a) $\cO^*_{z,z}=1$ for all $z\in \cQ$; (b) $\cO^*_{z,w}=1-\cO^*_{w,z}$ for all $z,w\in \cQ$ distinct; (c) $\cO^*_{z,w}$ is (conditionally) Bernoulli$(1/2)$ for any such $z,w$; and (d) for all $z,w_1,w_2\in \cQ$ with $z\ne w_1, w_2$, if ${\lcB}_z$ separates $z$ from $w_1$ at the same time as it separates $z$ from $w_1$ then $\cO^*_{z,w_1}=\cO^*_{z,w_2}$; otherwise $\cO^*_{z,w_1}$ and $\cO^*_{z,w_2}$ are (conditionally) independent. 
	
	Observe that (a) and (b) follow by definition of the $\cO_{z,w}^\eps$, and (c) follows from Claim 3 above. The first case of (d) also follows by definition, and the second follows from the definition of $\cO_{z,w_1}^\eps, \cO_{z,w_2}^\eps$ together with the branching property of $({\lcB}_z^\eps)_{z\in \cQ}$ and the convergence of the separation times.
\end{proofof}

\subsection{Joint convergence of SLE, CLE and the order variables}

The results of Sections \ref{sec:conv_clesle} and \ref{sec:conv_order} give the final combined result: 
\begin{proposition}
	\begin{eqnarray*} & (({\lcB}^\eps_z)_{z\in \cQ},(\cL^\eps_{z,i})_{z\in \cQ, i\ge 1}, (\bub^\eps_{z,i})_{z\in \cQ, i\ge 1},(\mathcal{O}^\eps_{z,w})_{z,w\in \cQ} ) & \\ & \Rightarrow & \\ & (({\lcB}_z)_{z\in \cQ},(\cL_{z,i})_{z\in \cQ, i\ge 1}, (\bub_{z,i})_{z\in \cQ, i\ge 1},(\mathcal{O}_{z,w})_{z,w\in \cQ} )& \end{eqnarray*}
	as $\eps\downarrow 0$, with respect to the product topology $$\prod_\cQ \lcS_z \times \prod_{\cQ\times \N} \text{Hausdorff} \times \prod_{\cQ\times \N} \text{\cart viewed from } z \times \prod_{\cQ\times\cQ} \text{discrete}.$$
	\label{prop:cle-conv}
\end{proposition}

\begin{proof}
	Since we know that all the individual elements in the above tuples converge, the laws are tight in $\eps$. Combining Proposition \ref{prop:convbranchingsleorder} and Corollary \ref{cor:cleloopconv} \corr{(in particular, using that $(\cL_{z,i})_{z\in \cQ, i\ge 1}, (\bub_{z,i})_{z\in \cQ, i\ge 1}$ are deterministic functions of $({\lcB}_z)_{z\in \cQ}$)} ensures that any subsequential limit has the correct law.
\end{proof}

\section{Liouville quantum gravity and mating of trees}

\subsection{Liouville quantum gravity}\label{sec:LQG}

Let $D\subset \C$ be a {simply connected} domain with harmonically non-trivial boundary. For $f,g\in C^\infty(D)$ define the Dirichlet inner product by 
\eqbn
(f,g)_\nabla = \frac{1}{2\pi} \int_D \nabla f(z) \cdot \nabla g(z)\, d^2\hspace{-0.05cm}z.
\eqen
Let $H(D)$ be the Hilbert space closure of the subspace of functions $f\in C^\infty(D)$ for which $(f,f)_\nabla<\infty$, where we identify two functions that differ by a constant. Letting $(f_n)$ be an orthonormal basis for $H(D)$, the free boundary GFF $h$ on $D$ is defined by 
\eqbn
h = \sum_{n=1}^{\infty}
\alpha_n f_n,
\eqen
where $(\alpha_n)$ is a sequence of i.i.d.\ standard normal random variables \corr{and the convergence is almost sure in the space of generalized functions modulo constants}. The free boundary GFF is only defined modulo additive constant here, but we remark that there are several natural ways to fix the additive constant\corr{, for example, by requiring that testing the field against a fixed test function gives zero. If this is done in an arbitrary way (i.e., picking some arbitrary test function in the previous sentence) the resulting field almost surely lives in the space $H^{-1}_{\text{loc}}(D)$: this is the space of generalized functions whose restriction to any bounded \nina{domain} $U\subset D$ is an element of the Sobolev space $H^{-1}(U)$.} See \cite{SheGFF,Ber16} for more details.

Let $\cS=\R\times(0,\pi)$ denote the infinite strip. By, for example, \cite[Lemma 4.3]{DMS14}, $H(\cS)$ has an orthogonal decomposition $H(\cS)=H_1(\cS)\oplus H_2(\cS)$, where $H_1(\cS)$ is the subspace of $H(\cS)$ consisting of functions (modulo constants) which are constant on vertical lines of the form $u+[0,\im \pi]$ and $H_1(\cS)$ is the subspace of $H(\cS)$ consisting of functions which have mean zero on all such vertical lines. This leads to a decomposition $h=h_1+h_2$ of the free boundary GFF $h$ on $\cS$, where $h_1$ (resp.\ $h_2$) is the projection of $h$ onto $H_1(\cS)$ (resp.\ $H_2(\cS)$). We call $h_2$ the \emph{lateral component} of $h$.

Now let $D\subset\C$ be as before, and let $\frk h$ be an instance of the free-boundary Gaussian free field (GFF) on $D$ with the additive constant fixed in an arbitrary way. Set $h=\frk h+f$, where $f$ is a (possibly random) continuous function on $D$.  For $\delta>0$ and $z\in D$ let $h_\delta(z)$ denote the average of $h$ on the circle $\partial B_\delta(z)$ if $B_\delta(z)\subset D$; otherwise set $h_\delta(z)=0$. For $\gamma\in (\sqrt{2},2)$ and $\eps=2-\gamma$ the field $h$ induces an area measure $\mu_h^\eps$ on $D$, which is defined by the following limit in probability for any bounded open set $A\subseteq D$:
$$
\mu_h^\eps(A)  = \lim_{\delta\rta 0} (2\eps)^{-1}\int_A \exp\Big(\gamma h_\delta(z)\Big)\delta^{\gamma^2/2} \, d^2\hspace{-0.05cm}z.
$$
Note that the definitions for $\eps>0$ differ by a factor of $2\eps$ from the definitions normally found in the literature. This is natural in the context of this paper, where we will be concerned with taking $\eps\downarrow 0$ (see below). Indeed, for $\gamma=2$ (which will correspond to the limit as $\eps\downarrow 0$) we define:
$$
\mu_{h}(A) = \lim_{\delta\rta 0} \int_A \Big(-h_\delta +\log(1/\delta)\Big)\exp(2h_\delta(z))\delta \, d^2\hspace{-0.05cm}z.$$

If $f$ extends continuously to $\partial D$, boundary measures $\nu^\eps_h$ and $\nu_h$ can be defined similarly by
\begin{align*}
	\nu_{h}^{\eps}(A) & = \lim_{\delta\rta 0}\, (2\eps)^{-1}\int_A \exp\Big(\frac{\gamma}{2}h_\delta(z)\Big)\delta^{\gamma^2/4} \, dz,\\
	\nu_{h}(A) & = \lim_{\delta\rta 0}\, \int_A \Big(-\frac{h_\delta}{2} +\log(1/\delta)\Big)\, \delta \, \exp(h_\delta(z)) \, dz.
\end{align*}
See \cite{DS11, Ber17, Pow18chaos} for proofs of these facts.

A pair $(D,h)$ 
defines a so-called \emph{$\gamma$-Liouville quantum gravity (LQG) surface}. More precisely, a $\gamma$-LQG surface is an equivalence class of pairs $(D,h)$ where {$D$ is as above and $h$ is a distribution, and} we define two pairs $(D_1,h_1)$ and $(D_2,h_2)$ to be equivalent if there is a conformal map $\phi:D_1\to D_2$ such that
\begin{equation}\label{eq:coc}
	h_1 = h_2\circ\phi+Q_\gamma \log|\phi'|,\qquad Q_\gamma:=2/\gamma+\gamma/2.
\end{equation}
With this definition, {if $h_1,h_2$ are absolutely continuous with respect to a GFF plus a continuous function we have} $\mu_{h_2}^\eps=\phi_*(\mu_{h_1}^\eps)$ and $\nu_{h_2}^\eps=\phi_*(\nu_{h_1}^\eps)$ for $\eps\in (0,2-\sqrt{2})$. 
The analogous identities also hold for $\eps=0$. 

The \emph{LQG disk} is an LQG surface of special interest, since it arises in scaling limit results concerning random planar maps, for example, \cite{BMdisk,GMdisk}.
The following is our definition of the unit boundary length $\gamma$-LQG disk in the subcritical case. Our field is equal to $-2\gamma^{-1}\log(2\eps)$ plus the field defined in, e.g.\ \cite{DMS14}: this is because we want it to have boundary length $1$ for our definition of $\nu_h^\eps$ (which is $(2\eps)^{-1}$ times the usual one).

\begin{definition}[unit boundary length $\gamma$-LQG disk for $\gamma\in (\sqrt{2},2)$]	\label{def:disk}
	Let $h_2$ be a field on the strip $\cS=\R\times  (0,\im \pi)$ with the law of the lateral component of a free boundary GFF on $\cS$.  
	Let $h_1^\eps$ be a function on $\cS$ such that $h^\eps_1(s+\im y)=\cB^\eps_s$, where:
	\begin{itemize}
		\item[(i)] $(\cB^\eps_s)_{s\geq 0}$ has the law of $B_{2s}-(2/\gamma-\gamma/2)s$ conditioned to be negative for all time, for $B$ a standard Brownian motion started from $0$;
		\item[(ii)] $(\cB^\eps_{-s})_{s\geq 0}$ is independent of $(\cB^\eps_s)_{s\geq 0}$ and satisfies $(\cB^\eps_{-s})_{s\geq 0}\eqD(\cB^\eps_s)_{s\geq 0}$.
	\end{itemize}
	Set $h_{\op s}^\eps=h_1^\eps+h_2$ and let ${\wh h^\eps}$ be the distribution on $\cS$ whose law is given by
	\eqb
	h_{\op s}^\eps-2\gamma^{-1}\log \nu^\eps_{h_{\op s}^\eps}(\partial \cS) 
	\qquad
	\text{reweighted\,\,by\,\,}
	\nu^\eps_{h^\eps_{\op s}}(\partial \cS)^{4/\gamma^2-1}. 
	\label{eq:reweight}
	\eqe
	Then the surface defined by $(\cS,{\wh h^\eps})$ has the law of a unit boundary length $\gamma$-LQG disk.
\end{definition}

See \cite[Definition 2.4 and Remark 2.5]{HS19} for a proof that the above does correspond to $-2\gamma^{-1}\log(2\eps)$ + the unit boundary length disk of \cite{DMS14}. Note that (see for example \cite[Lemma 4.20]{DMS14}) $\nu_{h_{\op s}}^\eps(\partial \cS)$ is finite for each fixed $\eps>0$, so that the above definition makes sense. In fact, we can say something stronger, {namely Lemma \ref{lem:tail} just below.  We remark that the power $1/17$ in the lemma has not been optimized.}

\begin{lemma}\label{lem:tail}
	There exists $C\in (0,\infty)$ not depending on $\eps\in (0,2-\sqrt{2})$ such that  $$\BB P[ \nu_{h_{\op s}^\eps}^\eps(\partial \cS)>x]\le Cx^{-1/17}  \text{ for all } x\ge 1.$$ \corr{Moreover, for any fixed $x$, $\BB P[ \nu_{h_{\op s}^\eps}^\eps((-\infty,-K)\cup(K\cup \infty)\times \im \{0,\pi\})>x]\to 0$ as $K\to \infty$, uniformly in $\eps$.} 	
	
	\corr{Finally,} if $h_{\op s}$ is defined in the same way as $h_{\op s}^\eps$ above but instead letting $(\cB_s)_{s\geq 0}$ have the law of $(-\sqrt{2})$ times a 3-dimensional Bessel process, then we also have that
	$$\BB P[ \nu_{h_{\op s}}(\partial \cS)>x]\le Cx^{-1/17} \text{ for all } x\ge 1.$$
	
\end{lemma}

\begin{proof} 
	Let us first deal with the subcritical measures. In this case, we write
	$$b_k^\eps=\nu_{h_2}^\eps([k,k+1]\times \{0,\im\pi\})$$ for $k\in \Z$. Then the law of $b_k^\eps$ does not depend on $k$ since the law of $h_2$ is translation invariant, see for example \cite[Remark 5.48]{Ber16}. Furthermore, by \cite[Theorem 1.1]{REMY}, $\mathbb{E}((b_0^\eps)^q)$ is uniformly bounded in $\eps$ for any $q<1$. (The result of \cite{REMY} shows uniform boundedness of the moment for a field that differs from $h_2$ in $[0,1]\times \{0\}$ or $[0,1]\times \{\im \pi\}$ by a centered Gaussian function with uniformly bounded variance.) Letting $a_k^\eps=\sup_{s\in[k,k+1]} e^{(\gamma/2) \cB^\eps_s}$ we then have that $$\nu_{h_{\op s}^\eps}^\eps(\partial \cS)\le \sum_{k\in \Z} a_k^\eps b_k^\eps.$$ Thus, since $\sum_{k\in\Z} (|k|\vee 1)^{-2}<10$, a union bound gives
	\eqb
	\begin{split}
		\BB P[ \nu_{h_{\op s}^\eps}^\eps(\partial \cS)>x]	
		\leq \sum_{k\in\Z} \left(
		\BB P[ a_k^\eps > x^{1/2}(|k|\vee 1)^{-4} ]+
		\BB P[ b_k^\eps > 0.1x^{1/2}(|k|\vee 1)^{2} ]\right).
	\end{split}
	\label{eq18}
	\eqe
	Taking $q=3/4$ (for example), using the uniform bound on $\mathbb{E}((b_k^\eps)^q)$ and applying Chebyshev's inequality gives that $ \sum_{k\in \Z}	\BB P[ b_k^\eps > 0.1x^{1/2}(|k|\vee 1)^{2} ] \le c_0 x^{-3/8}$ for some universal constant $c_0$.
	Furthermore, since $\cB^\eps$ is stochastically dominated by $(-\sqrt{2})$ times a three dimensional Bessel process, see \cite[Lemma 12.4]{KRV18}, we have that for $(Z(t))_{t\geq 0}$ such a process and $(W(t))_{t\geq 0}$ a standard linear Brownian motion:
	\eqbn
	\begin{split}
		\BB P[ a_k^\eps > x^{1/2}(|k|\vee 1)^{-4} ] 
		&\leq \BB P\Big[ \inf_{s\in[k,k+1]} Z(s)
		< \gamma^{-1} \log\big(x^{-1/2}(|k|\vee 1)^{4}\big)
		\Big]\\
		&\leq \BB P\Big[ \inf_{s\in[k,k+1]} |W(s)|
		< \gamma^{-1} \log\big(x^{-1/2}(|k|\vee 1)^{4}\big)
		\Big]^3
	\end{split}
	\eqen
	for all $x$ and $k${, where we used to get the second inequality that $Z\eqD|(W_1,W_2,W_3)|$ for $W_1,W_2,W_3$ independent copies of $W$.}
	The probability on the right side is 0 if $|k|\leq x^{1/8}$ and otherwise it is bounded above by $c_1|k|^{-1/2} \gamma^{-1} \log\big(x^{-1/2}(|k|\vee 1)^{4}\big)$ where $c_1$ is another universal constant. Therefore, for a final universal constant $c_2>0$,
	$$
	\sum_{k\in\Z}\BB P[ a_k^\eps > x^{1/2}(|k|\vee 1)^{-4} ] 
	\leq 2\sum_{k\in\Z\,:\,|k|>x^{1/8}} \Big(c_1|k|^{-1/2} \gamma^{-1} \log\big(x^{-1/2}(|k|\vee 1)^{4}\big)\Big)^3
	\leq c_2 x^{-1/17}.$$
	\corr{The same bounds yield the second statement of the lemma}.
	
\corr{Finally,} exactly the same proof works in the case of the critical measure, using \cite[Section 1.1.1]{REMY} to see that $b_k=\nu_{h_2}([k,k+1])$ has a finite $q$th moment, that does not depend on $k$ by translation invariance.
\end{proof}\\

We may now define the critical unit boundary length LQG disk as follows.
\begin{definition}[unit boundary length 2-LQG disk]\label{def:critical-disk}
	Letting $h_{\op s}$ be as in Lemma \ref{lem:tail} we define the unit boundary length 2-LQG disk to be the surface $(\cS, \wh h)$, where
	$$\wh h:=h_{\op s}-\log \nu_{h_{\op s}}(\partial \cS).$$
\end{definition}

{Note that $\nu_{h_{\op s}}(\partial S)$ is finite by Lemma \ref{lem:tail}.}

\begin{remark}
	Readers may have previously encountered the above as the definition of a quantum disk with \emph{two marked boundary points}. A quantum surface with $k$ marked points is 
	an equivalence class of $(D,h,x_1,\dots, x_k)$ with $x_1,\dots, x_k\in \overline{D}$, using  the  equivalence relation described by \eqref{eq:coc}, but with the additional requirement that  $\phi$ maps marked points to marked points. In this paper we will use Definitions \ref{def:disk} and \ref{def:critical-disk} to define specific equivalence class representatives of quantum disks, but we will always consider them as quantum surfaces without any marked points. That is, we will consider their equivalence classes under the simple relation \eqref{eq:coc}.
\end{remark}

The following lemma says that the subcritical disk converges to the critical disk as $\eps\downarrow 0$ (equivalently, $\gamma\uparrow 2$). {We say that a sequence of measures $(\bar\mu_n)_{n\in\N}$ on a metric space $E$ (equipped with the Borel $\sigma$-algebra) converges weakly to a measure $\bar\mu$ if for all $A\subseteq E$ such that $\bar\mu(\partial A)=0$ we have $\bar\mu_n(A)\rta \bar\mu(A)$.}

\begin{lemma}\label{lem:disk-conv} 
	For $\eps>0$ let $\wh{h}^\eps$ be the field of Definition \ref{def:disk} and $\wh{h}$ be the field of Definition \ref{def:critical-disk}. Then
	$(\wh{h}^\eps,\mu^\eps_{\wh{h}^\eps},\nu^\eps_{\wh{h}^\eps})\Rightarrow (\wh{h},\mu_{\wh{h}},\nu_{\wh{h}})$, where the first coordinate is equipped with the \corr{$H^{-1}_{\mathrm{loc}}(\cS)$} topology and the second and third coordinates are equipped with the weak topology of measures on $\cS$ and $\partial \cS$ respectively. 
	\label{prop:disk-conv}
\end{lemma}

\begin{proof}
	To conclude it is sufficient to prove the following, for an arbitrary sequence $\eps_n\downarrow 0$:
	\begin{enumerate}[(i)]
		\item  we have convergence in law along the sequence $\epn$ if we replace $\wh h$ by $h_{\op s}$, and $\wh h^\epn$ by $h_{\op s}^\epn$ for every $n$; and 
		\item  there exists a coupling of the $(\nu_{h^\epn_{\op s}})$ such that $\nu_{h^\eps_{\op s}}^\epn(\partial \cS)^{4/\gamma^2-1}\rta 1$ in $L^1$ as $n\to \infty$.
	\end{enumerate}
	To see (i), first observe that the processes $\cB^\eps$ converge to $\cB$ in law as $\eps\to 0$, with respect to the topology of uniform convergence on compacts of time. \corr{Indeed for any fixed $\delta>0$, if $T_\delta^\eps$ (resp.\ $T_\delta$) is the first time that $\cB^\eps$ (resp.\ $\cB$) hits $-\delta$, it is easy to see that $\cB^\eps(\cdot+T_\delta^\eps)$ converges to $\cB(\cdot+T_\delta)$ in law in the specified topology as $\eps\to 0$: a consequence of the fact that the drift coefficient in $\cB^\eps$ goes to $0$, and by applying the Markov property at time $T_\delta^\eps, T_\delta$. Moreover, $T_\delta, T_\delta^\eps$ converge to $0$ in probability as $\delta\to 0$, uniformly in $\eps$: this is true since $T_\delta,T_\delta^\eps$ are stochastically dominated by their counterparts for non-conditioned (drifted) Brownian motion, and the result plainly holds for the non-conditioned versions. Combining these observations yields the assertion.}
	
	We may therefore couple ${h_{\op s}^\epn}$ and $h_{\op s}$ so that their lateral components are identical, and the components that are constant on vertical lines converge a.s.\ on compacts as $n\to \infty$. For this coupling, the result of \cite{APS18two} implies that \eqb\nu^\epn_{h_{\op s}^\epn}(A)\to \nu_{h_{\op s}}(A) \text{ and } \mu^\epn_{h_{\op s}^\epn}(U)\to \mu_{h_{\op s}}(U)\label{eq:aps18}\eqe in probability as $n\to \infty$, for any bounded subsets $A\subset \partial \cS$ and $U\subset \cS$.  More precisely \cite[Section 4.1.1-4.1.2]{APS18two} proves that $\nu_{h}^\epn(A)\to \nu_h^\epn(A)$, when $h$ is a specific field on $\cS$ that differs from $h_{\op s}$ by a bounded continuous function on $A$ (similarly for $\mu$). Since adding a continuous function $f$ to $h$ modifies the boundary measure locally by $\exp((\gamma/2)f)$ and the bulk measure by $\exp(\gamma f)$ we deduce \eqref{eq:aps18}. To conclude that $$(h_{\op s}^\epn,\nu^\epn_{h_{\op s}^\epn}, \mu^\epn_{h_{\op s}^\epn})\to (h_{\op s},\nu_{h_{\op s}},\mu_{h_{\op s}})$$ in probability for this coupling (with the correct topology), and thus complete the proof of (i), it remains to show that $\nu_{h_{\op s}^\epn}^n(\partial \cS)\to \nu_{h_{\op s}}(\partial \cS)$ and $\mu_{h_{\op s}^\epn}^n(\cS)\to \mu_{h_{\op s}}( \cS)$ in probability as $n\to \infty$. For this, %we note that the proof of
	\corr{we use the second assertion of  Lemma \ref{lem:tail}} %gives the convergence in probability of $\nu_{h_{\op s}^\epn}^{\epn}((-\infty,-K)\cup (K,\infty)\times \im \{0,\pi\})$ to $0$ as $K\to \infty$, \emph{uniformly} in $n$, while 
	\corr{together with the fact that } $\nu_{h_{\op s}}(\cS)=\lim_{K\to \infty} \nu_{h_{\op s}}((-K,K)\times \im \{0,\pi\})$ by definition. Combining with \eqref{eq:aps18} yields the desired conclusion for the boundary measures. A similar argument can be applied for the bulk measures, where we may use, for example \cite[Theorem 1.2]{AG19} or \cite[Theorem 1.2]{ARS} to get the uniform $q$th moment bound for $q<1$ as in the proof of \ref{lem:tail}.

	For (ii), first observe that $$\nu_{h^{\eps_n}_{\op s}}^{\eps_n}(\partial \cS)^{4/\gamma^2-1}\Rightarrow 1$$ in law since $$4/\gamma^2-1\rta 0 \text{ and } \nu_{h^{\eps_n}_{\op s}}^{\eps_n}(\partial \cS)\rta\nu_{h_{\op s}}(\partial \cS).$$ Furthermore, Lemma \ref{lem:tail} gives the uniform integrability of $\nu_{h^\eps_{\op s}}^\eps(\partial \cS)^{4/\gamma^2-1}$  in $\eps$. Combining these two results we get (ii).
\end{proof}

\begin{remark}\label{rmk:disk-conv3}
	We reiterate that $\mu_{\wh{h}}(\cS)<\infty$ and $\nu_{\wh{h}}(\partial\cS)=1$ almost surely. Moreover, we have the convergence $\mu_{\wh h^\eps}^\eps(\cS)\Rightarrow \mu_{\wh h}(\cS)<\infty$ as $\eps\to 0$.
\end{remark}

\begin{remark}\label{rmk:disk-conv2} 
	For $b>0$ we define the $b$-boundary length disk to be a surface with the law of $(\cS,h^b)$, where $h^b=h+2\gamma^{-1}\log(b)$ for $h$ as in Definition \ref{def:disk} or \ref{def:critical-disk}. 
	Lemma \ref{prop:disk-conv} also holds if we assume all the disks are $b$-boundary length disks.
\end{remark}

The fields that appear in the statement of our main theorem are defined as follows.
\begin{definition} 
	We define fields $h^\eps$ (resp. $h$) to be parameterizations of unit boundary length $\gamma$-LQG disks (resp. the 2-LQG disk) by $\D$ instead of $\cS$.
	More specifically we take $\phi:\D\to \cS$ to be the conformal map from $\cS$ to $\D$ that sends $+\infty,-\infty,\im\pi$ to $1,-1,\im$, respectively. Then we set 
	$$h^\eps=\wh h^\eps \circ\phi+Q_\gamma \log|\phi'| \text{ and } h=\wh h \circ\phi+2\log|\phi'|,$$ where $\wh h^\eps$ (resp. $\wh{h}$) is the field in the strip $ \cS$ corresponding to Definition \ref{def:disk} (resp. Definition \ref{def:critical-disk}). 
	\label{conv:disk}
\end{definition}

\begin{remark}\label{rmk:disk-conv} Lemma \ref{lem:disk-conv} clearly also implies the convergence  $$(h^\eps,\mu^\eps_{h^\eps}, \nu_{h^\eps}^\eps)\Rightarrow (h,\mu_h,\nu_h)$$ as $\eps\to 0$ (with respect to $\corr{H^{-1}_{\mathrm{loc}}(\D)}$ convergence in the first coordinate, and weak convergence of measures on $\D,\partial \D$ in the final coordinates).
\end{remark}

In fact, it implies the convergence of various embeddings of quantum disks. Of particular use to us will be the following:

\begin{lemma}\label{cor:disc_mapped}
	Suppose that for each $\eps$, $\wh{h}^\eps$ is as in Remark \ref{rmk:disk-conv2} for some $b>0$ and that $\wt{h}^\eps$ is defined by: 
	choosing a point $z^\eps$ from $\mu^\eps_{\wh{h}^\eps}$ in $\cS$; defining $\psi^\eps:\cS\to \D$ conformal such that $\psi^\eps(z^\eps)=0$ and $(\psi^\eps)'(z^\eps)>0$;
	and setting $$\wt{h}^\eps:= \wh{h}^\eps \circ (\psi^\eps)^{-1}\corr{+Q_\gamma\log |((\psi^\eps)^{-1})'|.}$$
	
	Suppose similarly that $(\wt{h}, \wt{\mu})$ is defined by: 
	taking the field $\wh{h}$ in Remark \ref{rmk:disk-conv2} with the same $b>0$, picking a point $z$ from $\mu_{\wh{h}}$; taking $\psi:\cS\to \D$ conformal with $\psi'(z)>0$ and $\psi(z)=0$; and setting $$\wt{h}=\wh{h}+\psi^{-1}
\corr{+2\log|(\psi^{-1})'|}\, , \, \wt{\mu}=\mu_{\wt{h}}.$$
	
	Then as $\eps\to 0$, we have that
	$$(\wt h^\eps,\mu^\eps_{\wt h^\eps})\Rightarrow (\wt h, \wt \mu).$$
	Moreover, for any $m>0$  \begin{equation}
		\label{eq:nomasstoboundary}
		\mathbb{P}(\mu^\eps_{\wt{h}^\eps}( \D\setminus (1-\delta) \D)>m) \to 0 \text{ as } \delta\to 0 
	\end{equation} uniformly in $\eps$. This convergence is also uniform over $b\in[0,C]$ for any $0<C<\infty$.
\end{lemma}

\begin{proof}
	We assume that $b=1$; the result for other $b$ and the uniform convergence in \eqref{eq:nomasstoboundary} follows immediately from the definition in Remark \ref{rmk:disk-conv2}. 
	
	The proof then follows from Lemma \ref{lem:disk-conv}. We take a coupling where the convergence is almost sure: in particular, the fields $\wh h^\eps$ converge almost surely to $\wh h$ in $H^{-1}_{\corr{\mathrm{loc}}}(\cS)$ and the measures $\mu_{\wh h^\eps}^\eps$ converge weakly almost surely to $\mu_{\wh h}$ in $\cS$. 
	This means that we can sample a sequence of $z^\eps$ from the $\mu_{\wh h^\eps}^\eps$ and $z$ from $\mu_{\wh h}$, such that $z^\eps\to z\in \cS$ almost surely. 
	Since $z\in \cS$ is at positive distance from $\partial \cS$, this implies that the conformal maps $\psi^\eps$ converge to $\psi$ almost surely on compacts of $\cS$ and therefore that $\wt h^\eps \to \wt h$ in $H^{-1}_{\corr{\mathrm{loc}}}(\D)$ and $\mu_{\wt h^\eps}^\eps$ converges weakly to $\wt \mu$.
	Finally, \eqref{eq:nomasstoboundary} follows from the convergence proved above, and the fact that it holds for the limit measure $\mu_{\wt h}$.
\end{proof}

{Later, we will also need to consider fields obtained from the field $\wt h^\eps$ of Lemma \ref{cor:disc_mapped} via a random rotation. For this purpose, we record the following remark.}
\begin{remark}\label{rmk:changing_lengths}
	Suppose that $h_n$ are a sequence of fields coupled with some rotations $\theta_n$ such that $\bar{h}_n=h_n\circ \theta_n-2\gamma_n^{-1}\log \nu_{h_n}(\partial \D)$ has the law of $\wt h^{\eps_n}$ from Lemma \ref{cor:disc_mapped} with $b=1$, for some $\epn\downarrow 0$, $\gamma_n=\gamma(\eps_n)$. Suppose further that $(h_n, \nu_{h_n}(\partial \D),\mu_{h_n}(\D))\Rightarrow (h,\nu^*, \mu^*)$ in $H^{-1}_{\corr{\mathrm{loc}}}(\D)\times \R \times \R$ as $n\to \infty$. Then $\nu^*=\nu_h(\partial \D)$ and $\mu^*=\mu_h(\D)$ almost surely.  Indeed, $(h_n, \nu_{h_n}(\partial \D),\mu_{h_n}(\D),\theta_n,\bar{h}_n)$ is tight in $n$, and any subsequential limit $(h,\nu^*,\mu^*,\theta,\bar{h})$ has $(h,\nu^*,\mu^*)$ coupled as above. Since $\mu_{h_n}(A)=(\nu_{h_n}(\partial \D))^2\mu_{\bar{h}_n}(\theta_n^{-1}(A))$ for every $n$ and $A\subset \D$ it follows from Lemma \ref{cor:disc_mapped} that $\mu^*=(\nu^*)^2\mu_{\bar{h}}(\D)$ and $\nu_{\bar{h}}(\partial \D)=1$ a.s. On the other hand, it is not hard to see that $\bar{h}$ must be equal to $h\circ\theta-\log \nu^*$ a.s., which implies the result.
\end{remark}

\subsection{Mating of trees}

Mating of trees theory, \cite{DMS14}, provides a powerful encoding of LQG and SLE in terms of Brownian motion. 
We will state the version in the unit disk $\D$ below. 

Let $\alpha\in(-1,1)$ and let $Z^{(c)}$ be $c$ times a standard planar Brownian motion with correlation $\alpha>0$, started from $(1,0)$ or $(0,1)$. Condition on the event that $Z$ first leaves the first quadrant at the origin $(0,0)$; this is a zero probability event but can be made sense of via a limiting procedure, see for example \cite[Proposition 4.2]{AG19}. We call the resulting conditioned process (restricted until the time at which the process first leaves the first quadrant) a \emph{Brownian cone excursion with correlation $\alpha$}. Note that we use the same terminology for the resulting process for any $c$ and either choice of $(1,0)$ or $(0,1)$ for the starting point.  

To state the mating-of-trees theorem (disk version) we first introduce some notation. Let $(\D,h^\eps)$ denote a unit boundary length $\gamma$-LQG disk for $\gamma\in(\sqrt{2},2)$, embedded as described in Definition \ref{conv:disk}. Let $\eta^\eps$ denote a space-filling SLE$_{\kappa'}$ in $\D$, starting and ending at 1, which is independent of $h$. \corr{Recall that this is defined from a branching SLE$_\kp$ as described in Section \ref{sec:sf_sle}, where the branch targeted towards $z\in \D$ is denoted by $\eta_z^\eps$ (one can obtain $\eta_z^\eps$) from $\eta^\eps$ by deleting time intervals on which $\eta^\eps$ is exploring regions of $\D$ that have been disconnected from $z$).} Parametrize $\eta^\eps$ by the area measure induced by $h$. Let $Z^\eps=(L^\eps,R^\eps)$ denote the process started at $(0,1)$ and ending at $(0,0)$ which encodes the evolution of the left and right boundary lengths of $\eta^\eps$: see Figure \ref{fig:spacefilling_bl}.

\begin{figure}
	\centering
	\includegraphics[width=\textwidth]{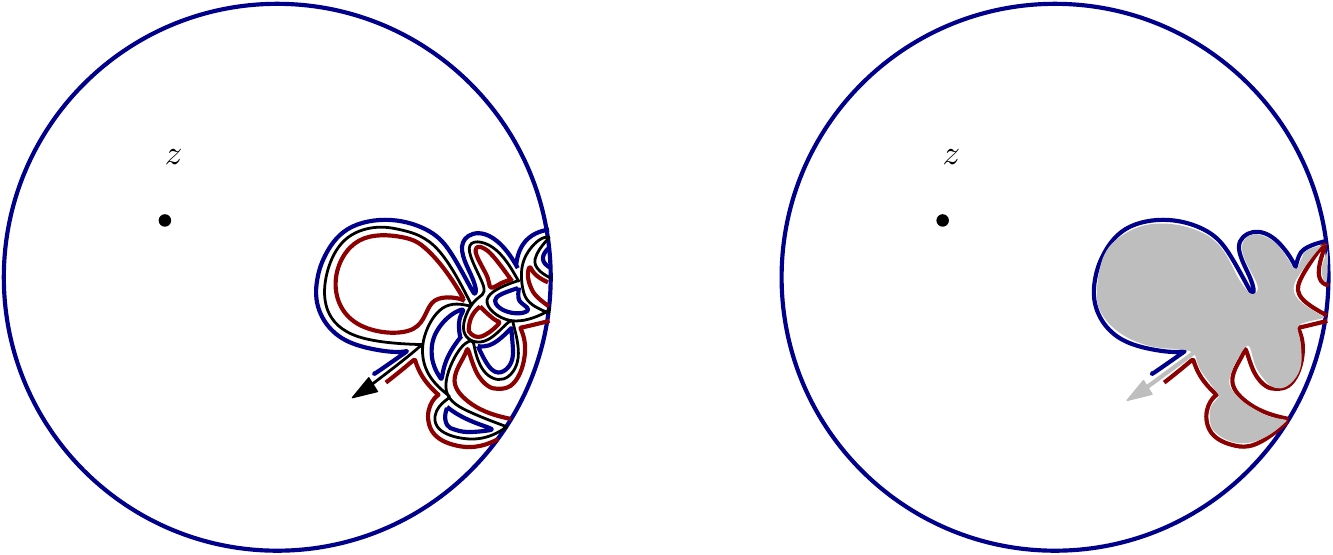}
	\caption{The left figure is an illustration of the branch of a space-filling SLE$_\kp$ ($\kp>4$) towards some point $z\in \D$, and stopped at some time before it reaches $z$. The space-filling SLE itself will fill in the monocolored components that are separated from $z$ as it creates them, so if $t$ is equal to the total $\gamma$-LQG area of the grey shaded region on the right figure, then the  space-filling SLE has visited precisely this grey region at time $t$. We then define the left (resp.\ right) boundary length of the space-filling SLE at time $t$ to be the $\gamma$-LQG boundary length of the red (resp.\ blue) curve shown on the right figure.} 
	\label{fig:spacefilling_bl}
\end{figure}

{The following theorem follows essentially from \cite{DMS14}. For precise statements, see 
	\cite[Theorem 2.1]{FDMOT} for the law of $Z^\eps$ and see \cite[Theorem 7.3]{FDMOT} for the law of the monocolored components.}
\begin{theorem}[\cite{DMS14,FDMOT}]
	In the setting above, $Z^\eps$ has the law of a Brownian cone excursion with correlation $-\cos(\pi\gamma^2/4)$. The pair $(h^\eps,\eta^\eps)$ is measurable with respect to the $\sigma$-algebra generated by $Z^\eps$. Furthermore, if $z$ is sampled from $\mu_{h^\eps}^\eps$ renormalized to be a probability measure, then the monocolored complementary components of $\eta^\eps_z$ define independent $\gamma$-LQG disks conditioned on their $\gamma$-LQG boundary lengths {and areas}, i.e., if we condition on the ordered sequence of boundary lengths {and areas} of the monocolored domains $U$ disconnected from $z$ by $\eta^\eps_z$ then the corresponding LQG surfaces $(U,h|_U)$ are independent $\gamma$-LQG disks with the given boundary lengths and areas.
	\label{thm:MOT}
\end{theorem}
\begin{remark}\label{rmk:varZ}
	In fact, we now know from \cite{ARS} that the variance $c^2$ of the Brownian motion from which the law of $Z^\eps$ can be constructed is equal to $1/(\eps\sin(\pi\gamma^2/4))$, where $\gamma=\gamma(\eps)=2-\eps$. In particular, the variance is of order $\eps^{-2}$. 
\end{remark}

For each fixed $z\in\D$ there is a natural parametrization of $\eta^\eps_z$ called its \emph{quantum natural parametrization} which is defined in terms of $Z^\eps$ as follows. First define $\frk t=\inf\{ t\geq 0\,:\,\eta^\eps(t)=z \}$ to be the time at which $\eta^\eps$ first hits $z$. Then let $\cI^{\eps,\frk t}$ denote the set of $s\in[0,\frk t]$ for which we \emph{cannot} find a cone excursion $J\subset[0,\frk t]$ %for $Z^\eps$ 
\corr{(that is, $J=[t_1,t_2]\subset [0,\frk t]$ such that $(X^\eps_s,Y^\eps_s)\ge (X^\eps_{t_2},Y^\eps_{t_2})$ on $J$, and either $X^\eps_{t_1}=X^\eps_{t_2}$ or $Y^\eps_{t_1}=Y^\eps_{t_2}$)} such that $s\in J$. We call the times in $\cI^{\eps,\frk t}$ \emph{ancestor-free times relative to time $\frk t$.} It is possible to show (see \cite[Section 1.4.2]{DMS14}) that the local time of $\cI^{\eps,\frk t}$ is well-defined.\footnote{This local time (and the corresponding local time for $\eps=0$ defined below) is only defined up to a deterministic multiplicative constant. We fix this constant in the proof of Lemma \ref{lem:BM-conv}.} Let $(\ell^{\eps,\frk t}_t)_{t\geq 0}$ denote the increasing function describing the local time of $\cI^{\eps,\frk t}$ such that $\ell^{\eps,\frk t}_0=0$ and $\ell^{\eps,\frk t}_t= \ell^{\eps,\frk t}_{\frk t}$ for $t\geq \frk t$. Then let $T^{\eps,\frk t}_t$ for $t\in[0,\ell_{\frk t}^{\eps,\frk t}]$ denote the right-continuous inverse of $\ell^{\eps,\frk t}$. 

\begin{definition}[Quantum natural parametrization]
	With the above definitions $$(\eta^\eps_{z}(T^{\eps,\frk t}_t))_{t\in[0,\ell_{\frk t}^{\eps,\frk t}]}$$ defines a parametrization of $\eta^\eps_z$ which is called its quantum natural parametrization. \label{def:QNT}
\end{definition}

\subsection{Convergence of the mating of trees Brownian functionals}\label{sec:Bfs}

Let $Z^\eps$ be the process from Theorem \ref{thm:MOT} and let 
$X^\eps=(A^\eps,B^\eps)$, where
$$
A^\eps_t=a_\eps(L^\eps_t+R^\eps_t),\qquad
B^\eps_t=R^\eps_t-L^\eps_t,\qquad 
a_\eps^2=\frac{1+\cos(\pi\gamma^2/4)}{1-\cos(\pi\gamma^2/4)},
t\geq 0.
$$

\begin{figure}[h]
	\centering
	\includegraphics[width=\textwidth]{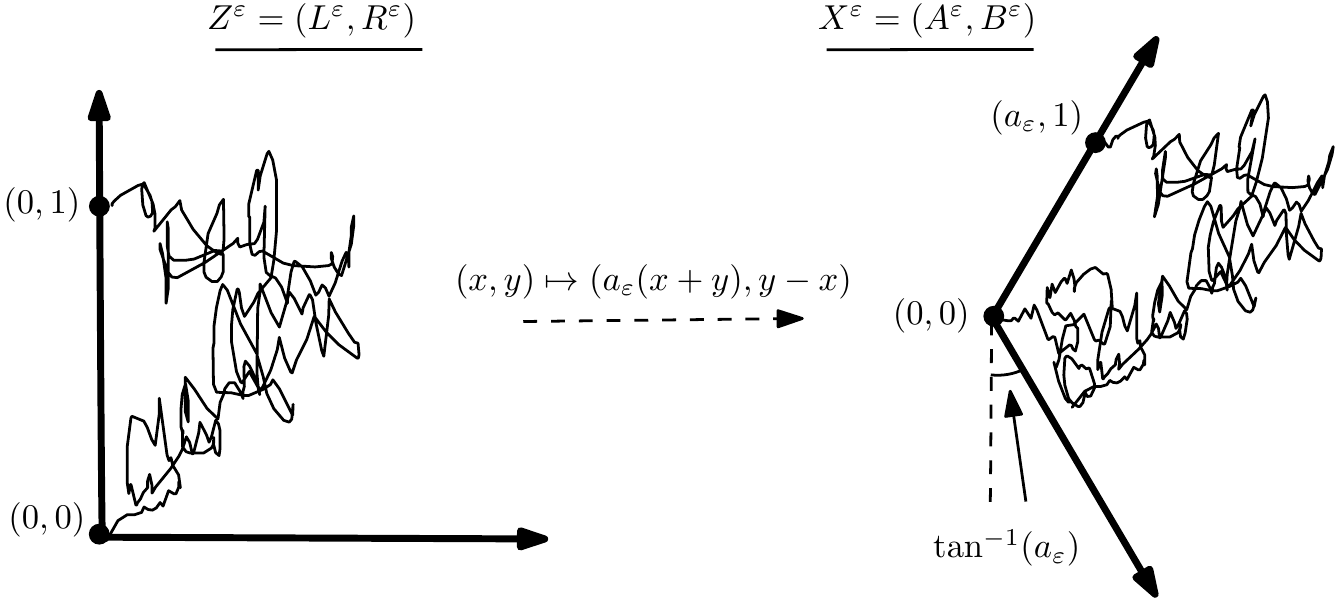}
	\caption{\corr{The transformation from $Z^\eps$ to $X^\eps$.}}
\end{figure}
	
Note that $a_\eps=\eps\pi/2+o(\eps)$ and that $X^\eps$ is an uncorrelated Brownian excursion with variance {$2(1+\cos(\pi\gamma^2/4))(\eps \sin(\pi\gamma^2/4))^{-1}=\pi+o(\eps)$} in the cone $\{z\in \C: \arg(z)\in [-\pi/2+\tan^{-1}(a_\eps),\pi/2-\tan^{-1}(a_\eps))\}$, starting from $(a_\eps,1)$ and ending at the origin.
Also define the processes
$\wh X^{\eps,\frk t}=(\wh A^{\eps,\frk t},\wh B^{\eps,\frk t})$ for each $\frk t<\mu^\eps(\D)$, by setting
$$\wh X^{\eps,\frk t}_t= X^{\eps,\frk t}_{T^{\eps,\frk t}_t}\, ; \, t> 0. $$

We will prove in this subsection that all the quantities defined above have a joint limit in law as $\eps\downarrow 0$. 
Namely, let us consider an uncorrelated Brownian excursion $X=(A,B)$ in the right half plane from $(0,1)$ to $(0,0)$; {the process can for example be constructed via a limiting procedure where we condition a standard planar Brownian motion from $(0,1)$ to $(0,0)$ on first leaving $\{z\,:\,\op{Re}(z)>-\delta \}$ at a point $\wh z$ where $|\op{Im}(\wh z)|<\delta$}. For $\frk{t}$ less than the total duration of $X$, let $\cI^{\frk t}\subset[0,\frk t]$ denote the set of times at which $A$ has a backward running infimum relative to time $\frk t$, i.e., $s\in\cI^{\frk t}$ if $A_u>A_s$ for all $u\in(s,\frk t]$. Let $(\ell^{\frk t}_t)_{t\geq 0}$ denote the increasing function describing the local time of $\cI^{\frk t}$ such that $\ell^{\frk t}_0=0$ and $\ell^{\frk t}_t= \ell^{\frk t}_{\frk t}$ for $t\geq \frk t$. Then let $T^{\frk t}$ denote the right-continuous inverse of $\ell^{\frk t}$, and define 
$\wh X^{\frk t}=(\wh A^{\frk t},\wh B^{\frk t})$ by 
$\wh X^{\frk t}_t= X^{\frk t}_{T^{\frk t}_t}$.\\

We set $$\BM^\eps=(X^\eps,(\cI^{\eps,\frk t})_{\frk t}, (\ell^{\eps,\frk t})_{\frk t}, (T^{\eps,\frk t})_{\frk t}, 
(\wh X^{\eps,\frk t})_{\frk t} )$$ and $$\BM=(X,(\cI^{\frk t})_{\frk t}, (\ell^{\frk t})_{\frk t}, (T^{\frk t})_{\frk t}, 
(\wh X^{\frk t})_{\frk t} )$$
where the indexing is over $\frk t\in \R_+\cap \mathbb{Q}$.

Then we have the following convergence. 

\begin{lemma}\label{lem:BM-conv}
	$\BM^\eps\Rightarrow \BM$ as $\eps\downarrow 0${, where we use the Hausdorff topology on the second coordinate and the Skorokhod topology on the remaining coordinates.}
\end{lemma}
\begin{proof}
	First we consider the infinite volume case where $X^\eps$ is a two-sided planar Brownian motion started from $0$, with  the same variance and covariance as before, namely variance $2(1+\cos(\pi\gamma^2/4))(\eps \sin(\pi\gamma^2/4))^{-1}=\pi+o(\eps)$ and covariance $0$.
	In this infinite volume setting we define $(\cI^{\eps,{\frk t}})_{\frk t}, (\ell^{\eps,{\frk t}})_{\frk t}, (T^{\eps,{\frk t}})_{\frk t}, 
	(\wh X^{\eps,{\frk t}})_{\frk t}$ similarly to before, such that for $\eps\in (0,2-\sqrt{2})$, 
	$\cI^{\eps,{\frk t}}\subset(-\infty,{\frk t})$ is the set of ancestor-free times relative to 
	time ${\frk t}$, $\ell^{\eps,{\frk t}}:\R\to(-\infty,0]$ is an increasing process given by the local time of $\cI^{\eps,{\frk t}}$ satisfying 
	$\ell^{\eps,{\frk t}}_s\equiv 0$ for $s\geq {\frk t}$, $T^{\eps,{\frk t}}:(-\infty,0)\to(-\infty,0)$ is the 
	right-inverse of $\ell^{\eps,{\frk t}}$, and $\wh X^{\eps,{\frk t}}_s=X^{\eps}_{T^{\eps,{\frk t}}_s}$. We make a similar adaptation of the definition to the infinite volume setting for $\eps=0$; in particular, $X$ is ($\sqrt{\pi}$ times) a {standard} uncorrelated two-sided Brownian motion planar motion. By translation invariance in law of $X^\eps$ and $X$, and since $X^\eps$ and $X$ determine the rest of the objects in question, it is sufficient to show convergence for ${\frk t}=0$. 
	
	\nina{First we claim that for all $\eps\in[0,2-\sqrt{2})$ we can sample $\cI^{\eps,0}$ by considering a PPP in the second quadrant with intensity $dx\times y^{-\alpha(\eps)}dy$ for $\alpha(\eps)=1+2/(2-\eps)^2=1+2/\gamma^2$, such that the points $(x,y)$ of this PPP are in bijection with the complementary components of $\cI^{\eps,0}$ with $y$ representing the length of the component and $x$ representing the relative ordering of the components. (In the case $\eps=0$, $\cI^{0,0}$ refers to $\cI^0$.) For $\eps=0$ the claim follows since $A$ restricted to the complementary components of $\cI^0$ has law given by the Brownian excursion measure. For $\eps\in(0,2-\sqrt{2})$ the claim follows from \cite{DMS14}: It is explained in \cite[Section 1.4.2]{DMS14} that $\cI^{\eps,0}$ has the law of the zero set of some Bessel process, which verifies the claim modulo the formula for $\alpha(\eps)$. The dimension of $\cI^{\eps,0}$ is $2/\gamma^2$ \cite[Table 1 and Example 2.3]{GHM-KPZ}, and we get the formula for $\alpha(\eps)$ by adding 1 to this number.}
	
	\nina{Next} we argue that the marginal law of $\cI^{\eps,0}$ converges to the marginal law of $\cI^{0}$. \nina{Consider the definition of these sets via PPP as described in the previous paragraph.} Since $\lim_{\eps\rta 0}\alpha(\eps)=\alpha(0)=3/2$, the PPP for $\eps\in (0,2-\sqrt{2})$ converge in law to the PPP for $\eps=0$ on all sets bounded away from $y=0$. %which implies  convergence in law of $\cI^{\eps,0}$ to $\cI^{0}$ for the Hausdorff distance \nina{restricted to any compact set}. 
	%First we argue that the marginal law of $\cI^{\eps,0}$ converges to the marginal law of $\cI^{0}$. For all $\eps\in[0,2-\sqrt{2})$ we can sample $\cI^{\eps,0}$ by considering a PPP in the second quadrant with intensity $dx\times y^{-\alpha(\eps)}dy$ for $\alpha(\eps)=1+2/(2-\eps)^2=1+2/\gamma^2$. The points $(x,y)$ of this PPP are in bijection with the complementary components of $\cI^{\eps,0}$ such that $y$ gives the length of the component and $x$ gives the relative ordering of the components. Since $\lim_{\eps\rta 0}\alpha(\eps)=\alpha(0)=3/2$, the PPP for $\eps\in (0,2-\sqrt{2})$ converge in law to the PPP for $\eps=0$ on all sets bounded away from $y=0$. %which implies  convergence in law of $\cI^{\eps,0}$ to $\cI^{0}$ for the Hausdorff distance \nina{restricted to any compact set}. 
	\nina{This implies that for any compact interval $I$ we have convergence in law of $\cI^{\eps,0}\cap I$ to $\cI^{0}\cap I$ for the Hausdorff distance}.
	
	\nina{Now} we will argue that if $\wt\cI^{\eps,0}\subset(-\infty,0)$ denotes the backward running infima of $A^\eps$ relative to time 0, then
	\eqbn
	(X^\eps,\cI^{\eps,0},\wt\cI^{\eps,0})\Rightarrow (X,\cI^{0},\cI^{0}).
	\eqen
	%Since $(X^\eps,\cI^{\eps,0})\Rightarrow (X,\cI^{0})$ and $\wt\cI^{\eps,0}\Rightarrow \cI^{0}$, we only need to prove that for any a.s.\ subsequential limit $(X,\cI^{0},\wt\cI^{0})$ we have $\cI^0=\wt\cI^0$ a.s. Observe that $\wt\cI^{\eps,0}\subset\cI^{\eps,0}$, which implies that $\wt\cI^{0}\subset\cI^{0}$ a.s. Since $\wt\cI^{0}\eqD\cI^{0}$, this implies that $\cI^0=\wt\cI^0$ a.s.
	Since $(X^\eps,\nina{\wt\cI^{\eps,0}})\Rightarrow (X,\cI^{0})$ and 
	$\nina{\cI^{\eps,0}}\Rightarrow \cI^{0}$, we only need to prove that for any a.s.\ subsequential limit $(X,\cI^{0},\wt\cI^{0})$ we have $\cI^0=\wt\cI^0$ a.s. Observe that $\wt\cI^{\eps,0}\subset\cI^{\eps,0}$ \nina{since $\wt\cI^{\eps,0}$ denotes the backward running infima of $A^\eps$, $\cI^{\eps,0}$ denotes the set of ancestor-free times of $A^\eps$ relative to time 0, and a time which is a backward running infimum of $A^\eps$ relative to time 0 cannot be inside a cone excursion, hence it is ancestor-free.} \nina{The observation $\wt\cI^{\eps,0}\subset\cI^{\eps,0}$} implies that 
	$\wt\cI^{0}\subset\cI^{0}$ a.s.\ \nina{in any subsequential limit $(X,\cI^{0},\wt\cI^{0})$.} Since $\wt\cI^{0}\eqD\cI^{0}$, this implies that $\cI^0=\wt\cI^0$ a.s.
	
	Next we will argue that $(\cI^{\eps,0},\ell^{\eps,0},T^{\eps,0})
	\Rightarrow 
	(\cI^{0},\ell^{0},T^{0})$, assuming we choose the multiplicative constant consistently when defining $\ell^{\eps,0}$ and $\ell^{0}$. The convergence result follows again from the construction of $\cI^{\eps,0}$ and $\cI^0$ via a PPP, since the $x$ coordinate of the PPP defines the local time (modulo multiplication by a deterministic constant). 
	
	Using that $(\cI^{\eps,0},\ell^{\eps,0},T^{\eps,0})
	\Rightarrow 
	(\cI^{0},\ell^{0},T^{0})$, that $\cI^{\eps,0}$ and $\cI^{0}$ determine the other two elements in this tuple, and that $(X^\eps,\cI^{\eps,0})\Rightarrow(X,\cI^{0})$, we get
	$$
	(X^\eps,\cI^{\eps,0},\ell^{\eps,0},T^{\eps,0})
	\Rightarrow 
	(X^0,\cI^{0},\ell^{0},T^{0}).
	$$ 
	We conclude that the lemma holds in the infinite volume setting by using that
	$$\wh X^{\eps,0}_s=X^{\eps}_{T^{\eps,0}_s} \text{ and  }\wh X_s=X_{T^{0}_s}.$$
	
	To conclude the proof we will transfer from the infinite volume setting to the finite volume setting. Let us start by recalling that there is a natural infinite measure $\theta_\eps$ on Brownian excursions in the cone $\cC_\eps:=\{z\in \C: \arg(z)\in (-\pi/2+\tan^{-1}(a_\eps),\pi/2-\tan^{-1}(a_\eps))\}$  which is uniquely characterized (modulo multiplication by a constant) by the following property. Let $X^\eps$ be as in the previous paragraph, let $\delta>0$ and let $J_\eps=[t_1,t_2]\subset\R_-$ be the interval with largest left end point $t_1$ of length at least $\delta$ during which $X^\eps$ makes an excursion in the cone $\cC_\eps$. Here a cone excursion in $\cC_\eps$ is a path starting at $(ba_\eps,b)+z_0$ for some $b>0$ and $z_0\in\C$, ending at $z_0$, and otherwise staying inside $z_0+\cC_\eps$. Define \begin{equation}\label{XYeq}
		Y^\eps_t=(X^\eps_{t+t_1}-X^\eps_{t_2})
	\end{equation} 
	for $t\in[0,t_2-t_1]$ so that $Y^\eps$ is a path that starts at $(ba_\eps,b)$ for some $b>0$, ends at the origin, and otherwise stays inside $\cC_\eps$. Then $Y^\eps$ has law $\theta_\eps$ restricted to excursions of length at least $\delta$. (Here and in the rest of the proof, when we work with a non-probability measure of finite mass, we will often assume that it been renormalized to be a probability measure.) 
	See \cite{Shi85}. 
	
	The measure $\theta_\eps$ allows a disintegration $\theta_\eps=\int_0^\infty \theta_\eps^b\,db$, where a path sampled from $\theta_\eps^b$ a.s.\ starts at $(ba_\eps,b)$. Furthermore, for $b,b'>0$, a path sampled from $\theta_\eps^b$ and rescaled by $b'/b$ so it ends at $(b'a_\eps,b')$ (and with Brownian scaling of time), has law $\theta_\eps^{b'}$. Finally, an excursion sampled from $\theta_\eps^1$ is equal in law to the excursion in the statement of the lemma. See \cite{AG19}. 
	
	Let us now use these facts to complete the proof.	We define a function $f^\eps$ such that for $X^\eps$ a two-sided planar Brownian motion as above we have $f^\eps(X^\eps)=((\cI^{\eps,\frk t})_{\frk t}, (\ell^{\eps,\frk t})_{\frk t}, (T^{\eps,\frk t})_{\frk t}, 
	(\wh X^{\eps,\frk t})_{\frk t} )$ a.s. For $Y^\eps$ a Brownian cone excursion in $\cC_\eps$ starting at $(a_\eps,1)$ we define $f^\eps(Y^\eps)$ such that $(Y^\eps,f^\eps(Y^\eps))$ is equal in law to the tuple $\BM^\eps$ in the theorem statement. We also extend the definition of $f^\eps$ to the case of Brownian excursions $Y^\eps$ in $\cC_\eps$ starting at $(ba_\eps,b)$ for general $b>0$ in the natural way.

	Now let $Y^\eps$ be coupled with $X^\eps$ as in \eqref{XYeq} for some fixed $\delta>0$, and let $E^\eps$ be the event that $Y^\eps$ starts at $(ba_\eps,b)$ for $b\in [1,2]$. Define $f,E$ similarly for $\eps=0$. We claim that
	\eqb
	(X^\eps,f^\eps(X^\eps),Y^\eps,f^\eps(Y^\eps),E^\eps)
	\Rightarrow
	(X,f(X),Y,f(Y),E)
	\label{eq:finite-infinite}
	\eqe
	as $\eps\to 0$. In fact, this claim is immediate since if $(X^\eps,f^\eps(X^\eps))$ converges to $(X,f(X))$ then (by convergence of  $\cI^{\eps,0}$) we also have convergence of the interval $J_\eps$, which further gives convergence of $
	(Y^\eps,f^\eps(Y^\eps),E^\eps)$ to $(Y,f(Y),E)$.
	
	With $Y^\eps$ as in the previous paragraph let $\wt Y^\eps$ denote a random variable which is obtained by conditioning on $E^\eps$ and then applying a Brownian rescaling of $Y^\eps$ so that $\wt Y^\eps$ starts at $(a_\eps,1)$. We get from \eqref{eq:finite-infinite} that $(\wt Y^\eps, f^\eps(\wt Y^\eps)) \Rightarrow (\wt Y,f(\wt Y))$. Note that if we condition the excursions in the statement of the lemma to have duration at least $\delta$, then these have exactly the same laws as $(\wt Y^\eps, f^\eps(\wt Y^\eps),\wt Y,f(\wt Y))$ conditioned to have duration at least $\delta$. Thus the lemma follows upon taking $\delta\to 0$, since the probability that the {considered} excursions 
	{have} duration at least $\delta$ tends to $1$, uniformly in $\eps$.
\end{proof} 

 \subsection{Proof of \eqref{eqn:bottleneck}}\label{sec:order_proof}

\corr{Let us first recall the statement of \eqref{eqn:bottleneck}. We have fixed $z,w\in \D$, and as usual, $\eta^\eps$ denotes a space filling SLE$_\kp$ in $\D$, while $\eta^\eps_z$ denotes the branch in the associated branching SLE$_\kp$ towards $z$, parameterized by $-\log$ conformal radius seen from $z$. For $\delta>0$, we have defined the times $\sigma_{z,w,\delta}^\eps$ that $w$ is sent first sent to within distance $\delta$ of $\partial \D$ by the  Loewner maps associated with $\eta^\eps_z$, and $\sigma_{z,w}^\eps=\sigma_{z,w,0}^\eps$ to be the first time that $z$ and $w$ are separated by $\eta^\eps_z$.  For $r>0$, we  denote the collection of faces (squares) of $r\Z^2$ that intersect $\D$ by $\cS_r$. Finally, we write $S_{\delta,r}^\eps$ for the event that  there exists $S\in \cS_r$ that is separated by $\eta_z^\eps$ from $z$ during the interval $\sigma_{z,w,\delta}^\eps, \sigma_{z,w}^\eps]$  \emph{and} such that $z$ is visited by the space-filling $\SLE_\kp$ $\eta^\eps$, before $S$. The statement of \eqref{eqn:bottleneck} is then that
	\begin{equation*}
		\lim_{\delta\downarrow 0} \lim_{\eps\downarrow 0} \mathbb{P}(S_{\delta,r}^\eps) =0.
	\end{equation*} }

The mating of trees theorem (Theorem \ref{thm:MOT}) together with the convergence proved in the previous subsection now make the proof of this statement reasonably straightforward. Indeed, in plain language, it says that the probability of an $\SLE_\kp(\kp-6)$ branch almost separating two points $z$ and $w$ (where ``almost'' is encoded by a small parameter $\delta$) but then going on to separate a bicolored component of macroscopic size from $z$ at some time $t$ \emph{strictly} before separating $z$ from $w$, goes to $0$ as $\delta\to 0$, uniformly in $\kp$. The idea is to couple this SLE with an independent $\gamma$-LQG disk and note that if the event mentioned above were to occur, then the component $U$ containing $z$ and $w$ at time $t$ would have a small ``bottleneck'' and hence define a very strange distribution of $\gamma$-LQG mass when viewed as a $\gamma$-LQG surface. On the other hand, if we sample several points from the $\gamma$-LQG area measure on the disk, then one of these is likely to be in the bicolored component separated from $z$ and $w$ at time $t$. So the mating of trees theorem says that $U$ should really look like a quantum disk, and in particular, have a rather well behaved distribution of $\gamma$-LQG mass {without bottlenecks}. This contradiction will lead us to the proof of \eqref{eqn:bottleneck}.

Let us now get on with the details. For $\eps\in(0,2-\sqrt{2})$ we consider a CLE$_{\kp}$ exploration alongside an \emph{independent} unit boundary length quantum disk $h^\eps$ as in Definition \ref{conv:disk}. We write $\mu^\eps$ for its associated LQG area measure and let $y^\eps$  be a point in $\D$ sampled from $\mu^\eps$ normalized to be a probability measure. We let $z\in \cQ$ be fixed.
\begin{corollary}\label{cor:nobottleneck}
	Consider the event $A^\eps_{\delta,m,v}$ that: 
	\begin{itemize}\item $\mathcal{O}_{z,y^\eps}^\eps=1$ (i.e., the component containing $z$ when $y^\eps$ and $z$ are separated is monocolored);
		\item when ${\lcTB}_{z,y^\eps}^\eps$ (this monocolored component) 
		is mapped to $\D$, with a point in the interior chosen proportionally to $\mu^\eps|_{{\lcTB}_{z,y^\eps}^\eps}$ sent to $0$, the resulting quantum mass of $\D\setminus (1-10\delta \D)$ is greater than $m$.
	\end{itemize} Then for every $m$ we have that
	\[ \lim_{\delta\to 0}\limsup_{\eps\to 0}\mathbb{P}(A_{\delta,m,v}^\eps)=0.\]
\end{corollary}

\begin{proof}
	Theorem \ref{thm:MOT} says that the monocolored components separated from $y^\eps$ by $\eta^\eps_{y^\eps}$ are quantum disks conditionally on their boundary lengths and areas. Moreover, we know that the total mass of the original disk $h^\eps$ converges in law to something a.s.\ finite as $\eps\to 0$, by Lemma \ref{lem:disk-conv} and Remark \ref{rmk:disk-conv3}. Recalling the definition of $\wh{B}$ from Section \ref{sec:Bfs}, we also know that the largest quantum boundary length among all monocolored components separated from $y^\eps$ has law given by the largest jump of $\wh{B}^{\frk t}$, for $\frk t$ chosen uniformly in $(0,\mu^\eps(\D))$. Indeed, if $\frk t$ corresponds to $y^\eps$ as in the paragraph above Definition \ref{def:QNT}, then $\frk{t}$ is a uniform time in $(0,\mu^\eps(\D))$ and the jumps of $\wh{B}^{\frk t}$ 
	are precisely the quantum boundary lengths of the monocolored components disconnected from $y^\eps$.  By Lemma \ref{lem:BM-conv} we may deduce that the law of this largest jump 
	converges to something a.s.\ finite as $\eps\to 0$. Thus, by choosing $N,L$ large enough, we may work on an event with arbitrarily high probability (uniformly in $\eps$) where there are fewer than $N$ monocolored components separated for $y^\eps$ with mass $\ge m$, and where they all have $\nu^\eps$ boundary length less than $L$.   Lemma \ref{cor:disc_mapped} then provides the result. 
\end{proof}\\

We also need one more elementary property of radial Loewner chains to assist with the proof of \eqref{eqn:bottleneck}. 

\begin{lemma}\label{lem:radial_flow_out}
	Consider the image $(g_t(z))_{t\ge 0}$ of a point $z\in \D$ under the radial Loewner flow $(g_t)_{t\ge 0}=(g_t[\lcB])_{t\ge 0}$ 
	corresponding to $\lcB\in\lcS$.  Then  with probability one, $|g_t(z)|$ is a non-decreasing function of time (until the point $z$ is swallowed).
\end{lemma}
\begin{proof}
	From the radial Loewner equation one can compute directly that, until the point $z$ is swallowed, $$\partial_t (|g_t(z)|^2) = 2 |g_t(z)| \Re \big(\frac{W_t+g_t(z)}{W_t-g_t(z)}\big).$$
	Since $\Re((1+x)/(1-x))>0$ for any $x\in \D$, the right-hand side above must be positive.
\end{proof}\\

\begin{proofof}{\eqref{eqn:bottleneck}}  
	Fix $r>0$ and suppose  that $\mathbb{P}(S_{\delta,r}^\eps)\ge a$ for some $a>0$. Recall that $S_{\delta,r}^\eps$ is the event that  there exists $S\in \cS_r$ that is separated by $\eta_z^\eps$ from $z$ during the interval $[\sigma_{z,w,\delta}^\eps, \sigma_{z,w}^\eps]$ {and} such that the disconnected component containing $z$ is monocolored.  Let  $h^\eps, \mu^\eps,y^\eps$ be as above Corollary \ref{cor:nobottleneck}, and let $a'=\inf_{\eps> 0}\min_{S\in \cS_r}\mathbb{P}(y^\eps\in S)$. Then $a'$ is strictly positive, due to the convergence result Lemma \ref{conv:disk}, plus the fact that $\min_{S\in \cS_r}\mathbb{P}(y\in S)>0$ when $y$ is picked from the critical LQG area measure for a critical unit boundary length disk. By independence, we then have $\mathbb{P}(E_\delta^\eps)\ge aa'$, where $E_{\delta}^\eps$ is the event that $\sigma_{z,y^\eps}\in [\sigma_{z,w,\delta}^\eps,\sigma_{z,w}^\eps]$ and $\cO_{z,y}^\eps=1$. 
	
	We can also choose $v,m$ small enough and $K$ large enough that on an event $F_{m,v,K}^\eps$ with probability $\ge 1-aa'/2$, uniformly in $\eps$: 
	\begin{itemize}
		\itemsep0em
		\item $B_z(v)\subset l_z^\eps$ (resp. $B_w(v)\subset l_w^\eps$) where $l_z$ (resp. $l_w^\eps$) is the first nested $\CLE_{\kp}$ bubble containing $z$ (resp. $w$) that is entirely contained in $B_z(|z-w|/3))$ (resp. $B_w(|z-w|/3)$;
		\item $B_z(v)$ and $B_w(v)$ have $\mu$-mass greater than or equal to $m$;  
		\item if we map $l_z^\eps$ (resp. $l_w^\eps$) to $\D$ with $z$ (resp. $w$) sent to $0$, then the images of $B_z(v)$ and $B_w(v)$ are contained in $(1/2)\D$;
		\item $\mu^\eps(\D)\le K$.
	\end{itemize}
	Again this is possible because such $v,m,K$ can be chosen when $\eps=0,\kp=4$, and we can appeal to the convergence results Proposition \ref{prop:cleloopconv} and Lemma \ref{conv:disk}. Note that on the event $F_{v,m,K}^\eps$:
	\begin{itemize}
		\itemsep0em
		\item[(i)] $B_w(v)$ and $B_z(v)$ are contained in $({\lcB}_{z}^\eps)_t$ for all $t\in(\sigma^\eps_{z,w,\delta},\sigma^\eps_{z,w})$;
		\item[(ii)] for any $t\in(\sigma^\eps_{z,w,\delta},\sigma^\eps_{z,w})$ and conformal map sending $({\lcB}_{z}^\eps)_t$ to $\D$ with $z'\in B_z(v)$ sent to $0$, the image of $B_w(v)$ is contained in a $10\delta$ neighbourhood of $\partial \D$. 
	\end{itemize}
	Point (ii) follows because any such conformal map can be written as the composition of a conformal map  $({\lcB}_{z}^\eps)_t$ to $\D$ sending $z$ to 0, and then a conformal map from $\D\to \D$ sending the image of $z'$, which lies in $(1/2)\D$, to $0$.  By Lemma \ref{lem:radial_flow_out}, $v$ is sent to distance $\le \delta$ from the boundary by the first of these two maps. The third bullet point in the definition of $F_{v,m,K}$ then implies that the whole of $B_w(v)$ is actually sent within distance $4\delta$ of $\partial \D$. Distortion estimates near the boundary for the second conformal map allow one to deduce (ii).
	
	To finish the proof, we consider the event $E_\delta^\eps \cap F_{m,v,K}^\eps$ which has probability $\ge aa'/2$ by  construction. Conditionally on this event, if we sample a point from $\lcB_{z,y^\eps}^\eps$ according to the measure $\mu^\eps$, then this point will lie in $B_z(v)$ with conditional probability $\ge m/K$.  If this happens, then upon mapping to the unit disk with this point sent to the origin, a set of $\mu^\eps$ mass $\ge m$ (namely $B_z(v)$) will necessarily be sent to $\D\setminus (1-10\delta)\D$ (see point (ii) above). Note that $m/K$ is a function $c(a)$ of $a$ only (and in particular does not depend on $\eps,\delta$).
	
	So in summary, if $\mathbb{P}(S_{\delta,r}^\eps)\ge a$, then $\mathbb{P}(A_{\delta, m,v}^\eps)>aa'c(a)$ for some $m(a),v(a),c(a)$ depending only on $a$, where $A_{\delta,m,v}^\eps$ is as in Corollary \ref{cor:nobottleneck}. This means that if \eqref{eqn:bottleneck} does not hold, then $\lim_{\delta\to 0} \limsup_{\eps\to 0} \mathbb{P}(A^\eps_{\delta, m,v})>0$ for some $m,v$. This contradicts Corollary \ref{cor:nobottleneck}, and hence \eqref{eqn:bottleneck} is proved.
	
\end{proofof}

\section{Mating of trees for $\kappa=4$ and  joint convergence of CLE, LQG, and Brownian motions as $\kappa'\downarrow 4$}
\label{sec:joint-conv}
Before stating the main theorems, let us briefly take stock of the progress so far. 
Recall that to each $\eps\in (0,2-\sqrt{2})$ we associate $\kp=\kp(\eps)=16/(2-\eps)^2$, and write $(\lcB_z^\eps)_{z\in \cQ}$ for the  SLE$_{\kappa'}(\kappa'-6)$ branches from 1 to $z$ in a branching SLE$_\kp$ in $\D$. These are generated by curves $(\eta^\eps_z)_{z\in \cQ}$, so that $(\lcB_z^\eps)_t$ is the connected component of $\D\setminus \eta_z^\eps$ containing $z$ for every $z$ and $t$. Recall that this branching SLE defines a nested CLE$_{\kappa'}$ which we denote by $\Gamma^{\eps}$, and a space-filling $\SLE_\kp$ which we denote by $\eta^\eps$. {The space-filling SLE$_\kp$} $\eta^\eps$ then determines an order on the points in $\cQ$: for $z,w\in \cQ$ we denote by $\cO_{z,w}^\eps$ the random variable that is $1$ if $z$ is visited before $w$ by $\eta^\eps$ (or $z=w$) and $0$ otherwise. We combine these and set
$$\loops^\eps=((\lcB^\eps_z)_z,\Gamma^\eps,(\cO_{z,w}^{\eps})_{z,w})$$
for each $\eps$, where $z,w$ are indexed by $\cQ$.

When $\kp=4$ we have analogous objects. We write $\Gamma$ for a nested CLE$_4$ in $\D$, and we assume that $\Gamma$ is coupled with a branching uniform $\CLE_4$ exploration that explores its loops. We write
$\lcB_z$ for the branch towards each $z\in \cQ$ in this exploration. Finally, we define a collection of independent coin tosses $(\cO_{z,w})_{z,w\in \cQ}$ as described at the start of Section \ref{sec:conv_order}. Combining these, we set 
$$ \loops=((\lcB_z)_z,\Gamma,(\cO_{z,w})_{z,w}).$$

The processes $\lcB^\eps_z,\lcB_z$ are each parameterized by $-\log $ conformal radius seen from $z$, and equipped with the topology of $\lcS_z$ for every $z\in \cQ$. 
{The loop ensembles} $\Gamma^\eps,\Gamma$ are equipped with the topology of Hausdorff convergence for the countable collection of loops surrounding each $z\in\cQ$.

We also consider, for each $\eps$, a unit boundary length Liouville quantum gravity disk as in Definition \ref{conv:disk}, 
{\emph{independent of}} 
$\loops^\eps$, and write
$$	\LQG^\eps=(\mu^\eps_{h^\eps},\nu^\eps_{h^\eps},h^\eps)$$ 
for the associated area measure, boundary length measure and field. 
We denote by  $$\LQG=(\mu_h,\nu_h,h)$$
its critical counterpart, which we also sample independently of $\loops$.
We equip the fields with the $H^{-1}(\D)$ topology, and the measures with the weak topology for measures on $\D$ and $\partial \D$ respectively.

Then by Remark \ref{rmk:disk-conv}, Proposition \ref{prop:cle-conv}, and the independence of $\loops^\eps$ and $\LQG^\eps$ (resp.\ $\loops$ and $\LQG$), we have that:
\begin{proposition}
	\label{prop:p1}
	$(\loops^\eps,\LQG^\eps)\Rightarrow (\loops,\LQG)$
	as $\eps\to 0$.
\end{proposition}

Additionally, for every $\eps\in (0,2-\sqrt{2})$ by the mating of trees theorem, Theorem \ref{thm:MOT}, $(\loops^\eps,\LQG^\eps)$ determines a collection of Brownian observables
$$\BM^\eps=(X^\eps,(\cI^{\eps,\frk t})_{\frk t}, (\ell^{\eps,\frk t})_{\frk t}, (T^{\eps,\frk t})_{\frk t}, 
(\wh X^{\eps,\frk t})_{\frk t} )$$
as explained in Section \ref{sec:Bfs}. Recall that $X^\eps$ is $\sqrt{\pi}$ times an uncorrelated Brownian excursion in the cone $\{z\in \C: \arg(z)\in [-\pi/2+\tan^{-1}(a_\eps),\pi/2-\tan^{-1}(a_\eps))\}$, starting from $(a_\eps,1)$ and ending at the origin, where $a_\eps=\sqrt{(1+\cos(\pi\gamma^2/4))/(1-\cos(\pi\gamma^2/4)))}=\pi\eps/2+o(\eps)$. The indexing of the above processes is over $\frk{t}\in \mathbb{R}_+\cap \mathbb{Q}$. If we also write $$\BM=(X,(\cI^{\frk t})_{\frk t}, (\ell^{\frk t})_{\frk t}, (T^{\frk t})_{\frk t}, 
(\wh X^{\frk t})_{\frk t} ),$$
for a tuple with law as described in Section \ref{sec:Bfs}, then by Lemma \ref{lem:BM-conv} we  have that: 
\begin{proposition}
	\label{prop:p2}
	$\BM^\eps\Rightarrow \BM$
	as $\eps\to 0$.
\end{proposition} Here, $\cI^{\eps,\frk t},\cI^{\frk t}$ are equipped with the Hausdorff topology, and the stochastic processes in the definition of $\BM^\eps, \BM$ are equipped with the Skorokhod topology.\\

We now wish to describe the \emph{joint} limit of 	$(\loops^\eps, \LQG^\eps, \BM^\eps)$ as $\eps\to 0$.  For this, we first need to introduce a little notation. 

For 
{$z, w\in\cQ$, $z\ne w$,}
we can consider the first time $\sigma_{z,w}^\eps$ (defined by $\loops^\eps$)  at which $z$ and $w$ are in different complementary components of $\D\setminus \eta_z^\eps$. We let $U^\eps=U^\eps(z,w)\subset\D$  denote the component  which is visited first by the space-filling SLE$_\kp$ $\eta^\eps$. We say that $U^\eps=U^\eps(z,w)$ is the \emph{monocolored} component when $z$ and $w$ are separated. 
Let us define $$\mathfrak{U}^\eps_z:=\{U\subset \D: U=U^\eps(z,w) \text{ for some } z\ne w \text{ with } \cO^\eps_{z,w}=0\}$$
to be the set of monocolored components separated from $z$ by $\eta_z^\eps$. Note that these are naturally ordered, according to the order that they are visited by $\eta^\eps$. In fact, we may also associate orientations to the elements of $\frk U_z^\eps$: we say that $U\in \frk U_z^\eps$ is ordered clockwise (resp.\ counterclockwise) if the boundary of $U$ is visited by $\eta_z^\eps$ in a clockwise (resp.\ counterclockwise) order, and in this case we write $\mathrm{sgn}(U)=-1$ (resp.\ $+1$).

\begin{remark}\label{rmk:MOT}
	For $\eps\in (0,2-\sqrt{2})$, by Theorem \ref{thm:MOT} and the definitions above, we have that: \begin{compactitem}
		\item the duration of $Z^\eps$ is equal to $\mu_{h^\eps}^\eps(\D)$, hence $X^\eps=0$ for all $t\ge \mu^\eps_{h^\eps}(\D)$ almost surely;
		\item for $z\in \cQ$, the time $t_z^\eps$ at which $\eta^\eps$ visits $z$ is almost surely given by $\mu_{h^\eps}^\eps(\cup_{U\in \frk{U}_{z}^\eps} U)=\sum_{\frk{U}_{z}^\eps} \mu_{h^\eps}^\eps(U)$;
		\item  the ordered $\nu_{h^\eps}^\eps$ boundary lengths of the components of $\frk{U}_z^\eps$ are almost surely equal to the ordered  jumps of $(\wh B^{\eps,t_{z}^\eps})$, and the sign of each jump is equal to the sign of the corresponding element of $\frk{U}_z^\eps$;
		\item the ordered $\mu^\eps_{h^\eps}$ masses of the components of $\frk{U}_z^\eps$ are almost surely equal to the ordered jumps of $T^{\eps,t_z^\eps}$. 
	\end{compactitem}
\end{remark}

We can also define analogous objects associated with the CLE$_4$ exploration: if $z$ and $w$ are separated at time $\sigma_{z,w}$ by the $\CLE_4$ exploration branch towards $z$, and $\cO_{z,w}=1$ we set $U(z,w)=(\lcB_z)_{\sigma_{z,w}}$; if $\cO_{z,w}={0}$ we set $U(z,w)=(\lcB_w)_{\sigma_{w,z}}$. {The set} $\frk{U}_z$ is then defined in exactly the same way. Note that in this case the elements of $\frk{U}_z$ are ordered by declaring that $U$ comes before $U'$ iff $U=U(z,w)$ and $U'=U(z,w')$ for $w\ne w'$ such that $\cO_{w',w}=0$. We now say that $U\in \frk U_z$ is ordered clockwise (resp.\ counterclockwise) if there is an even (resp.\ odd) number of loops which enclose $U$, and write $\mathrm{sgn}(U)=-1$ (resp. $+1$). 

The main ingredient that will allow us to describe the joint limit {of $(\loops^\eps, \LQG^\eps, \BM^\eps)$} is the following:

\begin{proposition}
	\label{prop:p3}
	Given $(\loops^\eps, \LQG^\eps)$, denote by $ z^\eps$ a point sampled from $\mu^\eps_{h^\eps}$ in $\D$ (normalised to be a probability measure) and given $(\loops, \LQG)$, denote by $z$ a point sampled in the same way from $\mu_{h}$. For given $\delta>0$, write $(U_1^{\eps},\dots, U_{N^\eps}^\eps)$ for the ordered components of $\frk{U}_{z^\eps}^{\eps}$ with $\mu^\eps_{h^\eps}$ area $\ge \delta$, and define $(U_1,\dots, U_N)$ similarly for the ordered components of $\frk{U}_{z}$ with $\mu_h$ area $\ge \delta$. Suppose that $w_i^\eps$ for $1\le i \le N_\eps$ (resp.\ $w_i$ for $1\le i \le N$) are sampled from $\mu^\eps|_{U_i^\eps}$ (resp.\ $\mu|_{U_i}$) normalized to be probability measures, and $g_i^\eps:U_i^\eps \to \D$ (resp.\ $g_i:U_i\to \D$) are the conformal maps that send $w_i^\eps$ to $0$ (resp. $w_i$ to $0$) with positive real derivative at $w_i^\eps$ (resp.\ $w_i$). Set $\mathrm{sgn}(U_i^\eps)=w_i^\eps=0$ (resp.\  $\mathrm{sgn}(U_i)=w_i=0$) and $g_i^\eps(h^\eps)$ (resp.\ $g_i(h)$) to be the $0$ function for $i>N^\eps$ (resp.\ $i>N$).  
	Then	$$(\loops^\eps, \LQG^\eps,  z^\eps,(\mathrm{sgn}(U_i^\eps))_{i\ge 1}, (w_i^\eps)_{i\ge 1},
	(g_i^\eps(h^\eps))_{i\ge 1})
	\Rightarrow (\loops,\LQG,z,(\mathrm{sgn}(U_i))_{i\ge 1}, (w_i)_{i\ge 1},
	(g_i(h))_{i\ge 1})$$
	as $\eps\to 0$.\footnote{with respect to the Euclidean topology in the third coordinate, and the topology in the final coordinates defined such that $((s_i^n)_{i\ge 1}, (w_i^n)_{i\ge 1}, (h^n_i)_{i\ge 1})\to ((s_i)_{i\ge 1}, (w_i)_{i\ge 1}, (h_i)_{i\ge 1})$ as $n\to \infty$ iff the number of non-zero components on the left hand side is equal to the number $N_n$ of non-zero components on the right hand side for all $n$ large enough, and the first $N$ components converge in the product discrete $\times$ Euclidean $\times$ $H^{-1}(\D)$ topology.} The fields $g_i^\eps(h^\eps)$ and $g(h)$ above are defined using the change of coordinates formula \eqref{eq:coc}. 
\end{proposition}

In other words, the ordered and signed sequence of monocolored quantum surfaces separated from $z^{\eps_n}$ converges almost surely, as a sequence of quantum surfaces {(see above \eqref{eq:coc})} to the ordered sequence of monocolored quantum surfaces separated from $z$ as $n\to \infty$.

From this, we can deduce our main theorem.
\begin{theorem}\label{thm_main} 
	$(\loops^\eps, \LQG^\eps, \BM^\eps)$ converges jointly in law to {a tuple} $(\loops, \LQG, \BM)$ as $\eps\downarrow 0$.
	In the limiting tuple, $\loops, \LQG, \BM$ have {marginal} laws as above, $\loops$ and $\LQG$ are independent, and $(\loops, \LQG)$ determines $\BM$.
	
	Furthermore, we have the following explicit description of the correspondence between $(\loops,\LQG)$ and $\BM$ in the limit. Suppose that $z\in\D$ is sampled from the critical Liouville measure $\mu$  normalized to be a probability measure. Then
	\begin{itemize} \item $X_t=0$ for all $ t\ge \mu(\D)$ almost surely and the conditional law of \begin{equation}\label{eq:tz}t_z:= \mu_h\left( \cup_{U\in\frk{U}_z} U \right)\end{equation}
		given $(\loops, \LQG, \BM)$ is uniform on $(0,\mu(\D))$,
		\item $X_{t_z}=(A_{t_z},B_{t_z})$ satisfies {the following for a deterministic constant $c>0$}:
		\begin{equation}\label{eq:AB} A_{t_z}={c}\liminf_{\delta\to 0} \delta N_\delta 
			\text{ and } B_{t_z}=1+\sum_{U\in \frk U_z} \mathrm{sgn}(U) \nu_h(U)\end{equation}
		almost surely,	where for $\delta>0$, $N_\delta$ is the number of domains $U\in \frk{U}_z$ such that $\nu_h(\partial U)\in (\delta/2,\delta)$,
		\item the ordered collection $(\mu_h(U),\mathrm{sgn}(U)\nu_h(\partial U))_{U\in \frk{U}_z}$ is almost surely equal to the ordered collection of jumps of $(T^{t_z},\wh B^{t_z})$ (where $(T^{t_z},\wh B^{t_z})$ are defined from $\BM$ as in Section \ref{sec:Bfs}).
	\end{itemize}
\end{theorem}

Notice that \begin{equation}\label{eq:Al}A_{t_z}=\wh A_{\ell^{t_z}_{t_z}}=\ell^{t_z}_{t_z}\end{equation} is the limit as $\eps\to 0$ of the total length of the $\SLE_\kp(\kp-6)$ branch towards $z$ in the quantum natural parameterization. We can therefore view $A_{t_z}$ as a limiting ``quantum natural distance'' of $z$ from the boundary of the disk. In a similar vein, we record in Table \ref{table1} some of the correspondences between the $\CLE_4$ decorated critical LQG disk with order variables $(\loops, \LQG)$ and the Brownian excursion $\BM$, where $z,w$ are points sampled from the critical LQG measure $\mu_h$ in the bulk. 
\begin{table*}[h]\centering
	\begin{tabular}{@{}l l@{}} \toprule
		$\BM$ & $(\loops, \LQG)$  \\ \midrule
		duration of $X$ & $\mu_h(\D)$  \\ 
		$\{t_w<t_z\}$ & $\{\cO_{w,z}=1\}=$``$w$ ordered before $z$'' \\
		$t_z$ & $\mu_h(\overline{\{w\in \cQ: O_{w,z}=1\}})=$``quantum area of points ordered before $z$''  \\
		$A_{t_z}$ & quantum natural distance of $z$ from $\partial \D$  \\
		jumps of $\wh B^{t_z}$ & LQG boundary lengths of ``components ordered before $z$''\\
		sign of jump & parity of $
		\#\ \{\CLE_4 \text{ loops surrounding component}\}$ \\
		jumps of $T^{t_z}$ & LQG areas of ``components ordered before $z$''\\
		CRT encoded by $A$ & $\CLE_4$ exploration branches  parameterized by quantum natural distance \\ \bottomrule
	\end{tabular}
	\label{table1}
\end{table*}

\begin{proofof}{Theorem \ref{thm_main} given Proposition \ref{prop:p3}}
	Since we know the marginal convergence of each component of $(\loops^\eps, \LQG^\eps, \BM^\eps)$, we know that the triple is tight in $\eps$. Thus our task is to characterize any subsequential limit $(\loops, \LQG, \BM)$ of $(\loops^\eps, \LQG^\eps,\BM^\eps)$. Note that Proposition \ref{prop:p1} already tells us that $(\loops, \LQG)$ are independent, and Proposition \ref{prop:p2} tells us that the marginal law of $\BM$ is that of a Brownian half plane excursion plus associated observables. 
	
	To characterize the law of $(\loops, \LQG, \BM)$ we will prove that if $z\in \D$ is sampled according to $\mu_h$ in $\D$, conditionally independently of the rest of $(\loops, \LQG, \BM)$ then: 
	\begin{enumerate}[(i)]
		\item the duration of $X$ is equal to $\mu_h(\D)$ almost surely;
		\item
		$t_z$ defined by \eqref{eq:tz} is conditionally uniform on $(0,\mu_h(\D))$ given $(\loops, \LQG,\BM)$;
		\item the ordered collection $(\mu_h(U),\mathrm{sgn}(U)\nu_h(\partial U))_{U\in \frk{U}_z}$ is almost surely equal to the ordered collection of jumps of $(T^{t_z},\wh B^{t_z})$ (defined from $\BM$ as in Section \ref{sec:Bfs}); and
		\item $A_{t_z}, B_{t_z}$ satisfy \eqref{eq:AB} almost surely.
	\end{enumerate}
	Let us remark already that the above claim is enough to complete the proof of the theorem. Indeed, suppose that $(\loops, \LQG,\BM)$ is a subsequential limit in law of  $(\loops^\eps, \LQG^\eps, 
	\BM^\eps)$ as $\eps\to 0$ and let $(\loops, \LQG, \BM, \BM')$ be coupled so that $(\loops,\LQG,\BM)$ is equal in law to $(\loops,\LQG, \BM')$, while $\BM,\BM'$ are conditionally independent given $\loops, \LQG$. Further sample $z$ from $\mu_h$ in $\D$, conditionally independently of the rest of $(\loops, \LQG, \BM,\BM')$, so that (i)-(iv) hold for $(\loops, \LQG, \BM, z)$ and for $(\loops, \LQG, \BM\nina{'}, z)$ (with $X,A,B$ replaced by their counterparts $X',A',B'$ for $\BM'$.) Then by (i)-(ii) and since $X(\BM)$, $X(\BM')$ are almost surely continuous, if $\mathbb{P}(\BM\ne \BM')$ were strictly positive then $\mathbb{P}(X(\BM)_{t_z}\ne X(\BM')_{t_z})$ would be strictly positive as well. This would contradict (iii) and (iv), so we conclude that $\BM=\BM'$ almost surely. This means that $\BM$ is determined by $(\loops, \LQG)$, and the explicit description in the statement of the {theorem} also follows immediately. 
	
	\corr{The same argument implies that the law of any subsequential limit is unique. More concretely, suppose that $\eps_n$, ${\eps}'_n$ are two sequences tending to $0$ as $n\to \infty$, such that $(\loops^{\eps_n},\LQG^{\eps_n},\BM^{\eps_n})\Rightarrow (\loops, \LQG,\BM)$ and $(\loops^{\eps'_n},\LQG^{\eps'_n},\BM^{\eps'_n})\Rightarrow (\loops',\LQG',\BM')$ as $n\to \infty$. Then we can also take a joint subsequential limit of $(\loops^{\eps_n},\LQG^{\eps_n},\BM^{\eps_n},\loops^{\eps'_n},\LQG^{\eps'_n},\BM^{\eps'_n})$; call it $(\loops,\LQG,\BM,\loops',\LQG',\BM')$ where necessarily $\loops=\loops'$ and $\LQG=\LQG'$, since we already know the convergence $(\loops^\eps,\LQG^\eps)\Rightarrow (\loops,\LQG)$. Repeating the argument of the previous paragraph gives that $\BM=\BM'$ almost surely. In particular, the marginal law of $(\loops',\LQG',\BM')$ is the same as that of $(\loops,\LQG,\BM)$.}
	
	So we are left to justify the above claim.	To this end, let 
	\begin{equation}\label{E:triple} (\loops, \LQG, \BM)
	\end{equation} be a subsequential limit, along some subsequence of $\eps$. By Proposition \ref{prop:p3} and passing to a further subsequence  if necessary we may extend this to the convergence \begin{gather}(\loops^{\epn},\LQG^{\epn}, z^{\epn},  \BM^{\epn},\big((\mathrm{sgn}(U_i^{\corr{\epn,\delta}}))_{i\ge 1},
		(g_i^{\corr{\epn,\delta}}(h^\epn))_{i\ge 1}\big)_{\delta\in \BB Q\cap (0,1)} )\nonumber \\ \Rightarrow \nonumber \\(\loops, \LQG,z,\BM, \big((\mathrm{sgn}(U_i^{\corr{\delta}}))_{i\ge 1}, (g_i^{\corr{\delta}}(h))_{i\ge 1}\big)_{\delta \in \BB Q \cap (0,1)})\label{eq:ssmain}\end{gather}
	along some $\epn\downarrow 0$, where for every $\delta {\in \BB Q \cap (0,1)}$ the joint law of  \corr{$$(\loops^{\epn},\LQG^{\epn}, z^{\epn},  \BM^{\epn},\big((\mathrm{sgn}(U_i^{\corr{\epn,\delta}}))_{i\ge 1},
			(g_i^{\corr{\epn,\delta}}(h^\epn))_{i\ge 1}\big)_{\delta\in \BB Q\cap (0,1)} )) \text{ and }(\loops, \LQG, z, (\mathrm{sgn}(U_i^{\corr{\delta}})_{i\ge1}, g_i^{\corr{\delta}}(h)_{i\ge 1}))$$} 
	are as in Proposition \ref{prop:p3} \corr{(now with the depedence on $\delta$ indicated for clarity)} and the joint law of $(\loops, \LQG, \BM)$ is the one assumed in \eqref{E:triple}. Note that the conditional law of $z$ given $(\loops,\LQG,\BM)$ is that of a sample from $\mu_h$, since the same is true at every approximate level and since $\mu^{\epn}_{h^\epn}$ converges as part of $\LQG^{\epn}$. 
	
	We next argue that the convergence \eqref{eq:ssmain} necessarily implies the joint convergence 
	\begin{gather}(\loops^{\epn},\LQG^{\epn}, z^{\epn},
		\BM^{\epn},\big((\mathrm{sgn}(U_i^{\corr{\epn,\delta}}))_{i\ge 1},(g_i^{\corr{\epn,\delta}}(h^\epn))_{i\ge 1}, 
		(\mu^\epn_{h^\epn}(U_i^{\corr{\epn,\delta}}))_{i\ge 1}, (\nu_{h^\epn}^\epn(\partial U_i^{\corr{\epn,\delta}}))_{i\ge 1 }\big)_{\delta\in \BB Q\cap (0,1)}) \nonumber \\
		\Rightarrow \nonumber \\ (\loops, \LQG,z,
		\BM, \big((\mathrm{sgn}(U_i^{\corr{\delta}}))_{i\ge 1}, (g_i^{\corr{\delta}}(h))_{i\ge 1}, (\mu_h(U_i^{\corr{\delta}}))_{i\ge 1 }, (\nu_{h}(\partial U^{\corr{\delta}}_i))_{i\ge 1 }\big)_{\delta\in \BB Q\cap (0,1)})\label{eq:ssmain2}\end{gather}
	as $n\to \infty$, where the initial components are exactly as in \eqref{eq:ssmain}. Indeed, we know that the tuple on the left is tight in $n$, {because the first six terms are tight by above and both $(\mu^\epn_{h^\epn}(U_i^{\corr{\epn,\delta}}))_{i\ge 1}$ and $(\nu_{h^\epn}^\epn(\partial U_i^{\corr{\epn,\delta}}))_{i\ge 1 }$  are sequences with only a tight number of non-zero terms, and with all non-zero terms bounded by convergent quantities in $(\LQG^\epn,\BM^\epn)$.}
	On the other hand, for any fixed $\delta$, $i$ and $n$,  $$\mu_{h^\epn}^\epn(U_i^{\corr{\epn,\delta}})=\mu_{g_i^{\corr{\epn,\delta}}(h^\epn)}^\epn(\D) \text{ and } \nu_{h^\epn}^\epn(\partial U_i^{\corr{\epn,\delta}})=\nu_{g_i^{\corr{\epn,\delta}}(h^\epn)}^\epn(\partial \D),$$ so by Theorem \ref{thm:MOT},  $(g_i^{\corr{\epn,\delta}}(h^\epn), \mu_{h^\epn}^\epn(U_i^{\corr{\epn,\delta}}), \nu_{h^\epn}^\epn(\partial U_i^{\corr{\epn,\delta}}))$  is a sequence of  $\gamma(\epn)$-quantum disks together with their quantum boundary lengths and areas.  We can therefore apply  Remark \ref{rmk:changing_lengths} to deduce that any subsequential limit in law $(g_i(h),\mu^*,\nu^*)$ of \\ $(g_i^{\corr{\epn,\delta}}(h^\epn), \mu_{h^\epn}^\epn(U_i^{\corr{\epn,\delta}}), \nu_{h^\epn}^\epn(\partial U_i^{\corr{\epn,\delta}}))$ must be equal to  $$(g_i^{\corr{\delta}}(h),\mu_{g_i^{\corr{\delta}}(h)}(\D),\nu_{g_i^{\corr{\delta}}(h)}(\partial \D))=(g_i^{\corr{\delta}}(h),\mu_h(U_i^{\corr{\delta}}),\nu_h(\partial U_i^{\corr{\delta}})).$$ This concludes the proof of \eqref{eq:ssmain2}.

	So to summarize, if we have any subsequential limit $(\loops, \LQG, \BM)$ of $(\loops^\eps,\LQG^\eps,\BM^\eps)$ we can couple it with $z$ (whose conditional law given $(\loops,\LQG,\BM)$ is that of a sample from $\mu_h$) and with $(U_i,g_i)_{i\ge 1}$ for every positive $\delta\in \BB Q$, such that the joint convergence \eqref{eq:ssmain2} holds along some subsequence $\epn\downarrow 0$. By Skorokhod embedding we may assume that this convergence is almost sure, and so just need to prove that (i)-(iv) hold for the limit. This essentially follows from Remark \ref{rmk:MOT} and the convergence of the final coordinates in \eqref{eq:ssmain2}; we give the details for each point below.
	\begin{enumerate}[(i)]
		\item This holds since $X^{\epn}=0$ for all $t\ge \mu^{\epn}(\D)$ almost surely for every $n$, and $(\mu^{\epn}_{h^\epn}(\D),X^{\epn}) \to (\mu_h(\D),X)$ almost surely.
		\item The convergence of the areas in \eqref{eq:ssmain2} implies that $$t_{z^\epn}^\epn=\sum_{\frk{U}_{z^\epn}^\epn} \mu^\epn_{h^\epn}(U)$$ converges almost surely to $t_z$ defined in \eqref{eq:tz} along the subsequence $\epn\downarrow 0$. On the other hand, $t_z^\epn$ is conditionally uniform on $(0,\mu^\epn_{h^\epn}(\D))$ given $(\loops^\epn, \LQG^\epn, \BM^\epn)$ for every $n$. 
		\item The ordered collection of jumps of $(T^{{\epn},t^{\epn}_{z^{\epn}}},\wh B^{{\epn},t_{z^{\epn}}^{\epn}})$ converge almost surely to the ordered collection of jumps of $(T^{t_z},\wh B^{t_z})$ on the one hand, by definition of the convergence $(\BM^{\epn},z^{\epn})\to (\BM,z)$ (and by considering a sequence $z^n\in \cQ$ converging to $z$). On the other hand, they are equal to the ordered collection $(\mu^{\epn}_{h^\epn}(U),\mathrm{sgn}(U)\nu_{h^{\epn}}^{\epn}(\partial U))_{U\in \frk{U}_z^{\epn}}$ for every $n$. Since this latter collection converges almost surely to the ordered collection $(\mu_h(U),\mathrm{sgn}(U)\nu_h(\partial U))_{U\in \frk{U}_z}$, we obtain (iii). 
		\item This follows from (iii) and the fact that the marginal law of $X=(A,B)$ is that of a Brownian excursion in the right half plane. Specifically, the first coordinate of $X$ at a given time $t$ can a.s.\ be recovered from the jumps of its inverse local time at backwards running infima with respect to time $t$, see \eqref{eq:Al}, and the second coordinate can also be recovered from the collection of its signed jumps when reparameterized by this inverse local time. When $t=t_z$, the values are recovered exactly using the formula \eqref{eq:AB} after using (iii) to translate between $(\mu_h(U),\mathrm{sgn}(U)\nu_h(\partial U))_{U\in \frk{U}_z}$ and $(T^{t_z},\wh B^{t_z})$. 
	\end{enumerate}
\end{proofof}

\subsection{Proof of Proposition \ref{prop:p3}}

\corr{In this subsection, $\delta$ is fixed, so we omit it from the notation (just as in the statement of Proposition \ref{prop:p3}).} Since the convergence of $\mu_{h^\eps}^\eps$ to $\mu_h$ is included in the convergence of $(\loops^\eps, \LQG^\eps)$ to $(\loops, \LQG)$ it is clear (for example by working on a probability space where the convergence holds almost surely) that $(\loops^\eps, \LQG^\eps, z^\eps)\Rightarrow (\loops,\LQG, z)$ as $\eps\to 0$. From here, the proof proceeds via the following steps. 

\begin{enumerate}[(1)]
	\item The tuples on the left hand side in Proposition \ref{prop:p3} are tight in $\eps$,{ so we may take a subsequential limit $(\loops, \LQG, z, (s_i)_{i\ge 1}, (w_i)_{i\ge 1}, (h_i)_{i\ge 1})$ (that we will work with for the remainder of the proof)}.
	\item $w_i\in \D\setminus \Gamma$ \corr{(i.e. $w_i$ is not on any nested CLE$_4$ loop)} for all $i$ a.s.
	\item If $\wt g_i:U(z,w_i) {\to} \D$ are conformal with $\wt g_i(w_i)=0$ and $\wt g_i'(w_i)>0$, then $h_i=\wt g_i(h)$ for each $i$ a.s.\footnote{Once we have point (5), it follows that these are equal to the $(g_i)_{i=1}^N$.}
	\item Given $(\loops,\LQG,z)$, the $w_i$ are conditionally independent and distributed according to $\mu_h$ in each $U(z,w_i)$.
	\item $\{U\in \frk{U}_z: \mu_h(U)\ge \delta\}=\{U(z,w_i)\}_{i\ge 1}$ a.s., where the set on the left is ordered as usual.
	\item $s_i=\mathrm{sgn}(U(z,w_i))$ for each $i$ a.s.
\end{enumerate}
These clearly suffice for the proposition. \\

\begin{proofof}{(1)}
	Tightness of the first five components follows from the fact that $(\loops^\eps, \LQG^\eps, z^\eps)\Rightarrow (\loops,\LQG, z)$ as $\eps\to 0$, plus the tightness of the quantum boundary lengths in $\frk{U}_z^\eps$ (recall that these converge when $\BM^\eps$ converges). To see the tightness of $(g_i^\eps(h^\eps))_{i\ge 1}$ we note that there are at most $\mu^\eps_{h^\eps}(D)/\delta$ non-zero terms, where $\mu^\eps_{h^\eps}(\D)$ is tight in $\eps$. Moreover, each non-zero $g_i^\eps(h^\eps)$ has the law of $\wt{h}^\eps\circ \theta^\eps+a^\eps$, where $\wt{h}^\eps$ is as in Lemma \ref{cor:disc_mapped},  $\theta^\eps$ are random rotations (which automatically form a tight sequence in $\eps$) and $a^\eps$ are some tight sequence of real numbers. This implies the result by Lemma \ref{cor:disc_mapped}.
\end{proofof}\\

\begin{proofof}{(2)} Suppose that $(y_j^\eps)_{j\ge 1}$ are sampled conditionally independently according to $\mu^\eps_{h^\eps}$ in $\D$, normalized to be a probability measure. Then $(\loops^\eps, \LQG^\eps, (y_j^\eps)_{j\ge 1})\Rightarrow (\loops, \LQG, (y_j)_{j\ge 1})$ where the $(y_j)_{j\ge 1}$ are sampled conditionally independently from $\mu_h$ and almost surely all lie in $\D\setminus \Gamma$. On the other hand, since $\loops^\eps$ and $\LQG^\eps$ are independent, one can sample $(w_i^\eps)_{i\ge 1}$ by taking $(\loops^\eps,\LQG^\eps,(y_j^\eps)_{j\ge 1})$ and then setting $w_i^\eps=y_j^\eps$ for each $i$, with $j=\min\{k:y_k\in U_i^\eps\}$.
\end{proofof}\\

\begin{proofof}{(3)}
	By Skorokhod's theorem, we may work on a probability space where we have the almost sure convergence 
	\begin{equation}\label{eq:sko}(\loops^\epn, \LQG^\epn,  z^\epn,(\mathrm{sgn}(U_i^\epn))_{i}, (w_i^\epn)_{i},
		(g_i^\epn(h^\epn))_{i})\to (\loops, \LQG, z, (s_i)_{i}, (w_i)_{i}, (h_i)_{i})\end{equation} 
	along a sequence $\epn\downarrow 0$. It is then natural to expect, since the $w_i^\epn$ converge a.s.\ to the $w_i$ and $\loops^\epn$ converges a.s.\ to $\loops$, that the maps $g_i^\epn$ will converge to $\wt g_i$ described in (3). Since $h^\epn$ also converges a.s.\ to $h$ (as part of the convergence $\LQG^\epn\to \LQG$) it therefore follows $h_i$ will a.s.\ be equal to \corr{$\wt g_i(h)$} for each $i$. 
	This is the essence of the proof. However, one needs to take a little care with the statement concerning the convergence $g_i^{\epn} \to \wt g_i$, since the domains $U_i^\epn$ and $U(z,w_i)$ are defined in terms of points that are \emph{not} necessarily in $\cQ$, while the convergence of $\loops^\eps\to \loops$ is stated in terms pairs of points in $\cQ$.

	To carry out the careful argument, let us fix $i\ge 1$. Since $w_i\in \D\setminus \Gamma$ a.s.\ by (2), there exists $r>0$ and $y\in \cQ$ such that $B(y,r)\subset B(w_i,2r)\subset U(z,w_i)=(\lcB_{w_i})_{\sigma_{w_i,z}}$. By taking $r$ smaller if necessary, we can also find $x\in \cQ$ with $B(x,r)\subset  B(z,2r)\subset (\lcB_z)_{\sigma_{z,w_i}}$. Note that $\cO_{z,w_i}=\cO_{x,y}=0$ by definition. Due to the almost sure convergence $z^\epn\to z$, $w^\epn_i \to w_i$, and $\loops^\epn\to \loops$ it then follows that  $U^\epn(z^\epn,w_i^\epn)=U^\epn(x, y)=(\lcB_y^\epn)_{\sigma^\epn_{y,x}}$, and $\cO^\epn_{x,y}=\cO^\epn_{z^\epn,w_i^\epn}=0$ for all $n$ large enough. Moreover, we know that the maps $f^\epn:\D \to U^\epn(z^\epn, w_i^\epn)=(\lcB^\epn_{y})_{\sigma_{y,x}}$ with $f^\epn(0)=y$, $(f^\epn)'(0)>0$ converge on compacts of $\D$ to $f:\D\to U(x,y)=(\lcB_y)_{\sigma_{y,x}}$ sending $0$ to $y$ and with $f'(0)>0$. 
	
	On the other hand, $(\wt g_i)^{-1}=f\circ \phi$ where $\phi:\D\to \D$ sends $0\mapsto f^{-1}(w_i)$ and has $\phi'(0)>0$, and $(g_i^\epn)^{-1}=f^\epn\circ \phi^\epn$ for each $\epn$, where $\phi^\epn:\D\to \D$ has $\phi^\epn(0)=(f^\epn)^{-1}(w_i^\epn)$ and $(\phi^\epn)'(0)>0$. Since $w_i^\epn\to w_i$ almost surely, and the $w_i^\epn$ are uniformly close to $y$ and bounded away from the boundary of $U^\epn(x,y)$, this implies that $(g_i^{\epn})^{-1}$ converges to $\wt g_i^{-1}$ uniformly on compacts of $\D$. In turn, this implies that $h_i$ restricted to any compact of $\D$ is equal to $\wt g_i(h)$, which verifies that $h_i=g_i(h)$ a.s.
\end{proofof}\\

\begin{proofof}{(4)} For this it suffices to prove  that for each $i$, $$(\loops^\epn, \LQG^\epn,  z^\epn, w_i^\epn, g_i^\epn(h^\epn),\mu^\epn_{g_i^\epn(h^\epn)})\Rightarrow (\loops, \LQG, z,  w_i, h_i,\mu_{h_i})$$ as $n\to \infty$, where the convergence of the final components is in the sense of weak convergence for measures on $\D$. Note that if we work on a space where all but the last components converge a.s., as in (3), then the proof of (3) shows that $h_i=\wt g_i(h)$ and that $(g_i^\epn)^{-1}\to (\wt g_i)^{-1}$ a.s.\ when restricted to compact subsets of $\D$. This implies the a.s.\ convergence of  the measures \corr{$\mu^\epn_{g_i^\epn(h^\epn)}$ to $\mu_{h_i}$} when restricted to compact subsets of $\D$. On the other hand, $\mu_{g_i^\epn(h^\epn)}(\D)$ is a tight sequence in $n$, and by Remark \ref{rmk:changing_lengths}, any subsequential limit $(\loops, \LQG, z, w_i, h_i, m)$ of $(\loops^\epn, \LQG^\epn,  z^\epn, w_i^\epn, g_i^\epn(h^\epn),\mu^\epn_{g_i^\epn(h^\epn)}(\D))$ has $m=\mu_{h_i}(\D)$ a.s. Combining these observations yields the result.
	
\end{proofof}\\

\begin{proofof}{(5)}
	As in (3) we assume that we are working on a probability space where we have almost sure convergence along a sequence $\epn\downarrow 0$, so we need to show that the limiting domains $U(z,w_i)$ are precisely the elements of $\frk U_z$ that have $\mu_h$ area greater than or equal to $\delta$. The same argument as for (4) gives that each $U(z,w_i)$ is a component of $\frk{U}_z$ with $\mu_h$ area greater than or equal to $\delta$. So it remains to show that they are the only such elements of $\frk U_z$. 
	
	For this, suppose that $U\in \frk{U}_z$ has $\mu_h(U)\ge \delta$. Then $\mu_h(U)=\delta+r$ for some $r>0$ with probability $1$. Choosing $w\in \cQ$, $a>0$ such that $U=U(z,w)\supset B(w,a)$ it is easy to see that $U(z,w)$ is the a.s.\ \cart limit seen from $w$ of $U^\epn(z^\epn, w)$ as $\epn\to 0$. Using the convergence of $\mu^\epn_{h^\epn}$ to $\mu_h$ and Corollary \ref{cor:cart_inclusion}, we therefore see that $\lim_n \mu_{h^\epn}^\epn(U^\epn(z^\epn,w))\ge \mu_h(U(z,w))=\delta+r$ and so $U^\epn(z^\epn, w)=U_i^\epn=U^\epn(z^\epn,w_i^\epn)$ for some $i$ and all $n$ large enough. From here we may argue as in the proof of (3) to deduce that the \cart limit of $U^\epn(z^\epn,w_i^\epn)$ is equal to $U(z,w_i)$. Thus, since $U=U(z,w)$ is the \cart limit of $U^\epn(z^\epn,w)$ which is equal to $U^\epn(z^\eps,w_i^\epn)$ for all $n$ large enough, we conclude that $U=U(z,w_i)$.
	
	The fact that the orders of the collections in (3) coincide follows from the convergence of the order variables as part of $\loops^\eps\to \loops$ (and the argument we have now used several times that allows one to transfer from $z^\eps, w_i^\eps$ to points in $\cQ$: we omit the details).
\end{proofof}\\

\begin{proofof}{(6)} Let us work under almost sure convergence as in the proof of (3), fix $i\ge 1$ and define $x,y,r$ as in the proof of (3). By Proposition \ref{prop:convbranchingsleorder}, we know that $\sigma^\epn_{y,x}\to \sigma_{y,x}$ almost surely as $n\to \infty$, and that $\mathrm{sgn}(U_i^\epn)$ is determined by the 
	{number of loops nested around $y$ which $\lcB^\epn_y$ discovers} 
	before \emph{or at} time $\sigma^\epn_{y,x}$ (see the definition of CLE loops from the space-filling/branching $\SLE_\kp$ in Section \ref{sec:sletocle}). If $\sigma_{y,x}$ occurs between two such times for $\lcB_y$, it is clear from the a.s.\ convergence of $\sigma^\epn_{y,x}$ and $\lcB_{y}^\epn$ that the number of loop closure times for $\lcB^{\epn}_y$ occurring before or at $\sigma^\epn_{y,x}$ converges to the number of loop closure times for $\lcB_{y,x}$ occurring before or at time $\sigma_{y,x}$. If $\sigma_{y,x}$ is a loop closure time for $\lcB_y$, the result follows from Lemma \ref{lem:septime_endofloop}.
\end{proofof}

\subsection{Discussion and outlook}

The results obtained above open the road to several very natural questions related to the critical mating of trees picture. We will describe some of those below. Roughly, they can be stated as follows: 

\begin{enumerate}
	\item Can one obtain a version of critical mating of trees where there is bi-measurability between the decorated LQG surface and the pair of Brownian motions (with possibly additional information included)?
	\item There is an interesting relation to growth-fragmentation processes studied in \cite{ADS}. Can one combine these two point of views in a fruitful way?
	\item The Brownian motion $A$ encodes a distance of each point to the boundary, and in particular between any CLE$_4$ loop and the boundary. What is its relation to the CLE$_4$ metric introduced in \cite{SWW19}?
	\item Can one prove convergence of observables in critical FK-decorated random planar maps towards	the observables in the critical mating-of-trees picture?
\end{enumerate}

Let us finally mention that there are also other interesting questions in the realm of critical LQG, e.g.\ the behaviour of height-functions on top of critical planar maps, which are certainly worth exploring too.

\subsubsection{Measurability}\label{sec:meas}

In the subcritical mating of trees, i.e., when $\kappa' > 4$, $\gamma < 2$ and we consider the coupling $(\loops, \LQG, \BM)$ described in the introduction or in Section 5 (for simplicity without subscripts), \cite{DMS14} proves that in the infinite-volume setting the pair $(\loops, \LQG)$ determines $\BM$ and vice-versa. In particular, $(\loops, \LQG)$ can be obtained from $\BM$ via a measurable map. This result is extended to the finite volume case of LQG disks in \cite{AG19}. 

By contrast, some of this measurability gets lost when we consider our critical setting.
The easier direction to consider is whether $(\loops, \LQG)$ determine $\BM$. In the subcritical case this comes basically from the construction, and it does not matter what we really mean by $\loops$: the nested CLE$_{\kappa'}$, the space-filling SLE$_{\kappa'}$ and the radial exploration tree of CLE$_{\kappa'}$  are all measurable with respect to one another.
This, however, gets more complicated in the critical case. First, the question of whether the nested CLE$_4$ determines the uniform exploration tree of CLE$_4$ is already not straightforward; this is a theorem of an unpublished work \cite{SWW19}. Moreover, the nested CLE$_4$ no longer determines the space-filling exploration from Section \ref{sec:conv_order}: indeed, we saw that to go from the uniform exploration tree to the ordering on points, some additional order variables are needed. These order variables are, however, the only missing information when going from $(\loops, \LQG)$ to $\BM$: the conclusion of Theorem \ref{thm_main} is that when we include the order variables in $\loops$ (in other words consider the space-filling exploration) then indeed $\BM$ is measurable with respect to $(\loops, \LQG)$.

In the converse direction, things are trickier. 
In the coupling considered in this paper, $\BM$ does not determine the pair $(\loops, \LQG)$; however, we conjecture that $(\loops, \LQG)$ is determined modulo a countable number of ``rotations''.
Informally, one can think of these rotations as follows: a rotation is an operation where we stop the CLE$_4$ exploration at a time when the domain of exploration is split into two domains $D$ and $D'$, we consider the LQG surfaces $(D,h)$ and $(\D\setminus D,h)$, and we conformally weld these two surfaces together differently. The field and loop ensemble $(\wh\loops, \wh\LQG)$ of the new surface will be different than the the pair $(\loops, \LQG)$ of the original surface, but their law is unchanged if we choose the new welding appropriately (e.g.\ if we rotate by a fixed amount of LQG length), and $\BM$ is pathwise unchanged. Therefore performing such a rotation gives us two different pairs $(\loops, \LQG)$ and $(\wh\loops, \wh\LQG)$ with the same law, and which are associated with the same $\BM$. We believe that these rotations are the only missing part needed to obtain measurability in this coupling. In fact, by considering a different CLE$_4$ exploration, where loops are pinned in a predetermined way (e.g.\ where all loops are pinned to some trunk, such as in e.g.\ \cite{lehmkuehler21}), one could imagine obtaining a different coupling of $(\loops, \LQG, \BM)$, where $\BM$ does determine $(\loops, \LQG)$. 

\subsubsection{Growth fragmentation}\label{secGF} We saw below the statement of Theorem \ref{thm_main} how certain observables in the Brownian excursion $\BM$ map to observables (e.g., quantum boundary lengths and areas of discovered $\CLE$ loops) in $(\loops, \LQG)$, when we restrict to a single uniform CLE$_4$ exploration branch. Given the definition of the branching $\CLE_4$ exploration (recall that the explorations towards any two points coincide exactly until they are separated by the discovered loops and then evolve independently) this is one way to define an entire branching process from the Brownian excursion. 

In fact, this embedded branching process was already described completely, and independently, in an earlier work of A\"{i}dekon and Da Silva \cite{ADS}. Namely, given $X=(A,B)$ with law as in Theorem \ref{thm_main}, one can consider for any $a\ge 0$ the countable collection of excursions of $X$ to the right of the vertical line with horizontal component $a$. Associated with each such excursion is a total displacement (the difference between the vertical coordinate of the start and end points) and a sign (depending on which of these coordinates is larger). In \cite{ADS}, the authors prove that if one considers the evolution of these signed displacements as $a$ increases, then one obtains a signed \emph{growth fragmentation} process with completely explicit law. The fact that this process is a growth fragmentation means, roughly speaking, that it can be described by the evolving ``mass'' of a family of cells: the mass of the initial cell evolves according to a positive self-similar Markov process, and every time this mass has a jump, a new cell with exactly this mass is introduced into the system. Each such new cell initiates an independent cell system with the same law. In the setting of signed growth fragmentations, masses may be both positive and negative.

In the coupling $(\loops, \LQG, \BM)$, such a growth fragmentation is therefore naturally embedded in $\BM$. It corresponds to a parameterization of the branching uniform $\CLE_4$ exploration by quantum natural distance from the boundary (i.e., by the value of the $A$ component), and branching occurs whenever components of the disk become disconnected in the exploration. At any given time, the absolute mass of a fragment is equal to the quantum boundary length of the corresponding component, and the sign of the fragment is determined by the number of $\CLE_4$ loops that surround this component.

{Let us also mention that growth fragmentations in the setting of CLE on LQG were also studied in \cite{MSWfrag,MSWfrag2}, and coincide with the growth fragmentations obtained as scaling limits from random planar map explorations in  \cite{BBCK}. Taking $\kappa \rta 4$ in these settings (either $\kappa \uparrow 4$ in \cite{MSWfrag} or $\kappa \downarrow 4$ in \cite{MSWfrag2}) is also very natural and would give other insights about $\kappa =4$ than those obtained in this paper. Lehmkuehler takes this approach in \cite{lehmkuehler21}.} 

\subsubsection{Link with the conformally invariant metric on $\CLE_4$}

Recall the uniform CLE$_4$ exploration from Section 2.1.5, which was introduced by Werner and Wu \cite{WW13}. Werner and Wu interpret the time $t$ at which a loop $\cL$ of the CLE$_4$ $\Gamma$ is added, with the time parameterization \eqref{pppcle}, as the distance of $\cL$ to the boundary $\partial\D$; we refer to it here as the \emph{\corr{CLE$_4$ exploration} distance} of $\cL$ to $\partial\D$. In an unpublished work, Sheffield, Watson, and Wu \cite{SWW19} prove that this distance is the distance as measured by a conformally invariant metric on $\Gamma\cup\{\partial\D \}$. This metric is conjectured to be the limit of the adjacency metric on CLE$_\kp$ loops as $\kp\downarrow 4$. It is also argued in \cite{SWW19} that the uniform exploration of $\Gamma$ is determined by $\Gamma$.

Our process $A$ also provides a way to measure the distance of a CLE$_4$ loop $\cL$ to $\partial\D$, as we previously discussed below \eqref{eq:Al} in the case of a point. Namely, for an arbitrary point $z$ enclosed by $\cL$ define 
\begin{equation}
	t(\cL):= \mu_h\left( \cup_{U\in\frk{U}_z} U \setminus \op{int}(\cL) \right),
\end{equation}
where $\op{int}(\cL)\subset\D$ is the domain enclosed by $\cL$. It is not hard to see that $t(\cL)$ does not depend on the choice of $z$. We call $A_{t(\cL)}$ the \emph{quantum natural distance} of $\cL$ to $\partial\D$. Note that $A_{t(\cL)}$ can also be defined similarly as in \eqref{eq:AB} by counting the number of CLE$_4$ loops of length in $(\delta/2,\delta)$ that are encountered before $\cL$ in the CLE$_4$ exploration and then sending $\delta\rta 0$ while renormalizing appropriately. We remark that, in contrast to the \corr{CLE$_4$ exploration distances}, we do \emph{not} expect that the quantum natural distances to the boundary defined here correspond to a conformally invariant metric on $\Gamma$.

It is natural to conjecture that the \corr{CLE$_4$ exploration} distance and the quantum natural distance are related via a Lamperti type transform
\eqb
A_{t(\cL)}=c_0\int_0^T \nu_h(\partial D_t)\,dt.
\label{eq:lamperti}
\eqe
for some deterministic constant $c_0>0$, where $T$ is the \corr{CLE$_4$ exploration} distance of a loop $\cL$ from $\partial \D$ and for $t\in [0,T)$, $D_t$ is the connected component containing $\cL$ of $\D$ minus the loops at \corr{CLE$_4$ exploration} distance less than $t$ from $\partial \D$.
This is natural 
%since \corr{the distances are conformally invariant/covariant},
\nina{since the distances are invariant under the application of a conformal map (where the field $h$ is modified as in \eqref{eq:coc})},
since the CLE$_4$ exploration is \emph{uniform} for both distances (so if two loops $\cL,\cL'$ have \corr{CLE$_4$ exploration} distance $t,t'$, respectively, to $\partial\D$ then $t<t'$ if and only if $A_{t(\cL)}<A_{t(\cL')}$), and 
since the left and right sides of \eqref{eq:lamperti} transform similarly upon adding a constant $c$ to the field $h$ (namely, both sides are multiplied by $e^{c}$). Proving or disproving \eqref{eq:lamperti} is left as an open problem. 
We remark that several earlier papers \cite{Sha16,Sh16,Ben18,HS19,GM19} have proved uniqueness of lengths or distances in LQG via an axiomatic approach, with axioms of a rather similar flavor to the above, but these proofs do not immediately apply to our setting.

\subsubsection{Discrete models} 
The mating of trees approach to Liouville quantum gravity coupled with CLE is inspired by certain random walk encodings of random planar maps decorated by  statistical physics models. The first such encoding is the hamburger/cheeseburger bijection of Sheffield~\cite{She16HC} for random planar maps decorated by the critical Fortuin–Kasteleyn  random cluster  model (FK-decorated planar map). 

In the FK-decorated planar map each configuration is a planar map with an edge subset, 
whose weight is assigned according to the critical FK model with parameter $q>0$. 
Sheffield encodes this model by five-letter words whose symbol set consists of  hamburger, cheeseburger,  hamburger order, cheeseburger order, and fresh order. 
The fraction $p$ of fresh orders within all orders is given by $\sqrt q=\frac{2p}{1-p}$. As we read the word, a hamburger (resp.\ cheeseburger) will be consumed by either a hamburger (resp.\ cheeseburger) order or a fresh order, in a last-come-first-serve manner.  In this setting, the discrete analog of our Brownian motion $(A,B)$ 
is the net change in the burger count and the burger discrepancy since time zero, which we denote by $(\cC_n,\cD_n)$.

It was proved in~\cite{She16HC} that  $\eps(\cC_{t/\eps^2},\cD_{t/\eps^2})$ converges in law to $(B^1_{t}, B^2_{\alpha t})$, where $B^1,B^2$ are independent standard one-dimensional Brownian motions and $\alpha=\max\{1-2p, 0\}$. When $p\in (0,\frac12)$, the correlation of $(B^1_{t}+B^2_{\alpha t},B^1_{t}- B^2_{\alpha t})$ is the same as for the left and right boundary length processes of space filling $\SLE_\kp$ decorated $\gamma$-LQG (cf.\ Theorem \ref{thm:MOT}) 
where  $q=2+ 2\cos (8\pi/\kappa')$ and $\gamma^2=16/\kappa'$. This is consistent with the conjecture that under these parameter relations, LQG coupled with $\CLE$ (equivalently, space filling $\SLE$) is the scaling limit of the 
FK-decorated planar map for $q\in (0,4)$. Indeed, based on the Brownian motion convergence in~\cite{She16HC}, it was shown in~\cite{GMS19,GS-FK2,GS-FK3}
that geometric quantities such as loop lengths and areas converge as desired. 

When $q=4$ and $p=\frac12$, we have $B^2_{\alpha t}=0$, just as in the $\kappa'\downarrow 4$ limit of LQG coupled with $\CLE$, where the correlation of the left and right boundary length processes tend to 1. 
We believe that  the process  $(\eps\cC_{t/\eps^2}, \mathrm{Var}[\cD_{\eps^{-2}}]^{-1} \cD_{t/\eps^2})$  converges in law to 
$(B^1_{t}, B^2_{t})$; moreover, based on this convergence and results in our paper,  it should be possible to extract the convergence of the loop lengths and areas for FK decorated planar map to  the corresponding observables  in critical LQG coupled with $\CLE_4$.  We leave this as an open question. It would also be very interesting to identify the order of the normalization  $\mathrm{Var}[\cD_{\eps^{-2}}]^{-1}$, which is related to the asymptotic of the partition function of the FK-decorated planar map with $q=4$. 

Another model of decorated random planar maps that is believed to converge (after uniformisation) to CLE decorated LQG is the O($n$) loop model, where the critical case $\kappa=4$ corresponds to $n=2$. It is therefore also interesting to ask whether our Brownian half plane excursion $\BM$ can be obtained as a scaling limit of a suitable boundary length exploration process in this discrete setting. In fact, a very closely related question was considered in \cite{BCM18}, where the authors identify the scaling limit of the perimeter process in peeling explorations of infinite volume critical Boltzmann random planar maps (see \cite{BBG} for the relationship between these maps and the O(2) model). Modulo finite/infinite volume differences, this scaling limit - which is a Cauchy process - corresponds to a single ``branch'' in our Brownian motion (see Section \ref{secGF}).

\bibliographystyle{alpha}
\bibliography{EP_bibliography}
\end{document}